\documentclass[11pt,a4paper]{amsart}

\RequirePackage[OT1]{fontenc}
\RequirePackage{amsthm,amsmath}
\RequirePackage[numbers]{natbib}
\RequirePackage[colorlinks,citecolor=blue,urlcolor=blue]{hyperref}

\usepackage{graphicx}
\usepackage{tikz}
\usepackage{enumitem}

\usepackage{graphicx}
\newtheorem{theorem}{Theorem}
\newtheorem{lemma}[theorem]{Lemma}
\newtheorem*{lemma*}{Lemma}
\newtheorem{proposition}[theorem]{Proposition}
\newtheorem{corollary}[theorem]{Corollary}

\newtheorem{definition}[theorem]{Definition}
\newtheorem{remark}[theorem]{Remark}

\newtheorem*{fact*}{Fact}

\usepackage[utf8]{inputenc}
\usepackage[english]{babel}
\usepackage{latexsym}
\usepackage{amssymb}
\usepackage{amsmath}

\newcommand{\T}{\mathbb{T}}

\newcommand{\Z}{\mathbb{Z}}
\newcommand{\R}{\mathbb{R}}
\newcommand{\C}{\mathbb{C}}
\newcommand{\E}{\mathbb{E}}
\newcommand{\G}{\mathbb{G}}
\newcommand{\Q}{\mathbb{Q}}
\newcommand{\Prob}{\mathbb{P}}

\newcommand{\Tr}{\mathrm{Tr}}

\numberwithin{equation}{section}
\numberwithin{theorem}{section}

\usepackage{fullpage}

\author[M. Nikula]{Miika Nikula}
\address{Department of Computer Science, Aalto University,
P.O.Box 15400, 00076 Aalto, Finland}
\email{miika.nikula@aalto.fi}

\author[E. Saksman]{Eero Saksman}
\address{Department of Mathematical Sciences, Norwegian University of Science and Technology (NTNU), NO-7491 Trondheim, Norway}
\address{University of Helsinki, Department of Mathematics and Statistics,
         P.O. Box 68 , FIN-00014 University of Helsinki, Finland}
\email{eero.saksman@helsinki.fi}

\author[C. Webb]{Christian Webb}
\address{Department of mathematics and systems analysis, Aalto University, P.O.
Box 11000, 00076 Aalto, Finland}
\email{christian.webb@aalto.fi}

\title[Multiplicative chaos and the CUE]{Multiplicative chaos and the characteristic polynomial of the CUE: the $L^1$-phase}

\begin{document}

\begin{abstract} In this note we prove that suitable positive powers of the absolute value of the characteristic polynomial of a Haar distributed random unitary matrix converge in law, as the size of the matrix tends to infinity, to a Gaussian multiplicative chaos measure once correctly normalized. We prove this in the whole $L^1$- or subcritical phase of the chaos measure. 
\end{abstract}

\maketitle

\section{Introduction}
In this note we consider the characteristic polynomial of a Haar distributed random unitary matrix. Our main result is the following theorem.

\begin{theorem}\label{th:main}
Let $U_N$ be a Haar distributed $N\times N$ random unitary matrix. For fixed $\beta\in[\sqrt{2},2)$, as $N\to\infty$, the sequence

\begin{equation*}
\frac{|\det(I-e^{-i\theta}U_N)|^\beta}{\E |\det(I-e^{-i\theta}U_N)|^\beta}\frac{d\theta}{2\pi},
\end{equation*}

\noindent viewed as a sequence of measures on the unit circle $\T$, converges in law with respect to the topology of weak convergence of measures. The limit is a Gaussian multiplicative chaos measure which can be written formally as $\mu_{\beta}(d\theta)=e^{\beta X(\theta)-\frac{\beta^2}{2}\E X(\theta)^2}\frac{d\theta}{2\pi}$, where $X$ is a centered Gaussian field with covariance kernel

\begin{equation}\label{eq:Xcov}
\E X(\theta) X(\theta')=-\frac{1}{2}\log |e^{i\theta}-e^{i\theta'}|.
\end{equation}
\end{theorem}

\noindent We point out that the corresponding result for $\beta\in(-\frac{1}{2},\sqrt{2})$ was proven in \cite[Theorem 2.5]{W}, which is the reason for focusing on $\beta\geq\sqrt{2}$. Moreover, as we briefly discuss below, the limiting object is likely to be zero for $\beta\geq 2$, which is the reason to focus on $\beta<2$.

In the remainder of this introduction, we first discuss some background and motivation for this result and then briefly outline the remainder of the article.

\subsection{Background and motivation}

In the past two decades, the role of a class of stochastic processes, known as log-correlated fields, has gradually emerged in the setting of random matrix theory -- see e.g. \cite{HKO,RidVir,FKS}. More precisely, in basic models of random matrix theory, such as Haar distributed random unitary matrices, the complex Ginibre ensemble, or the GUE, the logarithm of the characteristic polynomial of a large random matrix behaves roughly like a Gaussian process whose covariance has a logarithmic singularity on the diagonal -- this is essentially equivalent to well known results concerning Gaussian fluctuations of global linear statistics. Due to the singularity in the covariance, such an object can't of course be realized as a random function, but can be viewed as a random generalized function. Such objects are known to show up in various other models of modern probability and mathematical physics as well -- e.g. in probabilistic models of combinatorics \cite{IG}, lattice models of statistical mechanics \cite{Kenyon}, the construction of conformally invariant random planar curves such as stochastic Loewner evolution \cite{AJKS,S}, and stochastic growth models \cite{BF}.

While these log-correlated fields are wild objects, random generalized functions instead of random functions, they are known to exhibit non-trivial fractal geometric behavior and the key object in studying such geometric properties of these fields is a random measure which is formally the exponential of this field -- rigorously defined through a suitable renormalization procedure. Such a random measure is known as a multiplicative chaos measure, and their study goes back to the work of Kahane \cite{K}. For a recent review, we refer to \cite{RV} as well as the short and elegant proof of existence and uniqueness in \cite{berestycki}. These multiplicative chaos measures have been used to give meaning to e.g. how the maximum of such a log-correlated field behaves and what the level sets of such a field look like. A basic result of the theory could be something of the following flavor: if $X_\epsilon$ is a nice regularization of a Gaussian field $X(x)$ on say $\R$ whose covariance is of the form $\E X(x)X(y)=-\frac{1}{2}\log |x-y|+g(x,y)$\footnote{Contrary to the typical normalization in the log-correlated literature, where one does not include the $\frac{1}{2}$ in front of the log, we have chosen this normalization to be in line with what occurs in the setting of random matrix theory.}, where $g$ is continuous and bounded, then the family of measures 

$$e^{\beta X_\epsilon(x)-\frac{\beta^2}{2}\E X_\epsilon(x)^2}dx$$

\noindent converges in probability (say with respect to the topology of vague convergence of measures on $\R$) to a non-trivial random measure for $\beta\in(-2,2)$, and the limiting measure lives on the random (fractal) set of points $x$ for which $\lim_{\epsilon\to 0}X_\epsilon(x)/\E X_\epsilon(x)^2=\beta$. Moreover, for any fixed $r>0$, $\lim_{\epsilon\to 0}\max_{x\in B(0,r)}X_\epsilon(x)/\E X_\epsilon(x)^2=2$. 

It is then a natural question to ask to which extent can one prove similar results in the setting of random matrix theory or models where log-correlated fields arise. There have indeed been significant developments in this direction in the past few years. First on a heuristic level, in particular motivated by the conjectured connection between characteristic polynomials of random matrices and the Riemann zeta function, the role of basic results of the geometry of log-correlated fields in the setting of random matrix theory (and the zeta function) was explored in \cite{FHK,FK}. We also refer to the recent numerical study \cite{FGK} on related topics. As already mentioned, in \cite{W}, the analogue of Theorem \ref{th:main} was proven for $-\frac{1}{2}<\beta<\sqrt{2}$. For this range of parameters, one is in the so-called $L^2$-phase of multiplicative chaos, where proving convergence is relatively simple. In \cite{los}, whose approach will serve as the foundation of our proof of Theorem \ref{th:main}, arguments of \cite{berestycki} were generalized to prove that certain smoothed versions of the "log-correlated fields of random matrix theory" give rise to multiplicative chaos measures in the whole $L^1$-phase of multiplicative chaos -- namely with our normalization for $-2<\beta<2$. Then in \cite{BWW} the result of \cite{W} was extended to a large class of unitary invariant models of random hermitian matrices. There have also been significant developments in studying geometric properties of these log-correlated fields coming from random matrices. Here we refer the interested reader to \cite{abb,pz,cmn,LP}, where estimates for the maximum of the field at various precision and in various models have been obtained. In fact, geometric properties as well as connections to multiplicative chaos have also been observed for the Riemann zeta function see e.g. \cite{ABBRS,N,SW}.

We also mention another direction in which there have been recent and exciting developments regarding properties of multiplicative chaos. Recently multiplicative chaos measures have played a significant role in two-dimensional quantum gravity \cite{DupShef,MS,KRV}. In particular, using ideas from \cite{KRV}, a conjecture of Fyodorov and Bouchaud \cite{BouFyo} characterizing the total mass of the limiting measure of Theorem \ref{th:main} was recently proven in \cite[Theorem 1.1]{R}. In view of this result, the following corollary follows immediately from combining Theorem \ref{th:main} with \cite[Theorem 2.5]{W} (as well as a slight reformulation of \cite[Theorem 1.1]{R}) -- we omit further details.

\begin{corollary}
For $0\leq \beta<2$, as $N\to\infty$

$$
\int_0^{2\pi}\frac{|\det(I-e^{-i\theta}U_N)|^\beta}{\E|\det(I-e^{-i\theta}U_N)|^\beta}\frac{d\theta}{2\pi}\stackrel{d}{\to}\frac{Y^{-\frac{\beta^2}{4}}}{\Gamma(1-\frac{\beta^2}{4})},
$$

\noindent where $Y$ is exponentially distributed with mean $1$.
\end{corollary}

\noindent This confirms  a conjecture implicit in \cite{FK} -- see e.g. \cite[(12)]{FK}, which has a very similar flavor as this.

We conclude this discussion about the background of Theorem \ref{th:main} with some remarks about extending this result. First of all, we note that $|\det(I-e^{-i\theta}U_N)|^\beta=e^{\beta \mathrm{Re}\log (\det(I-e^{-i\theta}U_N))}$. We expect that essentially our whole proof would work as it is if we replaced $\mathrm{Re}\log (\det(I-e^{-i\theta}U_N))$ by $\mathrm{Im}\log (\det(I-e^{-i\theta}U_N))$, where for the suitable interpretation of the imaginary part of the logarithm, see e.g. \cite{HKO}. Secondly, we expect that using ideas from \cite{BWW}, such a result can also be proven for other models of random matrices such as the GUE. Finally we mention that for $\beta=2$, one expects from the theory of multiplicative chaos that $\frac{|\det(e^{i\theta}-U_N)|^2}{\E |\det(e^{i\theta}-U_N)|^2}d\theta$ converges to the zero measure as $N\to\infty$. Again from general multiplicative chaos results, see e.g. \cite{DRSV}, one would expect that multiplying by a suitable deterministic diverging factor, or a suitable stochastic quantity, one can construct a non-trivial limit for $\beta=2$. While this does not follow from our results directly, we expect that some of the technical estimates we prove should be of use in proving such a result. In fact, this limiting object for $\beta=2$ should play an important role in verifying conjectures of \cite{FK} concerning the precise behavior of the maximum of the characteristic polynomial of a CUE matrix, so our main result can be viewed as a step in this direction.

\subsection{Outline of the article}

Our proof relies on recent ideas of Lambert, Ostrovsky, and Simm \cite{los} who generalized ideas of Berestycki \cite{berestycki} to develop a general approach of how to prove that something converges to Gaussian multiplicative chaos in the full $L^1$-phase. With this article we illustrate some of the difficulties that can stem from a complicated local structure of the approximation of a Gaussian field. In particular, as shown in \cite{los}, proof of convergence relies on suitably strong estimates for exponential moments of the (approximately) log-correlated field. In our setting, such objects can be expressed as Toeplitz determinants with Fisher-Hartwig singularities. While these determinants have been studied extensively in the literature, particularly efficiently through a connection to Riemann-Hilbert problems -- see e.g. \cite{ck,dik1,dik2}, the ones we encounter have some additional complications, and the Riemann-Hilbert analysis we require is rather complicated and has not been carried out elsewhere, though we rely heavily on results and ideas from \cite{abb,ck,dik2}.

The outline of this paper is the following. In Section \ref{sec:defs}, we'll introduce the relevant definitions and state general estimates we'll need for proving Theorem \ref{th:main}. After that, in Section \ref{sec:mainproof} we'll apply (a slight variant of) the argument of \cite{los} to prove Theorem \ref{th:main} assuming these estimates. The remainder of the paper is devoted to establishing the connection between our main technical estimate (namely Proposition \ref{pr:testimate}) and a suitable Riemann-Hilbert problem, along with the asymptotic analysis of this problem. As stated above, our analysis of this problem can be seen as a combination of ideas from \cite{abb,ck,dik1,dik2}. While this would certainly be routine for experts, we give a fairly detailed presentation of the approach as we are non-experts and wish that this article is accessible to other non-experts as well.

\vspace{0.3cm}

{\bf Acknowledgements:} We wish to thank Gaultier Lambert, Tom Claeys, and Tuomas Orponen for helpful discussions. E.S. was supported by the Academy of Finland CoE ‘Analysis and Dynamics’, as well as the Academy of Finland Project ‘Conformal methods in analysis and random geometry’. C.W. was supported by the Academy of Finland grants 288318 and 308123.

\section{Definitions and required estimates}\label{sec:defs}

We begin by giving a name to the logarithm of the absolute value of the characteristic polynomial as well as the measure we are considering. 

\begin{definition}\label{def:fieldmass}
Let $N\geq 1$ be a positive integer and let $U_N$ be a $N\times N$ random unitary matrix whose law is the Haar measure on the unitary group $\mathrm{U}(N)$. Consider the random function $X_N:[0,2\pi]\to [-\infty,\infty)$

\begin{equation*}
X_N(\theta)=\log |\det(U_N-e^{i\theta})|=-\frac{1}{2}\sum_{k=1}^\infty \frac{1}{k}\left[e^{-ik\theta}\Tr U_N^k+e^{ik\theta}\Tr U_N^{-k}\right]
\end{equation*}

\noindent and for $\beta>0$ the sequence of random measures

\begin{equation*}
\mu_{N,\beta}(d\theta)=\frac{e^{\beta X_N(\theta)}}{\E e^{\beta X_N(\theta)}}\frac{d\theta}{2\pi}=\frac{|\det(U_N-e^{i\theta})|^\beta}{\E |\det(U_N-e^{i\theta})|^\beta}\frac{d\theta}{2\pi}
\end{equation*}

\noindent understood as measures on the unit circle $\T$.
\end{definition}

Of critical importance in the approach of \cite{berestycki,los} are suitable approximations to the field $X_N$. Our approximation will be a truncation (though other ones could work just as well) and we introduce the following notation for it.

\begin{definition}\label{def:fieldtrunc}
For $N,M$ positive integers and $\beta>0$, write 

\begin{equation}\label{eq:fieldtrunc}
X_{N,M}(\theta)=-\frac{1}{2}\sum_{k=1}^M \frac{1}{k}\left[e^{-ik\theta}\Tr U_N^k+e^{ik\theta}\Tr U_N^{-k}\right]
\end{equation}

\noindent and 

\begin{equation}\label{eq:meastrunc}
\mu_{N,\beta}^{(M)}(d\theta)=\frac{e^{\beta X_{N,M}(\theta)}}{\E e^{\beta X_{N,M}(\theta)}}\frac{d\theta}{2\pi}
\end{equation}

\noindent interpreted as a measure on $\T$.
\end{definition}

The role such truncations play is that we'll be interested in for example how the process $M\mapsto X_{N,M}(\theta)$ behaves with respect to the probability measure $\frac{e^{\beta X_N(\theta)}}{\E e^{\beta X_N(\theta)}}\Prob(d\omega)$ for a fixed $\theta$. Here $\Prob(d\omega)$ denotes the law of the CUE random matrix $U_N$. What ends up being relevant are suitable exponential moments. As suggested earlier, the relevant estimates for these exponential moments will be obtained through Riemann-Hilbert methods relying heavily on results of \cite{dik2,ck}. Our main technical result is the following one, and its proof relies critically on ideas from \cite{abb}. 

\begin{proposition}\label{pr:testimate}
Let $\alpha_1,\alpha_2\in \R$, $\beta_1,\beta_2\geq 0$, $\delta>0$, $M\in \Z_+$, be independent of $N$. Moreover,   let $\mathcal{T}:\C\setminus \lbrace 0\rbrace\to \C$, $\mathcal{T}(z)=\sum_{0<|k|\leq M}\mathcal{T}_{k} z^k$ be independent of $N$ and real-valued. Then as $N\to\infty$, 

\begin{align*}
&\frac{\E e^{\beta_1 X_N(\theta)+\beta_2 X_N(\theta')+\alpha_1 X_{N,K_1}(\theta)+\alpha_2 X_{N,K_2}(\theta)+\mathrm{Tr}\mathcal{T}(U_N)}}{\E e^{\beta_1 X_N(\theta)+\beta_2 X_N(\theta')}}\\
&=(1+o(1))e^{\frac{\alpha_1^2+2\alpha_1\alpha_2+2\alpha_1\beta_1}{4}\sum_{k=1}^{K_1}\frac{1}{k}+\frac{\alpha_2^2+2\alpha_2\beta_1}{4}\sum_{k=1}^{K_2}\frac{1}{k}+\frac{\alpha_1\beta_2}{2}\sum_{k=1}^{K_1}\frac{\cos k(\theta-\theta')}{k}+\frac{\alpha_2\beta_2}{2}\sum_{k=1}^{K_1}\frac{\cos k(\theta-\theta')}{k}}\\
&\quad \times e^{\sum_{k=1}^Mk \mathcal{T}_k \mathcal{T}_{-k}-\frac{\alpha_1}{2}\sum_{k=1}^{\min(K_1,M)}\left(\mathcal{T}_k e^{ik\theta}+\mathcal{T}_{-k}e^{-ik\theta}\right)-\frac{\alpha_2}{2}\sum_{k=1}^{\min(K_2,M)}\left(\mathcal{T}_k e^{ik\theta}+\mathcal{T}_{-k}e^{-ik\theta}\right)-\frac{\beta_1}{2}\mathcal{T}(e^{i\theta})-\frac{\beta_2}{2}\mathcal{T}(e^{i\theta'})}
\end{align*}

\noindent uniformly in $\theta,\theta'\in[0,2\pi]$ and uniformly in $1\leq K_1\leq K_2\leq N^{1-\delta}$. Moreover, if we restrict $(\mathcal{T}_k)_{k=1}^M$ to lie in some compact subset of $\C^M$, then the error $o(1)$ is uniform on this compact set.
\end{proposition}
We actually apply this proposition in a few different ways, where we specialize to a situation where e.g. $\beta_1=0$, or $\mathcal{T}=0$. To save the reader some effort in following our argument later on, we state explicitly the different estimates that follow from this and will be used in our proof. The proof of the following corollary is immediate and we omit the details.
\begin{corollary}\label{cor:estimates}
Let $\alpha\in \R$, $\beta>0$, and $\delta\in(0,1)$ be fixed. Then as $N\to\infty$, 

\begin{equation}\label{eq:1ptest}
\frac{\E e^{\alpha X_{N,K}(0)}e^{\beta X_N(0)}}{\E e^{\beta X_N(0)}}=e^{\left(\frac{\alpha^2}{2}+\alpha\beta\right)\frac{1}{2}\sum_{k=1}^{K}\frac{1}{k}}(1+o(1)),
\end{equation}

\noindent uniformly in $K\leq N^{1-\delta}$.

Moreover, for fixed $M\in \Z_+$, $\alpha_1,\alpha_2\in \R$, $\beta>0$, and $\delta\in(0,1):$ as $N\to\infty$

\begin{align}\label{eq:2ptest1}
\frac{\E e^{\beta X_N(\theta)+\beta X_{N,M}(\theta')+\alpha_1 X_{N,K_1}(\theta)+\alpha_2 X_{N,K_2}(\theta)}}{\E e^{\beta X_N(\theta)+\beta X_{N,M}(\theta')}}&=(1+o(1))e^{\sum_{j=1}^{K_1}\frac{1}{j}\left[\frac{\alpha_1^2}{4}+\frac{\alpha_1\beta}{2}+\frac{\alpha_1\alpha_2}{2}\right]+\sum_{j=1}^{K_2}\frac{1}{j}\left[\frac{\alpha_2^2}{4}+\frac{\alpha_2\beta}{2}\right]}\\
\notag &\quad \times e^{\frac{\alpha_1\beta}{2}\sum_{j=1}^{\min(K_1,M)}\frac{\cos j(\theta-\theta')}{j}+\frac{\alpha_2\beta}{2}\sum_{j=1}^{\min(K_2,M)}\frac{\cos j(\theta-\theta')}{j}}
\end{align} 

\noindent uniformly in $\theta,\theta'\in[0,2\pi]$ and uniformly in $1\leq K_1\leq K_2\leq N^{1-\delta}$.

Finally let $\alpha_1,\alpha_2\in \R$, $\beta>0$, and $\delta\in(0,1)$ be fixed. Then as $N\to\infty$

\begin{align}\label{eq:2ptest2}
\frac{\E e^{\beta X_N(\theta)+\beta X_N(\theta')+\alpha_1 X_{N,K_1}(\theta)+\alpha_2 X_{N,K_2}(\theta)}}{\E e^{\beta X_N(\theta)+\beta X_N(\theta')}}&=e^{\sum_{j=1}^{K_1}\frac{1}{j}\left[\frac{\alpha_1^2}{4}+\frac{\alpha_1\beta}{2}(\cos j(\theta-\theta')+1)+\frac{\alpha_1\alpha_2}{2}\right]}\\
\notag &\quad \times e^{\sum_{j=1}^{K_2}\frac{1}{j}\left[\frac{\alpha_2^2}{4}+\frac{\alpha_2\beta}{2}(\cos j(\theta-\theta')+1)\right]}(1+o(1))
\end{align} 

\noindent uniformly in $\theta,\theta'\in[0,2\pi]$ and uniformly in $1\leq K_1\leq K_2\leq N^{1-\delta}$. 
\end{corollary}

Heuristically, this result is easy to motivate based on the well known central limit theorem going back to Diaconis and Shahshahani \cite{DS,DE} (or actually even further, to the strong Szeg\H{o} theorem), namely that for any fixed $M$, $(\mathrm{Tr}U_N^j)_{j=1}^M$ converges in law to $(\sqrt{j}Z_j)_{j=1}^M$, where $Z_j$ are i.i.d. standard complex Gaussian random variables (i.e. real and imaginary parts are i.i.d. $N(0,1/2)$ random variables). If one were to replace all $\mathrm{Tr}U_N^j$, by $\sqrt{j}Z_j$, then a formal Gaussian calculation would produce Proposition \ref{pr:testimate}. Naturally, $\mathrm{Tr}U_N^j$ are not exactly Gaussians, and Proposition \ref{pr:testimate} can be seen as a result stating that for $j\leq N^{1-\epsilon}$ they are all very close to being Gaussian even when we bias our probability measure by the factor $e^{\beta_1 X_N(\theta)+\beta_2 X_N(\theta')}$.

We need also a weaker version of this result, but in the case where $\mathcal{T}$ is allowed to be complex valued. Luckily what we need follows from classical results concerning asymptotics of Toeplitz determinants with Fisher-Hartwig singularities due to Widom \cite{widom} -- and as such, we do not give a proof here. For the connection between such expectations and Toeplitz determinants, we refer the reader to Section \ref{sec:rmtrhp}.

\begin{proposition}[Widom]\label{pr:2pfourier}
Let $L,M\in \Z_+$, $(s_j)_{j\in \Z_+},(t_j)_{j\in \Z_+},$ and $\theta,\theta'\in[0,2\pi)$, $\theta\neq \theta'$ be independent of $N$ and let $L\leq M$. Then

\begin{align*}
\lim_{N\to\infty}&\frac{\E e^{\sum_{j=1}^{L}\left[\frac{s_j}{\sqrt{j}}\left(\Tr U_N^j+\Tr U_N^{-j}\right)+i\frac{t_j}{\sqrt{j}}\left(\Tr U_N^j-\Tr U_N^{-j}\right)\right] } e^{\beta X_{N,M}(\theta)+\beta X_{N,M}(\theta')}}{\E e^{\beta X_{N,M}(\theta)+\beta X_{N,M}(\theta')}}\\
&=\lim_{N\to\infty}\frac{\E e^{\sum_{j=1}^{L}\left[\frac{s_j}{\sqrt{j}}\left(\Tr U_N^j+\Tr U_N^{-j}\right)+i\frac{t_j}{\sqrt{j}}\left(\Tr U_N^j-\Tr U_N^{-j}\right)\right] }  e^{\beta X_N(\theta)+\beta X_{N,M}(\theta')}}{\E e^{\beta X_N(\theta)+\beta X_{N,M}(\theta')}}\\
&=\lim_{N\to\infty}\frac{\E e^{\sum_{j=1}^{L}\left[\frac{s_j}{\sqrt{j}}\left(\Tr U_N^j+\Tr U_N^{-j}\right)+i\frac{t_j}{\sqrt{j}}\left(\Tr U_N^j-\Tr U_N^{-j}\right)\right] } e^{\beta X_N(\theta)+\beta X_N(\theta')}}{\E e^{\beta X_N(\theta)+\beta X_N(\theta')}}\\
&=e^{-\frac{\beta}{2}\sum_{j=1}^{L}\left[\frac{s_j+it_j}{\sqrt{j}}(e^{ij\theta}+e^{ij\theta'})+\frac{s_j-it_j}{\sqrt{j}}(e^{-ij\theta}+e^{-ij\theta'})\right]}e^{\sum_{j=1}^{L}(s_j^2+t_j^2)}\notag. 
\end{align*}
\end{proposition}

\begin{remark}\label{rem:bias}
Note that we can view the quantities of interest here as Laplace transforms of laws of random variables and the result can be viewed as saying that under the biased probability measure $\frac{e^{\beta X_N(\theta)+\beta X_N(\theta')}}{\E e^{\beta X_N(\theta)+\beta X_N(\theta')}}\Prob(d\omega)$ $($and the similar ones with one or two $X_N$ replaced by $X_{N,M})$, the random vector $(\Tr U_N^j/\sqrt{j})_{j=1}^{L}$ converges in law to the random vector $(-\beta \frac{e^{ij\theta}+e^{ij\theta'}}{2\sqrt{j}}+Z_j)_{j=1}^{L}$, where $Z_j$ are i.i.d. standard complex Gaussians. Again, in the above remark, $\Prob(d\omega)$ denotes the law of the CUE random matrix $U_N$.
\end{remark}

We'll also need asymptotics of the normalizing quantities in Proposition \ref{pr:2pfourier} and Proposition \ref{pr:testimate}. More precisely, we will need the following result, which in a pointwise form goes back to Widom \cite{widom}, but the uniformity is argued e.g. in \cite[Theorem 1.1 and Remark 1.4]{dik2}.

\begin{theorem}[Deift, Its, and Krasovsky]\label{th:dik}
Let $\epsilon>0$ and $\beta_1,\beta_2\geq 0$ be fixed. Also write 

\begin{equation}\label{eq:ddef}
d(\theta,\theta'):=\min(|\theta-\theta'|,2\pi-|\theta-\theta'|). 
\end{equation}

\noindent Then 

\begin{equation*}
\lim_{N\to\infty}\frac{\E e^{\beta_1 X_N(\theta)+\beta_2 X_N(\theta')}}{\E e^{\beta_1 X_N(\theta)}\E e^{\beta_2 X_N(\theta')}}=|e^{i\theta}-e^{i\theta'}|^{-\beta_1\beta_2/2}
\end{equation*}

\noindent uniformly in $\lbrace (\theta,\theta')\in[0,2\pi]^2: d(\theta,\theta')\geq \epsilon\rbrace$. Moreover for a fixed $M\in \Z_+$

\begin{align*}
\lim_{N\to\infty}&\frac{\E e^{\beta_1 X_N(\theta)+\beta_2 X_{N,M}(\theta')}}{\E e^{\beta_1 X_N(\theta)}\E e^{\beta_2 X_{N,M}(\theta')}}=\lim_{N\to\infty}\frac{\E e^{\beta_1 X_{N,M}(\theta)+\beta_2 X_{N,M}(\theta')}}{\E e^{\beta_1 X_{N,M}(\theta)}\E e^{\beta_2 X_{N,M}(\theta')}}=e^{\frac{\beta_1\beta_2}{2}\sum_{j=1}^M \frac{\cos j(\theta-\theta')}{j}}
\end{align*}

\noindent uniformly in $\theta,\theta'\in[0,2\pi]$.
\end{theorem}

We will also need to control the behavior of $\E e^{\beta X_N(\theta)+\beta X_N(\theta')}$ for $\theta$ close to $\theta'$ (note that we need this only for $\beta_1=\beta_2=\beta$), and to do this, we make use of the following result due to Claeys and Krasovsky (a combination of \cite[Theorem 1.11 and Theorem 1.15]{ck}).

\begin{theorem}[Claeys and Krasovsky]\label{th:ck}
For $\theta,\theta'\in[0,2\pi]$, let $d(\theta,\theta')$ be as in \eqref{eq:ddef}. For fixed $\epsilon,\delta\in(0,1)$ and $\beta\geq \sqrt{2}$, the following asymptotics hold

\begin{itemize}[leftmargin=0.5cm]
\item[1.] For $N^{\delta-1}\leq d(\theta,\theta')\leq \epsilon$, as $N\to\infty$

\begin{equation}\label{eq:CK1}
\frac{\E e^{\beta X_N(\theta)+\beta X_N(\theta')}}{\E e^{\beta X_N(\theta)} \E e^{\beta X_N(\theta')}}=\mathcal{O}\left(d(\theta,\theta')^{-\frac{\beta^2}{2}}\right),
\end{equation}

\noindent where the implied constant is uniform in $\lbrace (\theta,\theta')\in [0,2\pi]^2: N^{\delta-1}\leq d(\theta,\theta')\leq \epsilon\rbrace.$

\item[2.] As $N\to\infty$

\begin{equation}\label{eq:CK2}
\int_{d(\theta,\theta')\leq N^{\delta-1}}\frac{\E e^{\beta X_N(\theta)+\beta X_N(\theta')}}{\E e^{\beta X_N(\theta)}\E e^{\beta X_N(\theta')}}\frac{d\theta}{2\pi}\frac{d\theta'}{2\pi}=\mathcal{O}\left(N^{\frac{\beta^2}{2}-1}\log N\right).
\end{equation}

\end{itemize}

\end{theorem}

Finally we also point out the following basic result which can be proven e.g. from  \cite[the proof of Lemma 6.5]{js} -- we omit the details.

\begin{lemma}\label{le:logsum}
Let $M\in \Z_+$ and $d(\theta,\theta')$ be as in \eqref{eq:ddef}. Then

\begin{equation}
\sum_{j=1}^M \frac{\cos j(\theta-\theta')}{j}=\min(\log^{+} d(\theta,\theta')^{-1},\log M)+\mathcal{O}(1),
\end{equation}

\noindent where $\log^{+}x=\max(0,\log x)$ and $\mathcal{O}(1)$ denotes a quantity that is bounded in $M\in \Z_+$ and $\theta,\theta'\in[0,2\pi]$.

\end{lemma}

\section{Proof of Theorem \ref{th:main} assuming the exponential estimates}\label{sec:mainproof}

In this section, we'll prove Theorem \ref{th:main} assuming Proposition \ref{pr:testimate}. We will also rely heavily on Proposition \ref{pr:2pfourier}, Theorem \ref{th:dik}, and Theorem \ref{th:ck}. This is essentially the same as the main argument of \cite{los}, though we formulate things in a slightly different way. We start by noting that it is a basic fact that a sequence of random measures $(\nu_N)$ on the unit circle $\T$ converge in law (with respect to the topology of weak convergence of measures) to a measure $\nu$ on the unit circle if the random variables $\int_\T \varphi d\nu_N$ converge in law to $\int_\T \varphi d\nu$ for each (deterministic) $\varphi\in C(\T,[0,\infty))$ (i.e. the space of continuous functions from $\T$ to $[0,\infty)$). For details, see e.g. \cite[Chapter 4]{kall}. To prove this in our case (where $\nu_N=\mu_{N,\beta}$ and $\nu=\mu_\beta$), we introduce the following notation.

\begin{definition}\label{def:split}
Let $\delta\in(0,1)$ and $\beta\in(0,2)$ be fixed and write $k_N=k_N(\delta)=\lfloor \log_2 N^{1-\delta}\rfloor$. Also let $l,M\in \Z_+$ be fixed and satisfy $2^{l}\leq M$. Moreover, let $\gamma\in(\beta,2)$ and $\varphi\in C(\T,[0,\infty))$ be fixed. Define

\begin{equation}\label{eq:E1}
E^{(1)}=E^{(1)}_{N,l,\gamma,\beta}=\int_0^{2\pi}\varphi(e^{i\theta})\mathbf{1}\lbrace \exists k\in \lbrace l,...,k_N\rbrace: X_{N,2^k}(\theta)>\gamma\E X_{N,2^k}(\theta)^2\rbrace \mu_{N,\beta}(d\theta),
\end{equation}

\begin{align}\label{eq:E2}
E^{(2)}&=E^{(2)}_{N,M,l,\gamma,\beta}\\
\notag &=\int_0^{2\pi}\varphi(e^{i\theta})\mathbf{1}\big\lbrace X_{N,2^k}(\theta)\leq \gamma\E X_{N,2^k}(\theta)^2,\ \forall k\in\lbrace l,...,k_N\rbrace\big\rbrace \mu_{N,\beta}(d\theta)\\
\notag &\quad -\int_0^{2\pi}\varphi(e^{i\theta})\mathbf{1}\big\lbrace X_{N,2^k}(\theta)\leq \gamma\E X_{N,2^k}(\theta)^2,\ \forall k\in\lbrace l,...,\lfloor \log_2 M\rfloor\rbrace\big\rbrace \mu_{N,\beta}^{(M)}(d\theta),
\end{align}

\noindent and 

\begin{equation}\label{eq:C}
G=G_{N,M,l,\gamma,\beta}=\int_0^{2\pi}\varphi(e^{i\theta})\mathbf{1}\big\lbrace X_{N,2^k}(\theta)\leq \gamma \E X_{N,2^k}(\theta)^2,\ \forall k\in\lbrace l,...,\lfloor\log_2 M\rfloor \rbrace\big\rbrace \mu_{N,\beta}^{(M)}(d\theta).
\end{equation}
\end{definition}

Note that $\int_0^{2\pi}\varphi(e^{i\theta})\mu_{N,\beta}(d\theta)=G+E^{(1)}+E^{(2)}$. The interpretation is that the $E$-terms are error terms and $G$ is our main term in that with a suitable limiting procedure related to $N,M,l$, the error terms vanish in a suitably strong sense and $G$ converges to the correct object. We break this up into steps.

\begin{lemma}\label{le:E1}
When we first let $N\to\infty$ and then $l\to \infty$, $E^{(1)}$ tends to zero in $L^1(\Prob)$. 
\end{lemma}

\begin{proof}
This is a consequence of Proposition \ref{pr:testimate} applied to the case $\beta_1=\beta$, $\beta_2=0$, $\alpha_1=\gamma-\beta$, $\alpha_2=0$, and $\mathcal{T}=0$. By a simple union bound as well as the bound $\mathbf{1}\lbrace x>0\rbrace\leq e^{(\gamma-\beta) x}$ (as $\gamma>\beta$), we can write 

\begin{equation*}
\mathbf{1}\left\lbrace \exists k\in \lbrace l,...,k_N\rbrace: X_{N,2^k}(\theta)>\gamma\E X_{N,2^k}(\theta)^2\right\rbrace \leq \sum_{k=l}^{k_N} e^{(\gamma-\beta)(X_{N,2^k}(\theta)-\gamma\E X_{N,2^k}(\theta)^2)}.
\end{equation*}

First making use of the rotation invariance of the law of the eigenvalues of $U_N$, Proposition \ref{pr:testimate} (with the specialization described above), and the fact that $\E \Tr U_N^j \Tr U_N^k=\delta_{k,-j}\min(|k|,N)$ (see \cite{DS,DE}) one can write for large enough $N$

\begin{align*}
\E E^{(1)}&\leq ||\varphi||_\infty \sum_{k=l}^{k_N}\frac{\E e^{(\gamma-\beta)(X_{N,2^k}(0)-\gamma\E X_{N,2^k}(0)^2)+\beta X_N(0)}}{\E e^{\beta X_N(0)}}\\
&\leq  2 ||\varphi||_\infty \sum_{k=l}^{k_N}e^{\left(\frac{(\gamma-\beta)^2}{2}+(\gamma-\beta)\beta\right)\frac{1}{2}\sum_{j=1}^{2^k}\frac{1}{j}} e^{-(\gamma-\beta)\gamma \frac{1}{2}\sum_{j=1}^{2^k}\frac{1}{j}}\\
&\leq C \sum_{k=l}^\infty 2^{-\frac{(\gamma-\beta)^2}{4}k},
\end{align*}

\noindent where $C>0$ is a constant depending only on $\varphi$. This obviously tends to zero when $l\to\infty$.
\end{proof}

Our next task will be to show that $E^{(2)}$ converges to zero in this same limit, and we'll find it actually more convenient to prove convergence in $L^2(\Prob)$. To simplify notation slightly, we give names to some of the main objects we encounter. We'll start with an analogue of $X_{N,M}$ in the Gaussian setting --  namely a $N\to\infty$ limit of $X_{N,M}$.

\begin{definition}\label{def:gaussian}
For $M\in \Z_+$ and $\theta\in[0,2\pi]$, let 

\begin{equation*}
X^{(M)}(\theta)=\sum_{j=1}^{M}\frac{1}{2\sqrt{j}}(Z_j e^{-ij\theta}+Z_j^* e^{ij\theta}),
\end{equation*}

\noindent where $(Z_j)_{j=1}^\infty$ are i.i.d. standard complex Gaussians. Moreover, write 

\begin{equation*}
\Sigma^{(M)}(\theta-\theta')=\E X^{(M)}(\theta)X^{(M)}(\theta')=\sum_{j=1}^M \frac{\cos j(\theta-\theta')}{2j}.
\end{equation*} 
\end{definition}

\noindent We'll also write $\Prob$ for the distribution of the Gaussians $(Z_j)_{j=1}^\infty$. 

Next we'll introduce suitable biased probability measures both in the random matrix setting as well as in the Gaussian one.

\begin{definition}\label{def:bias}
For given $\theta,\theta'\in[0,2\pi]$ let

\begin{align}
\label{eq:qprob1}\Q_{N,1}^{(\theta,\theta')}(d\omega)&=\frac{e^{\beta X_N(\theta)+\beta X_N(\theta')}}{\E e^{\beta X_N(\theta)+\beta X_N(\theta')}}\Prob(d\omega)\\
\label{eq:qprob2}\Q_{N,M,2}^{(\theta,\theta')}(d\omega)&=\frac{e^{\beta X_N(\theta)+\beta X_{N,M}(\theta')}}{\E e^{\beta X_N(\theta)+\beta X_{N,M}(\theta')}}\Prob(d\omega)\\
\label{eq:qprob3}\Q_{N,M,3}^{(\theta,\theta')}(d\omega)&=\frac{e^{\beta X_{N,M}(\theta)+\beta X_{N,M}(\theta')}}{\E e^{\beta X_{N,M}(\theta)+\beta X_{N,M}(\theta')}}\Prob(d\omega)
\end{align}

\noindent and 

\begin{equation}\label{eq:qprobG}
\G_M^{(\theta,\theta')}(d\omega)=\frac{e^{\beta X^{(M)}(\theta)+\beta X^{(M)}(\theta')}}{\E e^{\beta X^{(M)}(\theta)}\E e^{\beta X^{(M)}(\theta')}}\Prob(d\omega).
\end{equation}

\end{definition}

We also introduce notation for barrier events appearing in $E^{(2)}$ and which will appear in the $N\to\infty$ limit of $E^{(2)}$ as well.

\begin{definition}\label{def:barrier}
For $\theta\in [0,2\pi]$, $l,M\in \Z_+$ satisfying $l\leq \lfloor \log_2 M\rfloor$ and $k_N$ as in Definition $\ref{def:split}$, let 

\begin{align*}
B_{N,l}(\theta)&=\big\lbrace X_{N,2^k}(\theta)\leq \gamma \E X_{N,2^k}(\theta)^2 \ \forall k\in \lbrace l,...,k_N\rbrace\big\rbrace \\
B_{N,l,M}(\theta)&=\big\lbrace X_{N,2^k}(\theta)\leq \gamma \E X_{N,2^k}(\theta)^2 \ \forall k\in \lbrace l,...,\lfloor\log_2 M\rfloor\rbrace\big\rbrace
\end{align*}

\noindent and in the Gaussian setting we find it useful to generalize slightly$:$ for $Y=(Y_k)_{k=1}^\infty\in \R^{\Z_+}$, let

\begin{align*}
B_{l,M}^G(\theta;Y)&=\big\lbrace X^{(2^k)}(\theta)\leq Y_{2^k} \ \forall k\in \lbrace l,...,\lfloor \log_2 M\rfloor\rbrace\big\rbrace.
\end{align*}
\end{definition}

Finally we'll give names to the integrals appearing when we expand the square while calculating $\E[E^{(2)}]^2$.

\begin{definition}\label{def:integrals}

For continuous $\varphi:\T\to[0,\infty)$, and fixed $l,M\in \Z_+$ satisfying $l\leq \lfloor \log_2 M\rfloor$, let

\begin{align}\label{eq:integrals}
I_{N,l,1}&=\int_{\T^2}\varphi(e^{i\theta})\varphi(e^{i\theta'}) \Q_{N,1}^{(\theta,\theta')}\left(B_{N,l}(\theta)\cap B_{N,l}(\theta')\right)\frac{\E e^{\beta X_N(\theta)+\beta X_N(\theta')}}{\E e^{\beta X_N(\theta)} \E e^{\beta X_N(\theta')}}\frac{d\theta}{2\pi}\frac{d\theta'}{2\pi}\\
I_{N,l,M,2}&=\int_{\T^2}\varphi(e^{i\theta})\varphi(e^{i\theta'})\Q_{N,M,2}^{(\theta,\theta')}(B_{N,l}(\theta)\cap B_{N,l,M}(\theta')) \frac{\E e^{\beta X_N(\theta)+\beta X_{N,M}(\theta')}}{\E e^{\beta X_N(\theta)} \E e^{\beta X_{N,M}(\theta')}}\frac{d\theta}{2\pi}\frac{d\theta'}{2\pi}\\
I_{N,l,M,3}&=\int_{\T^2}\varphi(e^{i\theta})\varphi(e^{i\theta'}) \Q_{N,M,3}^{(\theta,\theta')}\left(B_{N,l,M}(\theta)\cap B_{N,l,M}(\theta')\right)\frac{\E e^{\beta X_{N,M}(\theta)+\beta X_{N,M}(\theta')}}{\E e^{\beta X_{N,M}(\theta)} \E e^{\beta X_{N,M}(\theta')}}\frac{d\theta}{2\pi}\frac{d\theta'}{2\pi}.
\end{align}

Moreover, for fixed $\epsilon>0$ define

\begin{align}\label{eq:intsplit}
I_{N,l,1}^{(1)}(\epsilon)&=\int_{d(\theta,\theta')\geq \epsilon}\varphi(e^{i\theta})\varphi(e^{i\theta'}) \Q_{N,1}^{(\theta,\theta')}\left(B_{N,l}(\theta)\cap B_{N,l}(\theta')\right)\frac{\E e^{\beta X_N(\theta)+\beta X_N(\theta')}}{\E e^{\beta X_N(\theta)} \E e^{\beta X_N(\theta')}}\frac{d\theta}{2\pi}\frac{d\theta'}{2\pi}\\
I_{N,l,1}^{(2)}(\epsilon)&=\int_{d(\theta,\theta')< \epsilon}\varphi(e^{i\theta})\varphi(e^{i\theta'}) \Q_{N,1}^{(\theta,\theta')}\left(B_{N,l}(\theta)\cap B_{N,l}(\theta')\right)\frac{\E e^{\beta X_N(\theta)+\beta X_N(\theta')}}{\E e^{\beta X_N(\theta)} \E e^{\beta X_N(\theta')}}\frac{d\theta}{2\pi}\frac{d\theta'}{2\pi},
\end{align}

\noindent where we wrote again $d(\theta,\theta')=\min(|\theta-\theta'|,2\pi-|\theta-\theta'|)$.
\end{definition}

Note that 

\begin{equation*}
\E[E^{(2)}]^2=I_{N,l,1}^{(1)}(\epsilon)+I_{N,l,1}^{(2)}(\epsilon)-2I_{N,l,M,2}+I_{N,l,M,3}.
\end{equation*}

We now have all the required notation and results to understand the asymptotics of the different $I$-terms which will eventually show that for $\gamma$ close enough to $\beta$, $E^{(2)}$ tends to zero in $L^2(\Prob)$ when we first let $N\to\infty$, and then $M\to\infty$ for fixed large $l\in \Z_+$. We begin with $I_{N,l,M,3}$ as it is the simplest.

\begin{lemma}\label{le:I3}
For fixed $l,M\in \Z_+$ satisfying $l\leq \lfloor \log_2 M\rfloor$

\begin{align*}
\lim_{N\to\infty}&I_{N,l,M,3}=\int_{\T^2}\varphi(e^{i\theta})\varphi(e^{i\theta'})\Prob\left(B_{l,M}^G(\theta;Y)\cap B_{l,M}^G(\theta';Y)\right)
e^{\frac{\beta^2}{2}\sum_{j=1}^M \frac{\cos j(\theta-\theta')}{j}}\frac{d\theta}{2\pi}\frac{d\theta'}{2\pi},
\end{align*}

\noindent where 

\begin{equation}\label{eq:yseq}
Y_{2^k}=Y_{2^k}(\theta,\theta')=(\gamma-\beta)\E X^{(2^k)}(0)^2-\beta \E X^{(2^k)}(\theta)X^{(2^k)}(\theta').
\end{equation}
\end{lemma}

\begin{proof}

 By the dominated convergence theorem and Theorem \ref{th:dik}, we see that (provided the limit below exists -- which we show shortly)

\begin{equation}
\lim_{N\to\infty}I_{N,M,3}=\int_{\T^2}\varphi(e^{i\theta})\varphi(e^{i\theta'})\left[
\lim_{N\to\infty}\Q_{N,M,3}^{(\theta,\theta')}(B_{N,l,M}(\theta)\cap B_{N,l,M}(\theta'))\right]e^{\frac{\beta^2}{2}\sum_{j=1}^M \frac{\cos j(\theta-\theta')}{j}}\frac{d\theta}{2\pi}\frac{d\theta'}{2\pi}.
\end{equation}

As, for a fixed $\theta$ and for $2^k\leq M$, $X_{N,2^k}(\theta)$ is a continuous function of $(\Tr U_N^j/\sqrt{j})_{j=1}^M$, we see from Proposition \ref{pr:2pfourier} (and in particular Remark \ref{rem:bias}) that under $\Q_{N,M,3}^{(\theta,\theta')}$, as $N\to\infty$

\begin{align*}
X_{N,2^k}(\theta)&\stackrel{d}{\to} \sum_{j=1}^{2^k}\frac{1}{2\sqrt{j}}(Z_j e^{-ij\theta}+Z_j^* e^{ij\theta})+\frac{\beta}{2}\sum_{j=1}^{2^k} \frac{1+\cos j(\theta-\theta')}{j}\\
&=X^{(2^k)}(\theta)+\beta \Sigma^{(2^k)}(0)+\beta \Sigma^{(2^k)}(\theta-\theta'). 
\end{align*}

In fact, this extends to joint convergence for all $k$ and fixed $\theta,\theta'$: under $\Q_{N,M,3}^{(\theta,\theta')}$, as $N\to\infty$ 

\begin{align*}
&(X_{N,2^k}(\theta),X_{N,2^j}(\theta'))_{k,j=1}^{\lfloor \log_2 M\rfloor}\\
&\stackrel{d}{\to}\left(X^{(2^k)}(\theta)+\beta \Sigma^{(2^k)}(0)+\beta \Sigma^{(2^k)}(\theta-\theta'),X^{(2^j)}(\theta')+\beta \Sigma^{(2^j)}(0)+\beta \Sigma^{(2^j)}(\theta-\theta')\right)_{k,j=1}^{\lfloor \log_2 M\rfloor},
\end{align*}

\noindent where we used the notation of Definition \ref{def:gaussian} for the covariance of $X^{(2^k)}$. Again using the fact that $\E \mathrm{Tr}U_N^j \Tr U_N^{-k}=\min(|j|,N)\delta_{j,k},$ we see that $\E X_{N,2^k}(\theta)^2=\Sigma^{(2^k)}(0)$. Thus by the Portmanteau theorem, we find that for fixed $\theta,\theta'\in[0,2\pi]$

\begin{align*}
\lim_{N\to\infty}&\Q_{N,M,3}^{(\theta,\theta')}(B_{N,l,M}(\theta)\cap B_{N,l,M}(\theta'))=\Prob\left(B_{l,M}^G(\theta;Y)\cap B_{l,M}^G(\theta';Y)\right)
\end{align*}

\noindent with $Y$ as in \eqref{eq:yseq}. The application of the Portmanteau theorem is justified by the fact that this is a continuity set for the joint law of $(X^{(2^k)}(\theta),X^{(2^j)}(\theta'))_{k,j=1}^{\lfloor \log_2 M\rfloor}$ since we have Gaussian random variables so lower dimensional sets have zero measure. Thus we are done.
\end{proof}

Consider now $I_{N,l,M,2}$. 

\begin{lemma}\label{le:I2}
For $l,M$ independent of $N$, satisfying $l\leq \lfloor \log_2 M\rfloor$

\begin{align*}
I_{N,l,M,2}=\int_{\T^2}\varphi(e^{i\theta})\varphi(e^{i\theta'})\Prob\left(B_{l,M}^G(\theta;Y)\cap B_{l,M}^G(\theta';Y)\right)
e^{\frac{\beta^2}{2}\sum_{j=1}^M \frac{\cos j(\theta-\theta')}{j}}\frac{d\theta}{2\pi}\frac{d\theta'}{2\pi}+\mathcal{E}_{l,N,M},
\end{align*}

\noindent where again

\begin{equation}
Y_{2^k}=Y_{2^k}(\theta,\theta')=(\gamma-\beta)\E X^{(2^k)}(0)^2-\beta \E X^{(2^k)}(\theta)X^{(2^k)}(\theta')
\end{equation}

\noindent and for a fixed $l\in\Z_+$

\begin{equation*}
\lim_{M\to\infty}\limsup_{N\to\infty}\left|\mathcal{E}_{l,N,M}\right|=0.
\end{equation*}

\begin{proof}
First of all, we observe that by Theorem \ref{th:dik}, we can use the dominated convergence theorem to take the $N\to\infty$ limit inside the integrals that we consider below.

Let us now begin by writing 

\begin{align}\label{eq:splitting}
\Q_{N,M,2}^{(\theta,\theta')}(B_{N,l}(\theta)\cap B_{N,l,M}(\theta'))&=\Q_{N,M,2}^{(\theta,\theta')}(B_{N,l,M}(\theta)\cap B_{N,l,M}(\theta'))\\
\notag &\quad -\Q_{N,M,2}^{(\theta,\theta')}(B_{N,\lfloor \log_2 M\rfloor+1}(\theta)^\mathsf{c}\cap B_{N,l,M}(\theta)\cap B_{N,l,M}(\theta')),
\end{align}

\noindent where the superscript $\mathsf{c}$ denotes the complement of a set.

The $B_{N,l,M}(\theta)\cap B_{N,l,M}(\theta')$-term can be treated exactly as in the proof of Lemma \ref{le:I3} (though now using different parts of Proposition \ref{pr:2pfourier} and Theorem \ref{th:dik}) to conclude that as $N\to\infty$, this converges to the main term in the statement of this lemma. 

Let us thus focus on the error term here. We'll again use simple bounds on the indicator functions here. We bound the indicator of the event of $(X_{N,2^j}(\theta),X_{N,2^k}(\theta'))_{j,k}$ staying under the barrier either by an indicator of an event where $X_{N,2^j}(\theta)$ is below the barrier at a very special $j$ depending $M,\theta,$ and $\theta'$, or then simply by one if this special $j$ is too small. More precisely, we define

\begin{equation}\label{eq:jtt}
j(\theta,\theta')=j_M(\theta,\theta')=\min\left(\left\lfloor \log_{2}^+ d(\theta,\theta')^{-1}\right\rfloor, \lfloor \log_2 M\rfloor \right),
\end{equation}

\noindent where $\log_{2}^+ x=\max(0,\log_2 x)$. We now use the following bound (which follows from the inequality $\mathbf{1}\lbrace x<0\rbrace\leq e^{-\gamma x}$): 

\begin{equation}\label{eq:charbound}
\mathbf{1}\lbrace B_{N,l,M}(\theta)\cap B_{N,l,M}(\theta')\rbrace\leq \begin{cases}
e^{-\gamma \left(X_{N,2^{j(\theta,\theta')}}(\theta)-\gamma \E X_{N,2^{j(\theta,\theta')}}(\theta)^2\right)}, &   d(\theta,\theta')\leq 2^{-l}\\
1, & d(\theta,\theta')>2^{-l}
\end{cases}.
\end{equation}

\noindent while we write again 

\begin{equation*}
\mathbf{1}_{B_{N,\lfloor \log_2 M\rfloor+1}(\theta)^\mathsf{c}}\leq \sum_{k=\lfloor \log_2 M\rfloor+1}^{k_N} e^{(\gamma-\beta)(X_{N,2^k}(\theta)-\gamma \E X_{N,2^k}(\theta)^2)}.
\end{equation*}

With these bounds, using the definition of $\Q_{N,M,2}^{(\theta,\theta')}$ and two applications of Proposition \ref{pr:testimate} -- or in fact Corollary \ref{cor:estimates}\footnote{To study the numerator, we apply the proposition to the case $\beta_1=\beta$, $\beta_2=0$, $\alpha_1=-\gamma$, $\alpha_2=\gamma-\beta$,  and $\mathcal{T}(z)=-\frac{\beta}{2}\sum_{k=1}^M \frac{1}{k}(e^{-ik\theta'}z^k+e^{ik\theta'}z^{-k})$ -- which corresponds to an application of \eqref{eq:1ptest} in Corollary \ref{cor:estimates}. Similarly to study the denominator, we take $\beta_1=\beta$, $\beta_2=\alpha_1=\alpha_2=0$, and $\mathcal{T}(z)=-\frac{\beta}{2}\sum_{k=1}^M \frac{1}{k}(e^{-ik\theta'}z^k+e^{ik\theta'}z^{-k})$, which also corresponds to \eqref{eq:1ptest} of Corollary \ref{cor:estimates} -- note that the coefficients of $\mathcal{T}$ are uniformly bounded in $\theta'$.} -- we see that for $d(\theta,\theta')\leq 2^{-l}$, as $N\to\infty$

\begin{align}\label{eq:I2eq}
&\Q_{N,M,2}^{(\theta,\theta')}(B_{N,\lfloor \log_2 M\rfloor+1}(\theta)^\mathsf{c}\cap B_{N,l,M}(\theta)\cap B_{N,l,M}(\theta'))\\
\notag &\leq \sum_{k=\lfloor \log_2 M\rfloor+1}^{k_N}\frac{\E e^{(\gamma-\beta)(X_{N,2^k}(\theta)-\gamma \E X_{N,2^k}(\theta)^2)-\gamma \left(X_{N,2^{j(\theta,\theta')}}(\theta)-\gamma \E X_{N,2^{j(\theta,\theta')}}(\theta)^2\right)+\beta X_N(\theta)+\beta X_{N,M}(\theta')}}{\E e^{\beta X_N(\theta)+\beta X_{N,M}(\theta')}}\\
\notag &=(1+\mathit{o}(1))e^{ \sum_{p=1}^{2^{j(\theta,\theta')}}\frac{\frac{\gamma^2}{2}-\gamma\beta+\gamma^2-\gamma(\gamma-\beta)}{2p}-\frac{\gamma\beta}{2}\sum_{p=1}^{2^{j(\theta,\theta')}}\frac{\cos p(\theta-\theta')}{p}+\frac{(\gamma-\beta)\beta}{2}\sum_{p=1}^{M}\frac{\cos p(\theta-\theta')}{p}}\\
\notag &\quad \times\sum_{k=\lfloor \log_2 M\rfloor+1}^{k_N}e^{\sum_{p=1}^{2^k}\frac{\frac{(\gamma-\beta)^2}{2}-(\gamma-\beta)\gamma+(\gamma-\beta)\beta}{2p}}.
\end{align}

Noting that 

\begin{equation*}
\sum_{p=1}^{2^{j(\theta,\theta')}}\frac{\cos p(\theta-\theta')}{p}=\sum_{p=1}^{2^{j(\theta,\theta')}}\frac{1}{p}+\mathcal{O}(1),
\end{equation*}

\noindent uniformly in $\theta,\theta'$ and the sum over $k$ can be estimated as in the proof of Lemma \ref{le:E1}, we thus find (using Lemma \ref{le:logsum}) that for some constant $C>0$ independent of $N,M,l,\theta,\theta'$

\begin{align*}
\Q_{N,M,2}^{(\theta,\theta')}&(B_{N,\lfloor \log_2 M\rfloor+1}(\theta)^\mathsf{c}\cap B_{N,l,M}(\theta)\cap B_{N,l,M}(\theta'))\leq C M^{-\frac{(\gamma-\beta)^2}{4}}\left[\min(M,d(\theta,\theta')^{-1})\right]^{\frac{\gamma^2}{4}-\frac{\beta^2}{2}}.
\end{align*}

\noindent  We now see from Theorem \ref{th:dik} and again Lemma \ref{le:logsum} that there exists some constant $C>0$ (again possibly different from the previous ones) not depending on $N,M,\theta,\theta'$ such that 

\begin{align*}
\Q_{N,M,2}^{(\theta,\theta')}&(B_{N,\lfloor \log_2 M\rfloor+1}(\theta)^\mathsf{c}\cap B_{N,l,M}(\theta)\cap B_{N,l,M}(\theta'))\frac{\E e^{\beta X_N(\theta)+\beta X_{N,M}(\theta')}}{\E e^{\beta X_N(\theta)}\E e^{\beta X_{N,M}(\theta')}}\\
&\leq C M^{-\frac{(\gamma-\beta)^2}{4}}\left[\min(M,d(\theta,\theta')^{-1})\right]^{\frac{\gamma^2}{4}}\leq C M^{-\frac{(\gamma-\beta)^2}{4}}d(\theta,\theta')^{-\frac{\gamma^2}{4}}.
\end{align*}

As $\gamma<2$, this is integrable over $\T^2$ so we see that 

\begin{align}\label{eq:I2lim1}
\lim_{M\to\infty}\limsup_{N\to\infty}\int_{d(\theta,\theta')\leq 2^{-l}}&\varphi(e^{i\theta})\varphi(e^{i\theta'})\Q_{N,M,2}^{(\theta,\theta')}(B_{N,\lfloor \log_2 M\rfloor+1}(\theta)^\mathsf{c}\cap B_{N,l,M}(\theta)\cap B_{N,l,M}(\theta'))\\
\notag &\times \frac{\E e^{\beta X_N(\theta)+\beta X_{N,M}(\theta')}}{\E e^{\beta X_N(\theta)}\E e^{\beta X_{N,M}(\theta')}}=0.
\end{align}

Let us now consider the $d(\theta,\theta')\geq 2^{-l}$ case. As indicated above (again using a suitable version of Proposition \ref{pr:testimate} or Corollary \ref{cor:estimates}), we approximate here simply by 

\begin{align*}
\Q_{N,M,2}^{(\theta,\theta')}&(B_{N,\lfloor \log_2 M\rfloor+1}(\theta)^\mathsf{c}\cap B_{N,l,M}(\theta)\cap B_{N,l,M}(\theta'))\\
&\leq \sum_{k=\lfloor \log_2 M\rfloor+1}^{k_N}\frac{\E e^{(\gamma-\beta)\left(X_{N,2^k}(\theta)-\gamma \E X_{N,2^k}(\theta)^2\right)+\beta X_N(\theta)+\beta X_{N,M}(\theta')}}{\E e^{\beta X_{N}(\theta)+\beta X_{N,M}(\theta')}}\\
&\lesssim \sum_{k=\lfloor \log_2 M\rfloor +1}^{k_N} 2^{\frac{(\gamma-\beta)^2}{4}k-\frac{\gamma(\gamma-\beta)}{2}k+\frac{\beta(\gamma-\beta)}{2}k} \left[\min(d(\theta,\theta')^{-1},M)\right]^{\frac{\beta(\gamma-\beta)}{2}}\\
&\lesssim M^{-\frac{(\gamma-\beta)^2}{4}}\left[\min(d(\theta,\theta')^{-1},M)\right]^{\frac{\beta(\gamma-\beta)}{2}}.
\end{align*}

Thus by Theorem \ref{th:dik},

\begin{align*}
\Q_{N,M,2}^{(\theta,\theta')}&(B_{N,\lfloor \log_2 M\rfloor+1}(\theta)^\mathsf{c}\cap B_{N,l,M}(\theta)\cap B_{N,l,M}(\theta'))\frac{\E e^{\beta X_N(\theta)+\beta X_{N,M}(\theta')}}{\E e^{\beta X_N(\theta)}\E e^{\beta X_{N,M}(\theta')}}\lesssim \frac{M^{-\frac{(\gamma-\beta)^2}{4}}}{d(\theta,\theta')^{\frac{\gamma\beta}{2}}}.
\end{align*}

As $l$ is fixed, the fact that this is not an integrable singularity is not problematic and we find for any fixed $l\in \Z_+$

\begin{align*}
\lim_{M\to\infty}\limsup_{N\to\infty}\int_{d(\theta,\theta')> 2^{-l}}&\varphi(e^{i\theta})\varphi(e^{i\theta'})\Q_{N,M,2}^{(\theta,\theta')}(B_{N,\lfloor \log_2 M\rfloor+1}(\theta)^\mathsf{c}\cap B_{N,l,M}(\theta)\cap B_{N,l,M}(\theta'))\\
&\times \frac{\E e^{\beta X_N(\theta)+\beta X_{N,M}(\theta')}}{\E e^{\beta X_N(\theta)}\E e^{\beta X_{N,M}(\theta')}}=0.
\end{align*}

This along with \eqref{eq:I2lim1}, \eqref{eq:splitting}, and the discussion after \eqref{eq:splitting} yield the claim.

\end{proof}
\end{lemma}

Consider finally $I_{N,l,1}$. Our claim about this is the following one.

\begin{lemma}\label{le:I1}
There exists a $\gamma=\gamma(\beta)$ $($but independent of $N$ or $M)$ such that for $l,M$ independent of $N$ and for $\delta>0$ small enough $($but independent of $N$, $M$, and $l)$

\begin{align*}
I_{N,l,1}=\int_{\T^2}\varphi(e^{i\theta})\varphi(e^{i\theta'})\Prob\left(B_{l,M}^G(\theta;Y)\cap B_{l,M}^G(\theta';Y)\right)
e^{\frac{\beta^2}{2}\sum_{j=1}^M \frac{\cos j(\theta-\theta')}{j}}\frac{d\theta}{2\pi}\frac{d\theta'}{2\pi}+\mathcal{E}_{l,N,M}',
\end{align*}

\noindent where again

\begin{equation}
Y_{2^k}=Y_{2^k}(\theta,\theta')=(\gamma-\beta)\E X^{(2^k)}(0)^2-\beta \E X^{(2^k)}(\theta)X^{(2^k)}(\theta')
\end{equation}

\noindent and for all fixed $l\in \Z_+$

\begin{equation*}
\lim_{M\to\infty}\limsup_{N\to\infty}|\mathcal{E}_{l,N,M}'|=0.
\end{equation*}
\end{lemma}

\begin{proof}
Let us begin by writing $I_{N,l,1}=I_{N,l,1}^{(1)}(1/M)+I_{N,l,1}^{(2)}(1/M)$ as in the notation of Definition \ref{def:integrals}. We'll show that as $N\to\infty$, $I_{N,l,1}^{(1)}(1/M)$ behaves essentially as our main term while $I_{N,l,1}^{(2)}(1/M)$ tends to zero. Let us begin with the first claim.

We write

\begin{align*}
&\Q_{N,1}^{(\theta,\theta')}\left(B_{N,l}(\theta)\cap B_{N,l}(\theta')\right)\\
&\quad =\Q_{N,1}^{(\theta,\theta')}\left(B_{N,l,M}(\theta)\cap B_{N,l,M}(\theta')\right)-\Q_{N,1}^{(\theta,\theta')}\left(B_{N,l,M}(\theta)\cap B_{N,l,M}(\theta')\cap B_{N,\lfloor \log_2 M\rfloor+1}(\theta)^\mathsf{c}\right)\\
&\qquad -\Q_{N,1}^{(\theta,\theta')}\left(B_{N,l,M}(\theta)\cap B_{N,l,M}(\theta')\cap B_{N,\lfloor \log_2 M\rfloor+1}(\theta')^\mathsf{c}\right)\\
&\qquad +\Q_{N,1}^{(\theta,\theta')}\left(B_{N,l,M}(\theta)\cap B_{N,l,M}(\theta')\cap B_{N,\lfloor \log_2 M\rfloor+1}(\theta)^\mathsf{c}\cap B_{N,\lfloor \log_2 M\rfloor+1}(\theta')^\mathsf{c}\right).
\end{align*}

We wish to show that only the first term is relevant to the asymptotics (where first $N\to\infty$ and then $M\to\infty$) of $I_{N,l,1}^{(1)}(1/M)$. We note that the last term is of course bounded by the absolute value of the second term and by symmetry the second and third term give the same contribution, so it is enough to show that 

\begin{align*}
\lim_{M\to\infty}\limsup_{N\to\infty}&\int_{d(\theta,\theta')\geq 1/M}\varphi(e^{i\theta})\varphi(e^{i\theta'})\Q_{N,1}^{(\theta,\theta')}\left(B_{N,l,M}(\theta)\cap B_{N,l,M}(\theta')\cap B_{N,\lfloor \log_2 M\rfloor+1}(\theta)^\mathsf{c}\right)\\
&\quad \times\frac{\E e^{\beta X_N(\theta)+\beta X_N(\theta')}}{\E e^{\beta X_N(\theta)}\E e^{\beta X_N(\theta')}}\frac{d\theta}{2\pi}\frac{d\theta'}{2\pi}=0.
\end{align*}
This is handled essentially the same way as in the proof of Lemma \ref{le:I2}. The only differences are now that in \eqref{eq:I2eq}, instead of $e^{\sum_{p=1}^M \cos p(\theta-\theta')/p}$ in the exponential, one has $e^{\sum_{p=1}^k \cos p(\theta-\theta')/p}$, where $k\in \lbrace \lfloor \log_2 M\rfloor+1, ...,k_N\rbrace$. Also when applying Proposition \ref{pr:testimate}, we now have $\beta_1=\beta_2=\beta$. From Lemma \ref{le:logsum}, we see that since we are in the domain where $d(\theta,\theta')\geq 1/M$, this differs from $e^{\sum_{p=1}^M \cos p(\theta-\theta')/p}$ only by a uniformly bounded quantity. The second difference is that now 

\begin{equation*}
\lim_{N\to\infty}\frac{\E e^{\beta X_N(\theta)+\beta X_N(\theta')}}{\E e^{\beta X_N(\theta)}\E e^{\beta X_N(\theta')}}=|e^{i\theta}-e^{i\theta'}|^{-\frac{\beta^2}{2}}
\end{equation*}

\noindent instead of $e^{\frac{\beta^2}{2}\sum_{j=1}^M \frac{\cos j(\theta-\theta')}{j}}$, but again Lemma \ref{le:logsum} implies that these two are comparable in the regime $d(\theta,\theta')\geq 1/M$. With these modifications, Theorem \ref{th:dik}, and the dominated convergence theorem, one sees as in the proof of Lemma \ref{le:I1} and Lemma \ref{le:I2} that

\begin{equation*}
I_{N,l,1}^{(1)}(1/M)-\int_{d(\theta,\theta')\geq 1/M}\varphi(e^{i\theta})\varphi(e^{i\theta'})\Prob\left(B_{l,M}^G(\theta;Y)\cap B_{l,M}^G(\theta';Y)\right)
|e^{i\theta}-e^{i\theta'}|^{-\frac{\beta^2}{2}}\frac{d\theta}{2\pi}\frac{d\theta'}{2\pi}
\end{equation*}

\noindent converges to zero when we first let $N\to\infty$ and then $M\to\infty$. Let us next verify that 

\begin{align*}
\lim_{M\to\infty}&\Bigg[\int_{\T^2}\varphi(e^{i\theta})\varphi(e^{i\theta'})\Prob\left(B_{l,M}^G(\theta;Y)\cap B_{l,M}^G(\theta';Y)\right)\\
&\quad \times \left(\mathbf{1}(d(\theta,\theta')\geq 1/M)
|e^{i\theta}-e^{i\theta'}|^{-\frac{\beta^2}{2}}-e^{\frac{\beta^2}{2}\sum_{p=1}^M \frac{\cos p(\theta-\theta')}{p}}\right)\frac{d\theta}{2\pi}\frac{d\theta'}{2\pi}\Bigg]=0.
\end{align*}
If we could make use of the dominated convergence theorem, we would get this by taking the limit under the integral. To see that the dominated convergence theorem can be applied, one can proceed as in the proof of Lemma \ref{le:I2} namely choosing $j(\theta,\theta')$ as in \eqref{eq:jtt} and applying the bound \eqref{eq:charbound}, one finds as in the proof of Lemma \ref{le:I2} (except that one now has purely Gaussian estimates) that for $d(\theta,\theta')\leq 2^{-l}$ ($d(\theta,\theta')\to 0$ is the only place where one could have problems with integrability)

\begin{align}\label{eq:gestimate}
\Prob\left(B_{l,M}^G(\theta;Y)\cap B_{l,M}^G(\theta';Y)\right)&\leq C\left[\min(d(\theta,\theta)^{-1},M)\right]^{-\frac{\beta\gamma}{4}+\frac{3\gamma(\gamma-\beta)}{4}}.
\end{align}

\noindent as $\beta,\gamma<2$ and we can make $\gamma-\beta$ as small as we want, one can easily check that the use of the dominated convergence theorem is justified in a similar way as in the proof of Lemma \ref{le:I2}.

To conclude, we still need to know that $I_{N,l,1}^{(2)}(1/M)$ converges to zero when we first let $N\to\infty$ and then $M\to\infty$. The argument is again similar, though now instead of $j(\theta,\theta')$ from \eqref{eq:jtt}, we define 

\begin{equation}
\widetilde{j}(\theta,\theta')=\min(\lfloor \log_{2}^+d(\theta,\theta')^{-1}\rfloor,k_N).
\end{equation}

One then finds (once again using Proposition \ref{pr:testimate})

\begin{align*}
\Q_{N,1}^{(\theta,\theta')}\left(B_{N,l}(\theta)\cap B_{N,l}(\theta')\right)&\leq \frac{\E e^{-\gamma(X_{N,\widetilde{j}(\theta,\theta')}(\theta)-\gamma \E X_{N,\widetilde{j}(\theta,\theta')}(\theta)^2)+\beta X_N(\theta)+\beta X_N(\theta')}}{\E e^{\beta X_N(\theta)+\beta X_N(\theta')}}\\
&\leq C \left[\min(d(\theta,\theta')^{-1},N^{1-\delta})\right]^{\frac{3\gamma^2-4\beta \gamma}{4}}.
\end{align*}

From Theorem \ref{th:ck} we then conclude (for a large enough $M$), that there exists a constant $C$ independent of $N,M$ such that as $N\to\infty$, 

\begin{align*}
I_{N,l,1}^{(2)}(1/M)&\leq C\int_{N^{\delta-1}\leq d(\theta,\theta')\leq 1/M} d(\theta,\theta')^{\frac{-3\gamma^2+4\gamma\beta-2\beta^2}{4}}\frac{d\theta}{2\pi}\frac{d\theta'}{2\pi}\\
&\quad +N^{(1-\delta)\frac{3\gamma^2-4\gamma\beta}{4}}\int_{d(\theta,\theta')\leq N^{\delta-1}}\frac{\E e^{\beta X_N(\theta)+\beta X_N(\theta')}}{\E e^{\beta X_N(\theta)}\E e^{\beta X_N(\theta')}}\frac{d\theta}{2\pi}\frac{d\theta'}{2\pi}\\
&\leq C\int_{N^{\delta-1}\leq d(\theta,\theta')\leq 1/M} d(\theta,\theta')^{\frac{-3\gamma^2+4\gamma\beta-2\beta^2}{4}}\frac{d\theta}{2\pi}\frac{d\theta'}{2\pi}\\
&\quad +\mathcal{O}\left(\log N N^{(1-\delta)\frac{3\gamma^2-4\gamma\beta}{4}+\frac{\beta^2}{2}-1}\right).
\end{align*}

For $\gamma$ close enough to $\beta$, the singularity in the integral is integrable, so we see that as $M\to\infty$, the integral tends to zero. For the second term, we note that by taking $\gamma$ close enough to $\beta$ and $\delta$ close enough to zero (both independently of $N$), the exponent of $N$ is negative since $\beta<2$. We conclude that indeed $\lim_{M\to\infty}\limsup_{N\to\infty}I_{N,l,1}^{(2)}(1/M)=0$ and we are done.
\end{proof}

We are now in a position to prove Theorem \ref{th:main} (assuming all the relevant exponential estimates).

\begin{proof}[Proof of Theorem \ref{th:main}]
The strategy of the proof is to use the Portmanteau theorem along with some basic stochastic approximation arguments. Let $F\subset \R$ be closed and $\epsilon>0$. Let us write $F^\epsilon=\lbrace x\in \R: d(x,F)\leq \epsilon\rbrace$. As $\int_\T\varphi d\mu_{N,\beta}=G+E_1+E_2$, we have (by first using Markov for a fixed $N$ and then letting $N\to\infty$)

\begin{equation}\label{eq:p1}
\limsup_{N\to\infty}\Prob\left(\int_\T\varphi d\mu_{N,\beta}\in F\right)\leq \limsup_{N\to\infty}\Prob(G\in F^\epsilon)+\frac{2}{\epsilon}\limsup_{N\to\infty}\E |E^{(1)}|+\frac{2}{\epsilon}\limsup_{N\to\infty}\E |E^{(2)}|.
\end{equation}

By basic approximation, one can use the fact that as $N\to\infty$, $(\mathrm{Tr} U_N^j/\sqrt{j})_{j=1}^{M}$ converges in law to $(Z_j)_{j=1}^M$, where $Z_j$ are i.i.d. standard complex Gaussians, to deduce that when we let $N\to\infty$, $G$ converges in law to 

\begin{align*}
\widetilde{G}_{M,l}&=\int_0^{2\pi}\varphi(e^{i\theta})\mathbf{1}\left\lbrace X^{(2^k)}(\theta)\leq \gamma \E X^{(2^k)}(\theta)^2, \ \forall k\in\lbrace l,...,\lfloor \log_2 M\rfloor\rbrace\right\rbrace \frac{e^{\beta X^{(M)}(\theta)}}{\E e^{\beta X^{(M)}(\theta)}}\frac{d\theta}{2\pi}\\
&=\int_0^{2\pi}\varphi(e^{i\theta})\frac{e^{\beta X^{(M)}(\theta)}}{\E e^{\beta X^{(M)}(\theta)}}\frac{d\theta}{2\pi}\\
&\qquad -\int_0^{2\pi}\varphi(e^{i\theta})\mathbf{1}\left\lbrace \exists k\in\lbrace l,...,\lfloor \log_2 M\rfloor\rbrace: X^{(2^k)}(\theta)>\gamma \E X^{(2^k)}(\theta)^2 \right\rbrace \frac{e^{\beta X^{(M)}(\theta)}}{\E e^{\beta X^{(M)}(\theta)}}\frac{d\theta}{2\pi}\\
&=:\int_0^{2\pi}\varphi(e^{i\theta})\frac{e^{\beta X^{(M)}(\theta)}}{\E e^{\beta X^{(M)}(\theta)}}\frac{d\theta}{2\pi}+\mathsf{E}_{M,l}.
\end{align*}

\noindent The same argument as in the proof of Lemma \ref{le:E1} shows that $\lim_{l\to\infty}\limsup_{M\to\infty}\E|\mathsf{E}_{M,l}|=0$. On the other hand, if we realize our Gaussian random variables on the same probability space, $\int_0^{2\pi}\varphi(e^{i\theta})\frac{e^{\beta X^{(M)}(\theta)}}{\E e^{\beta X^{(M)}(\theta)}}\frac{d\theta}{2\pi}$ is a positive martingale, so it converges almost surely to a non-negative random variable which we call $\int_\T \varphi d\mu_\beta$.\footnote{It follows e.g. from \cite{js} or a modification of the argument of \cite{berestycki} (using in particular e.g. \eqref{eq:gestimate}) that this martingale is actually uniformly integrable and a non-trivial limit exists and with a standard separability argument, this can be used to define the non-trivial random measure $\mu_\beta$. We refer also to \cite{js} for other equivalent constructions of the measure $\mu_\beta$.} Thus taking $M\to\infty$ in \eqref{eq:p1}, we find

\begin{align}\label{eq:p2}
\limsup_{N\to\infty}\Prob\left(\int_\T\varphi d\mu_{N,\beta}\in F\right)&\leq \Prob\left(\int_\T \varphi d\mu_\beta\in F^{2\epsilon}\right)+\frac{1}{\epsilon}\limsup_{M\to\infty}\E|\mathsf{E}_{M,l}|+\frac{2}{\epsilon}\limsup_{N\to\infty}\E |E^{(1)}|\\
\notag &\quad +\frac{2}{\epsilon}\limsup_{M\to\infty}\limsup_{N\to\infty}\E |E^{(2)}|.
\end{align}

\noindent Now combining Lemma \ref{le:I3}, Lemma \ref{le:I2} and Lemma \ref{le:I1}, shows that the $E^{(2)}$-term vanishes.  Thus taking $l\to\infty$ of \eqref{eq:p2} and using Lemma \ref{le:E1} along with our remark that in such a limit also the $\mathsf{E}_{M,l}$ term converges to zero, we see that

$$
\limsup_{N\to\infty}\Prob\left(\int_\T\varphi d\mu_{N,\beta}\in F\right)\leq \Prob\left(\int_\T \varphi d\mu_\beta\in F^{2\epsilon}\right).
$$

\noindent Now letting $\epsilon\to 0$, 

$$
\limsup_{N\to\infty}\Prob\left(\int_\T\varphi d\mu_{N,\beta}\in F\right)\leq \Prob\left(\int_\T \varphi d\mu_\beta\in F\right)
$$

\noindent and we conclude by the Portmanteau theorem.

\end{proof}

\section{Connection between Haar distributed unitary matrices and Riemann-Hilbert problems}\label{sec:rmtrhp}

Our goal from now on will be to prove Proposition \ref{pr:testimate} -- which we finally do in Section \ref{sec:intdi}. We begin by reviewing how this is related to Riemann-Hilbert problems and then begin our analysis of the relevant Riemann-Hilbert problem. The first step in the connection is the well known Heine-Szeg\H{o} identity. Recall that we denote by $\T$ the unit circle. We'll write $(e^{i\theta_j})_{j=1}^N$ for the eigenvalues of $U_N$.

\begin{proposition}[Heine-Szeg\H{o}]\label{pr:hs}
Let $f\in L^1(\T,\C)$. Then

\begin{equation}\label{eq:hs}
\E \prod_{j=1}^N f(e^{i\theta_j})=\det\left(\widehat{f}_{j-k}\right)_{j,k=0}^{N-1}=:D_{N-1}(f),
\end{equation}

\noindent where 

\begin{equation}\label{eq:fourier}
\widehat{f}_j=\frac{1}{2\pi}\int_0^{2\pi}f(e^{i\theta})e^{-ij\theta}d\theta.
\end{equation}
\end{proposition}

\noindent For a proof, one can use e.g. Andreief's identity (see \cite[Lemma 3.2.3]{agz}) and the Vandermonde representation of the law of $(e^{i\theta_j})_{j=1}^N$. 

The next step in our argument is to recall how such determinants are related to orthogonal polynomials.

\begin{definition}
Let $f\in L^1(\T)$ and assume that $f\geq 0$ almost everywhere. Assume further that $D_j(f)\neq 0$ for each $j\geq 0$. Writing $D_{-1}(f)=1$, define for $j\geq 0$

\begin{equation}\label{eq:gramdet}
p_{j}(z;f)=\frac{1}{\sqrt{D_{j}(f)D_{j-1}(f)}}\begin{vmatrix}
\widehat{f}_0 & \widehat{f}_{-1} & \cdots & \widehat{f}_{-j}\\
\widehat{f}_1 & \widehat{f}_0 & \cdots & \widehat{f}_{-j+1}\\
\vdots & \vdots & \ddots & \vdots \\
\widehat{f}_{j-1} & \widehat{f}_{j-2} & \cdots & \widehat{f}_{-1}\\
1 & z & \cdots & z^j
\end{vmatrix}=\chi_j(f)z^j+\mathrm{l.o.t.},
\end{equation}

\noindent where the interpretation is that for $j=0$, the determinant is replaced by the number $1$. Also we note that 

\begin{equation}\label{eq:chi}
\chi_j(f)=\sqrt{\frac{D_{j-1}(f)}{D_j(f)}}.
\end{equation}

\end{definition}

This is just the determinantal representation of the polynomials obtained by applying the Gram-Schmidt procedure on $L^2(\T,f(e^{i\theta})d\theta/2\pi)$ to the monomials. In our case, $f$ will be non-negative and zero only at finitely many points, so these polynomials exist. As the polynomials are constructed by Gram-Schmidt, they are orthonormal: for $0\leq k\leq j$

\begin{equation}\label{eq:ortho1}
\frac{1}{2\pi}\int_0^{2\pi}p_j(e^{i\theta};f) e^{-ik\theta}f(e^{i\theta})d\theta=\frac{1}{\chi_k(f)}\delta_{j,k}.
\end{equation}

\noindent Note that this is equivalent to 

\begin{equation}\label{eq:ortho2}
\frac{1}{2\pi}\int_0^{2\pi}p_j(e^{i\theta};f) \overline{p_k(e^{i\theta};f)}f(e^{i\theta})d\theta=\int_{\T}p_j(z;f)\overline{p_k(z;f)}f(z)\frac{dz}{2\pi i z}=\delta_{j,k}.
\end{equation}

\begin{remark}
By the telescopic structure of the product, and the fact that we defined $D_{-1}(f)=1$, \eqref{eq:chi} implies that

\begin{equation}\label{eq:dchi}
D_{N-1}(f)=\prod_{j=0}^{N-1}\chi_j(f)^{-2}.
\end{equation}
\end{remark}
A classical reference for further information on orthogonal polynomials is \cite{Sz}, and in particular, \cite[Chapter 11]{Sz} for polynomials on the unit circle.

Let us point out that for Proposition \ref{pr:testimate}, the relevant symbol $f$ is 

\begin{equation*}
e^{\mathcal{T}(z)-\sum_{j=1}^{K_1}\frac{\alpha_1}{2j}(z^j e^{-ij\theta}+z^{-j}e^{ij\theta})-\sum_{j=1}^{K_2}\frac{\alpha_2}{2j}(z^j e^{-ij\theta}+z^{-j}e^{ij\theta})}|z-e^{i\theta}|^{\beta_1}|z-e^{i\theta'}|^{\beta_2}.
\end{equation*}

\noindent To have our notation in closer agreement with that in \cite{ck}, we find it convenient to make use of the rotation invariance of of the law of $(e^{i\theta_j})_{j=1}^N$ and use a slightly different symbol (which still produces the same determinant). More precisely, if we define 

\begin{equation}\label{eq:phidef}
\phi=\phi(\theta,\theta')=\begin{cases}
\frac{\theta+\theta'}{2}, & |\theta-\theta'|\leq \pi\\
\frac{\theta+\theta'+2\pi}{2}, & |\theta-\theta'|\in(\pi,2\pi)
\end{cases},
\end{equation}

\noindent then one has for any $F:\T\to \C$

$$
\E \prod_{j=1}^N F(e^{i\theta_j})=\E\prod_{j=1}^N F(e^{i\phi+i\theta_j}).
$$

\noindent Motivated by this, let us make the following definition:

\begin{definition}\label{def:f}
Let $u=\frac{d(\theta,\theta')}{2}$, $\phi$ be as in \eqref{eq:phidef}, and let $V:\C\setminus \lbrace 0\rbrace\to \C$,

\begin{align}\label{eq:Vdef}
\notag V(z)&=V(z;\theta,\theta')\\
&=\begin{cases}
\mathcal{T}(e^{i\phi}z)-\sum_{j=1}^{K_1}\frac{\alpha_1(z^j e^{-iju} +z^{-j} e^{iju})}{2j}-\sum_{j=1}^{K_2}\frac{\alpha_2(z^j e^{-iju}+z^{-j} e^{iju})}{2j}, & \theta-\theta'\in A_1\\
\mathcal{T}(e^{i\phi}z)-\sum_{j=1}^{K_1}\frac{\alpha_1(z^j e^{iju} +z^{-j} e^{-iju})}{2j}-\sum_{j=1}^{K_2}\frac{\alpha_2(z^j e^{iju}+z^{-j} e^{-iju})}{2j}, & \theta-\theta'\in A_2
\end{cases}\\
\notag &=:\sum_{j\in \Z} V_j z^j,
\end{align}

\noindent where we wrote 

$$
A_1=(0,\pi)\cup (-2\pi,-\pi) \qquad \text{and} \qquad A_2=(-\pi,0)\cup (\pi,2\pi).
$$

\noindent Let us also write 

$$
\widetilde{\beta}_1=\widetilde{\beta}_1(\theta,\theta')=\begin{cases}
\beta_1, & \theta-\theta'\in A_1\\
\beta_2, & \theta-\theta'\in A_2
\end{cases} \qquad \text{and} \qquad \widetilde{\beta}_2=\widetilde{\beta}_2(\theta,\theta')=\begin{cases}
\beta_2, & \theta-\theta'\in A_1\\
\beta_1, & \theta-\theta'\in A_2
\end{cases} 
$$

\noindent and finally $f:\T\to [0,\infty)$

\begin{equation}\label{eq:fdef}
f(z)=f(z;\theta,\theta')=e^{V(z)}|z-e^{iu}|^{\widetilde{\beta}_1}|z-e^{-iu}|^{\widetilde{\beta}_2}.
\end{equation}
\end{definition}

\noindent It is a simple calculation to check that 

\begin{align*}
f(z)&=e^{\mathcal{T}(e^{i\phi}z)-\sum_{j=1}^{K_1}\frac{\alpha_1}{2j}((e^{i\phi}z)^j e^{-ij\theta}+(e^{i\phi}z)^{-j}e^{ij\theta})-\sum_{j=1}^{K_2}\frac{\alpha_2}{2j}((e^{i\phi}z)^j e^{-ij\theta}+(e^{i\phi}z)^{-j}e^{ij\theta})}\\
&\qquad \times |e^{i\phi}z-e^{i\theta}|^{\beta_1}|e^{i\phi}z-e^{i\theta'}|^{\beta_2},
\end{align*}

\noindent so by our discussion leading to Definition \ref{def:f},

$$
\E e^{\mathrm{Tr}\mathcal{T}(U_N)+\alpha_1 X_{N,K_1}(e^{i\theta})+\alpha_2 X_{N,K_2}(e^{i\theta})+\beta_1 X_N(e^{i\theta})+\beta_2 X_N(e^{i\theta'})}=D_{N-1}(f).
$$

\noindent From now on, we'll focus on this symbol and a deformation of it. The next step in the argument is to encode the polynomials \eqref{eq:gramdet} associated to the symbol \eqref{eq:fdef} into a Riemann-Hilbert problem. Let us first define the object that will turn out to be the unique solution to a suitable Riemann-Hilbert problem.

\begin{definition}\label{def:Ydef}
For $|z|\neq 1$, define $Y(z)=Y(z,j;f)$ by 

\begin{equation}\label{eq:Ydef}
Y(z)=\begin{pmatrix}
\frac{1}{\chi_j(f)}p_j(z;f) & \frac{1}{\chi_j(f)}\oint_{\T}\frac{p_j(w;f)}{w-z}\frac{f(w)dw}{2\pi i w^j}\\
-\chi_{j-1}(f)z^{j-1}\overline{p}_{j-1}(z^{-1};f) & -\chi_{j-1}(f)\oint_{\T}\frac{\overline{p}_{j-1}(w^{-1};f)}{w-z}\frac{f(w)dw}{2\pi i w}
\end{pmatrix},
\end{equation}

\noindent where $\overline{p}$ denotes the polynomial obtained by complex conjugating the coefficients of $p$: $\overline{p}(z)=\overline{p(\overline{z})}$.

\end{definition}

The fundamental realization of Fokas, Its, and Kitaev \cite{fik} was that $Y$ can be recovered as a solution of a Riemann-Hilbert problem. 

\begin{proposition}[Fokas, Its, and Kitaev]\label{pr:Yrhp}
$Y(z)=Y(z,j;f)$ is the unique solution to the following Riemann-Hilbert problem

\begin{itemize}[leftmargin=0.5cm]
\item[1.] $Y:\C\setminus \T\to \C^{2\times 2}$ is analytic.
\item[2.] $Y$ has continuous boundary values on $\T$. If we denote by $Y_+$ the limit from inside the circle, and $Y_-$ the limit from outside, then 

\begin{equation}\label{eq:Yjump}
Y_+(z)=Y_-(z)\begin{pmatrix}
1 & z^{-j} f(z)\\
0 & 1
\end{pmatrix}.
\end{equation}

\item[3.] As $z\to\infty$, 

\begin{equation}\label{eq:Yasy}
Y(z)=(I+\mathcal{O}(z^{-1}))z^{j\sigma_3},
\end{equation}

\noindent where 

\begin{equation*}
\sigma_3=\begin{pmatrix}
1 & 0\\
0 & -1
\end{pmatrix}\qquad {and} \qquad z^{j\sigma_3}=\begin{pmatrix}
z^j & 0\\
0 & z^{-j}
\end{pmatrix}.
\end{equation*}
\end{itemize}
\end{proposition}

\begin{remark}
The notation $\mathcal{O}$ above refers to behavior when varying $z$. The implicit constants can increase with $j,N,\beta$, etc.
\end{remark}

To be precise, Fokas, Its, and Kitaev did not consider this specific RHP, but the realization about the connection between orthogonal polynomials and RHPs is due to them. Using Liouville's theorem, it's fairly easy to check that if a solution $Y$ exists, then $\det Y=1$ and further that the solution is unique. That the solution is given by $Y(z,j;f)$ requires expanding the Cauchy kernel and using orthogonality. Continuity of the boundary values follows essentially from the fact that the Hilbert transform preserves H\"older continuity. Details in slightly different cases can be found in e.g. Deift's book \cite{deift}, whence we omit further details.

When studying this type of problems, it's typical that using suitable differential identities, one can reduce the problem of understanding the asymptotics of all $Y(z,j;f)$ (which is needed to understand the asymptotics of $\chi_j$), to solving only one Riemann-Hilbert problem -- that for $Y(z,N,f_t)$ where $f_t$ is a deformation of $f$. We'll make use of the fact that for $V=0$ (actually for $V$ smooth enough and independent of $N$) the asymptotics were studied in \cite{ck}. We'll build on their results using \cite[Proposition 3.3]{dik2}. We define an interpolation between $V$ and $0$ - which translates into one between $f$ and $|z-e^{iu}|^{\widetilde{\beta}_1} |z-e^{-iu}|^{\widetilde{\beta}_2}$.

\begin{proposition}[Deift, Its, and Krasovsky]\label{pr:di}
For $t\in[0,1]$ and $z\in \T$, let  

\begin{equation}
V_t(z)=\log\left(1-t+te^{V(z)}\right), 
\end{equation}

\noindent and 

\begin{equation}
f_t(z)=e^{V_t(z)}e^{-V(z)}f(z)=e^{V_t(z)}|z-e^{iu}|^{\widetilde{\beta}_1} |z-e^{-iu}|^{\widetilde{\beta}_2},
\end{equation}

\noindent so that $V_0(z)=0$, $V_1(z)=V(z)$, $f_0(z)=|z-e^{iu}|^{\widetilde{\beta}_1} |z-e^{-iu}|^{\widetilde{\beta}_2}$, and $f_1(z)=f(z)$. 

\vspace{0.3cm}

Now in \eqref{eq:Ydef}, set $j=N$, $f=f_t$, and write $Y(z,t)=Y(z,N;f_t)$. Then 

\begin{align}\label{eq:di}
\frac{\partial}{\partial t}&\log D_{N-1}(f_t)=\frac{1}{2\pi i }\oint_\T z^{-N}\left(Y_{11}(z,t)\frac{\partial Y_{21}(z,t)}{\partial z}-Y_{21}(z,t)\frac{\partial Y_{11}(z,t)}{\partial z}\right)\frac{\partial f_t(z)}{\partial t}dz.
\end{align}
\end{proposition}

\begin{remark}
The differentiability of $\log D_{N-1}(f_t)$ can be seen for example from the representation in terms of $\chi_j$ and the determinantal representation of $\chi_j$. Also note that $Y_{11}$ and $Y_{21}$ are polynomials, so they don't have any jump across $\T$ and e.g. the quantity $\partial_z Y_{11}$ is well defined on $\T$.
\end{remark}

To analyze $D_{N-1}(f)$, we thus need a good understanding of what $Y(z,N,f_t)$ looks like on $\T$. The idea is to analyze this function in the $N\to\infty$ limit by solving the Riemann-Hilbert problem asymptotically. The way this is usually done is by transforming it into a problem where the jump matrix is asymptotically close to the identity matrix (when $N\to\infty$) and as $z\to \infty$, the sought function also converges to the identity matrix. In such a case, the problem can be expressed in terms of a suitable singular integral equation which can then be solved in terms of a Neumann series. Before going into transforming the problem, we'll have to discuss analytic continuation of $f_t(z)$. We need this as part of the transformation procedure involves deforming the jump contour so we will need to know what $f$ looks like off of $\T$.

\section{\texorpdfstring{Analytic continuation of $f_t$}{Analytic continuation of f}}\label{sec:cont}

Note that $V$ is a Laurent polynomial, so it continues analytically to $\C\setminus \lbrace 0\rbrace$. Thus $e^{V_t(z)}=1-t+te^{V(z)}$ also continues analytically into $\C\setminus\lbrace 0\rbrace$. Analytically continuing $f_t$ thus becomes an issue of analytically continuing the quantities $|z-e^{\pm iu}|^{\widetilde{\beta}_1/\widetilde{\beta}_2}$. To do this, define

\begin{equation}\label{eq:szegod}
\mathcal{D}(z)=\exp\left(\frac{1}{2\pi i}\int_{\T}\frac{\log f(\xi)}{\xi-z}d\xi\right).
\end{equation}

\noindent This is analytic inside and outside of $\T$. As in  \cite[(4.9) and (4.10)]{dik1}, one can write for $|z|<1$

\begin{equation}\label{eq:din}
\mathcal{D}(z)=e^{\sum_{j=0}^\infty V_j z^j}\frac{(z-e^{iu})^{\widetilde{\beta}_1/2}}{e^{i\widetilde{\beta}_1 u/2}e^{i\widetilde{\beta}_1 \pi/2}}\frac{(z-e^{-iu})^{\widetilde{\beta}_2/2}}{e^{-i\widetilde{\beta}_2 u/2}e^{i\widetilde{\beta}_2\pi/2}}=:\mathcal{D}_{in}(z)
\end{equation}

\noindent and for $|z|>1$

\begin{equation}\label{eq:dout}
\mathcal{D}(z)^{-1}=e^{\sum_{j=-\infty}^{-1}V_j z^j}\frac{(z-e^{iu})^{\widetilde{\beta}_1/2}}{z^{\frac{\widetilde{\beta}_1}{2}}}\frac{(z-e^{-iu})^{\widetilde{\beta}_2/2}}{z^{\frac{\widetilde{\beta}_2}{2}}}=:\mathcal{D}_{out}(z)^{-1},
\end{equation}

\noindent where the branches are fixed by the following conditions: in both cases the branch of $(z-e^{iu})^{\widetilde{\beta}_1/2}$ is fixed by the condition that $\mathrm{arg}(z-e^{iu})=2\pi$ on the line going from $e^{iu}$ to the right parallel to the real axis. The branch cut is the line $e^{iu}\times[1,\infty)$. For $(z-e^{-iu})^{\widetilde{\beta}_2/2}$, the situation is identical. In the latter case, the cut of $z^{\widetilde{\beta}_1/2}$ is the line $e^{iu}\times [0,\infty)$, and one has $\mathrm{arg}(z)\in(u,2\pi+u)$. For the $-u$-term one has the analogous conditions. We point out that these functions continue analytically (apart from the branch cuts) past $\T$.

Note that with this construction, if we denote by $+$ the boundary value as a limit from inside the unit disk and by $-$ the limit from outside of the disk, we see that for $z\in \T$

\begin{equation}\label{eq:fdrep}
f(z)=\mathcal{D}_{in,+}(z)\mathcal{D}_{out,-}(z)^{-1},
\end{equation}

\noindent and the function on the right hand side continues analytically to $\C$ with branch cuts along $e^{\pm iu}[0,\infty)$. Call this function $f(z)$.

We also introduce a $t$-dependent version of $\mathcal{D}$. For $|z|<1$, let 

\begin{equation}\label{eq:dint}
\mathcal{D}_t(z)=e^{\sum_{j=0}^\infty V_j(t) z^j}\frac{(z-e^{iu})^{\widetilde{\beta}_1/2}}{e^{i\widetilde{\beta}_1 u/2}e^{i\widetilde{\beta}_1 \pi/2}}\frac{(z-e^{-iu})^{\widetilde{\beta}_2/2}}{e^{-i\widetilde{\beta}_2 u/2}e^{i\widetilde{\beta}_2\pi/2}}=:\mathcal{D}_{t,in}(z)
\end{equation}

\noindent and for $|z|>1$

\begin{equation}\label{eq:doutt}
\mathcal{D}_t(z)^{-1}=e^{\sum_{j=-\infty}^{-1}V_j(t) z^j}\frac{(z-e^{iu})^{\widetilde{\beta}_1/2}}{z^{\frac{\widetilde{\beta}_1}{2}}}\frac{(z-e^{-iu})^{\widetilde{\beta}_2/2}}{z^{\frac{\widetilde{\beta}_2}{2}}}=:\mathcal{D}_{t,out}(z)^{-1},
\end{equation}

\noindent with the same conventions for the branch cuts as before as well as the notation $V_j(t)$ denoting the Fourier coefficients of $V_t$.

To see that this is well defined, we'll want to also continue $V_t(z)$ analytically to a neighborhood of $\T$. The size of this neighborhood will play a significant role in our analysis so we emphasize it's definition. 

\begin{definition}\label{def:kappadef}

Let $\varepsilon>0$ be small but fixed $($shortly taken to possibly depend on the compact set where $(\mathcal{T}_j)_{j=1}^M$ from Proposition \ref{pr:testimate} lie along with on the quantities $\alpha_1,\alpha_2$ from Proposition \ref{pr:testimate}$)$ and $K_2,M$ as in Proposition \ref{pr:testimate}. Then define

\begin{equation}\label{eq:kappadef}
\kappa=\varepsilon [\max(K_2,M^2)]^{-1}.
\end{equation}
\end{definition}

We point out that $\varepsilon$ here is independent of $\epsilon$ which appeared in Section \ref{sec:defs}. Let us show that $V_t$ is analytic in a $3\kappa$-neighborhood of $\T$ ($3$ here could be any constant, but later on we will define new contours and objects strictly inside of this set and we find it to be notationally convenient to have $3$ here).

\begin{lemma}\label{le:vtholo}
For small enough $\varepsilon>0$, which depends only on the compact subset of $\C^M$ where $(\mathcal{T}_j)_{j=1}^M$ lie and $\alpha_1,\alpha_2$,  the function 

\begin{equation*}
z\mapsto V_t(z)=\log(1-t+te^{V(z)})
\end{equation*}

\noindent is analytic in $\lbrace z\in \C: ||z|-1|<3\kappa\rbrace$ for arbitrary $t\in[0,1]$.
\end{lemma}

\begin{proof}
The proof consists of very crude estimates. It's enough to show that $1-t+t e^{V(z)}$ is non-zero for all $t\in[0,1]$ in this domain. To do this, let us look at the derivative of $V$: directly from \eqref{eq:Vdef} -- the definition of $V$, we see that for some positive numbers $C$ and $\widetilde{C}$ only depending on $\alpha_1,\alpha_2$ and the compact set where the coefficients of $\mathcal{T}$ lie,

\begin{equation*}
\sup_{||z|-1|<3\kappa}|V'(z)|\leq C M\sup_{|j|\leq M}|\mathcal{T}_j|\sum_{k=1}^M(1+3\kappa)^{j} + C\sum_{j=1}^{K_2}\left(1+3\kappa\right)^j\leq \widetilde{C}[\max(K_2,M^2)].
\end{equation*}

\noindent Combining this with the fact that for $z\in \T$, $\mathrm{Im}(V(z))=0$, we see that

\begin{equation*}
\sup_{||z|-1|<3\kappa}|\mathrm{Im}(V(z))|\leq \widetilde{C}\varepsilon.
\end{equation*}

\noindent Thus for a small enough $\varepsilon>0$, the real part of $e^{V(z)}$ stays positive and the function $1-t+t e^{V(z)}$ can't have zeroes. Thus $V_t(z)$ is analytic in the relevant domain.
\end{proof}

From now on, we'll consider $\varepsilon>0$ small enough (but fixed) so that the above analyticity condition is satisfied. To conclude this section, we record a result we'll need later on when estimating the jump matrices of our deformed problem. 

\begin{lemma}\label{le:Eest}

As $N\to\infty$, 

\begin{equation*}
\sup_{||z|-1|\leq 3\kappa}\max\left(|\mathcal{D}_{t,in}(z)\mathcal{D}_{t,out}(z)|,|\mathcal{D}_{t,in}(z)\mathcal{D}_{t,out}(z)|^{-1}\right)=\mathcal{O}(1)
\end{equation*}

\noindent uniformly in $t\in[0,1]$ and $\theta,\theta'\in[0,2\pi]$ $($though not necessarily in $M)$.
\end{lemma}

\begin{proof}
Let us begin by noting that for $||z|-1|\leq 3\kappa$,

$$
\left(\mathcal{D}_{t,in}(z)\mathcal{D}_{t,out}(z)\right)^{\pm 1}=\mathcal{O}(1) e^{\pm V_0(t)\pm \sum_{j=1}^\infty (V_j(t)z^j-V_{-j}(t)z^{-j})}
$$

\noindent where $\mathcal{O}(1)$ is uniform in $z$ and $t$. Now the boundedness of $\mathrm{Re} \sum_{j=1}^\infty (V_j(t)z^j-V_{-j}(t)z^{-j})$ can be proven with straightforward modifications of \cite[proof of Lemma 5.6]{abb}. Thus it remains to analyze 

$$
V_0(t)=\int_0^{2\pi}\log\left(1-t+te^{V(e^{i\phi})}\right)\frac{d\phi}{2\pi}.
$$

\noindent To do this, let us first note that from the definition of $V$ \eqref{eq:Vdef} and Lemma \ref{le:logsum}, we see that 

\begin{equation}\label{eq:Vest}
V(e^{i\phi})=-\alpha_1\min(\log^+ d(\phi,\pm u)^{-1},\log K_1)-\alpha_2\min(\log^+ d(\phi,\pm u)^{-1},\log K_2)+\mathcal{O}(1),
\end{equation}

\noindent where $\mathcal{O}(1)$ is uniform in $u,\phi,K_1$, and $K_2$, and the sign of $u$ depends on $\theta,\theta'$ as in \eqref{eq:Vdef}. Now using the elementary inequality 

$$
\log(1+e^x)\leq 1+|x|
$$

\noindent valid for all $x\in \R$, we find (noting that as $V$ is real on $\T$, so that $e^V$ is non-negative, implying that $1-t+te^V\leq 1+e^V$)

$$
V_0(t)\leq \int_0^{2\pi}(1+|V(e^{i\phi})|)\frac{d\phi}{2\pi}=\mathcal{O}(1),
$$

\noindent where the last estimate is uniform in everything relevant and follows from \eqref{eq:Vest} easily. For a lower bound, note that we can write e.g. for $t\in[0,1/2]$, $\log(1-t+te^V)\geq \log(1-t)\geq -\log 2$ and for $t\in[1/2,1]$, $\log(1-t+te^V)\geq \log t+V\geq -\log 2-|V|$ and again from \eqref{eq:Vest}, one concludes that $V_0(t)\geq \mathcal{O}(1)$ with the required uniformity.
\end{proof}

We will now begin transforming our Riemann-Hilbert problem into a form where it can be solved asymptotically.

\section{Transforming the Riemann-Hilbert problem}\label{sec:trans}

Our first transformation will be a rather trivial one -- it will normalize the behavior of $Y(z,t)$ as $z\to\infty$. Our next transformation will be less trivial. It involves introducing further jump contours to our problem -- a procedure often called opening lenses. Moreover, we'll actually have to do this in two different ways -- essentially depending on whether $u=\mathcal{O}(\kappa)$ or not. We have some room in this definition and we'll say that when  $u<\kappa/2$, we're in the small $u$ case and when $u \geq \kappa/2$, we'll say that we're in the large $u$ case.

\subsection{The first transformation}

The first transformation one usually performs when analyzing RHPs such as the one $Y$ satisfies is to normalize the behavior of $Y$ at infinity. More precisely, one defines

\begin{equation}\label{eq:Tdef}
T(z,t)=\begin{cases}
Y(z,t), & |z|<1\\
Y(z,t) z^{-N\sigma_3}, & |z|>1
\end{cases}.
\end{equation}

The following result is standard and immediate from Proposition \ref{pr:Yrhp} so we omit the proof.

\begin{lemma}\label{le:Trhp}
For each $t\in[0,1]$, $T(\cdot,t)$ is the unique solution to the following RHP$:$

\begin{itemize}[leftmargin=0.5cm]
\item[1.] $T(\cdot,t):\C\setminus \T\to \C^{2\times 2}$ is analytic.
\item[2.] $T(\cdot,t)$ has continuous boundary values on $\T$ and they satisfy the jump condition

\begin{equation}
T_+(z,t)=T_-(z,t)\begin{pmatrix}
z^N & f_t(z)\\
0 & z^{-N}
\end{pmatrix}, \qquad z\in \T.
\end{equation}

\item[3.] As $z\to\infty$, 

\begin{equation}
T(z,t)=I+\mathcal{O}(z^{-1}).
\end{equation}
\end{itemize}
\end{lemma}

\subsection{The second transformation -- the small \texorpdfstring{$u$}{u} case}

In our second transformation, we'll want to alter our jump contour in such a way that our problem will become close to one which can be exactly solved. To do this in the small $u$ case, we need to consider a suitable neighborhood of the point $1$ (or actually, we consider its complement for now). We'll need to consider separately $T$ close to the point $1$ and far away from it -- the scale of close and far being given by $\kappa$. We'll also need to define a suitable set near $\T$ where we'll make use of the bounds we have for $V$: let

\begin{align}\label{eq:Ltildef}
\widetilde{L}_{{\mathrm{sm}}}&=\lbrace z\in \C: |z-1|>\kappa\rbrace\cap \left\lbrace z\in \C:1-\frac{2}{3}\kappa<|z|< 1+\frac{2}{3}\kappa\right\rbrace,
\end{align}

\noindent  For an illustration of the set $\widetilde{L}_{{\mathrm{sm}}}$, see Figure \ref{fig:L}.

\begin{figure}
\begin{center}
\begin{tikzpicture}[xscale=0.02,yscale=0.02]
\clip (-135,-135) rectangle (170,135); 

\fill [color=black] (0,0) circle (3);

\draw [black,thick,domain=15:345] plot ({80*cos(\x)}, {80*sin(\x)});
\draw [black,thick,domain=12:348] plot ({120*cos(\x)}, {120*sin(\x)});

\fill [color=black] (99,14) circle (3);
\fill [color=black] (99,-14) circle (3);

\node at (0,-100) {\small $\widetilde{L}_{{\mathrm{sm}}}$};
\node at (-10,0) {\small $0$};
\node at (120,15) {\tiny $e^{iu}$};
\node at (120,-15) {\tiny $e^{-iu}$};

\draw [black,thick,domain=55:138] plot ({100+30*cos(\x)}, {30*sin(\x)});

\draw [black,thick,domain=55:138] plot ({100+30*cos(\x)}, {-30*sin(\x)});
   
\end{tikzpicture}
\begin{tikzpicture}[xscale=0.02,yscale=0.02]
\clip (-135,-135) rectangle (170,135);

\draw [black,thick,domain=15:345] plot ({80*cos(\x)}, {80*sin(\x)});
\draw [black,thick,domain=12:348] plot ({120*cos(\x)}, {120*sin(\x)});

\fill [color=black] (99,14) circle (3);
\fill [color=black] (99,-14) circle (3);

\node at (120,15) {\tiny $e^{iu}$};
\node at (120,-11) {\tiny $e^{-iu}$};

\draw[->,thick] (-80,2) -- (-80,-2);
\draw[->,thick] (-120,2) -- (-120,-2);

\node at (-70,0) {\small $+$};
\node at (-90,0) {\small $-$};

\node at (-110,0) {\small $+$};
\node at (-130,0) {\small $-$};

\node at (0,-100) {\small ${L_{{\mathrm{sm}}}}$};

\draw [black, thick] ({100+30*cos(55)}, { 30*sin(55)}) -- (99,14);
\draw [black, thick] ({100+30*cos(138)}, { 30*sin(138)}) -- (99,14);

\draw [black, thick] ({100+30*cos(55)}, {-30*sin(55)}) -- (99,-14);
\draw [black, thick] ({100+30*cos(138)}, {-30*sin(138)}) -- (99,-14);

\end{tikzpicture}
\end{center}
\vspace{-0.3cm}
\caption{Left: A caricature  of the set $\widetilde{L}_{{\mathrm{sm}}}$. Right: a caricature of the set $L_{{\mathrm{sm}}}$ along with the orientation of its boundary contour for the $S$-RHP. }
\label{fig:L}
\end{figure}
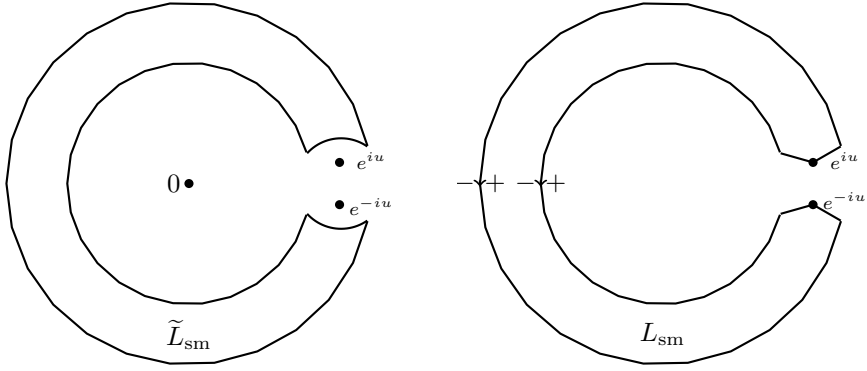

To obtain a suitable jump contour for our deformed RHP, we'll want to connect $\widetilde{L}_{{\mathrm{sm}}}$ to $e^{iu}$ and $e^{-iu}$ in a specific way. We won't be too precise right now, but we take two curves going from $e^{iu}$ to the two closest points on $\overline{\widetilde{L}}_{{\mathrm{sm}}}\cap \lbrace |z|=1\pm \frac{2}{3}\kappa\rbrace$, and similarly for $e^{-iu}$. The contour will turn out to be such that it doesn't intersect the cut of $f$ apart from at the points $e^{\pm i u}$. We add the (interior of the) domain bounded by these curves to $\widetilde{L}_{{\mathrm{sm}}}$ and call the new set $L_{{\mathrm{sm}}}$. Again, for a caricature of the set, see Figure \ref{fig:L}.

We now make our next transformation. Let

\begin{equation}
S_{{\mathrm{sm}}}(z,t)=\begin{cases}
T(z,t), & z\notin L_{{\mathrm{sm}}}\\
T(z,t)\begin{pmatrix}
1 & 0\\
z^{-N}f_t(z)^{-1} & 1
\end{pmatrix}, & z\in L_{{\mathrm{sm}}}\ \mathrm{and} \ |z|>1\\
T(z,t)\begin{pmatrix}
1 & 0\\
-z^{N}f_t(z)^{-1} & 1
\end{pmatrix}, & z\in L_{{\mathrm{sm}}}\ \mathrm{and} \ |z|<1
\end{cases}.
\end{equation}

Again this function can be characterized in terms of a Riemann-Hilbert problem, and we omit the proof. We choose the orientation of $\partial L_{{\mathrm{sm}}}$ in such a way that the side of the origin is always the $+$-side. See Figure \ref{fig:L}. For $\T$ we choose the same orientation as before: the side of the origin is the $+$-side.

\begin{lemma}\label{le:Srhp}
For $t\in[0,1]$, $S_{{\mathrm{sm}}}(\cdot,t)$ is the unique solution to the following problem.

\begin{itemize}[leftmargin=0.5cm]
\item[1.] $S_{{\mathrm{sm}}}(\cdot,t):\C \setminus (\partial L_{{\mathrm{sm}}}\cup \T)\to \C^{2\times 2}$ is analytic.
\item[2.] $S_{\mathrm{sm}}(\cdot,t)$ has continuous boundary values on $[\partial L_{{\mathrm{sm}}}\cup \T]\setminus\lbrace e^{iu},e^{-iu}\rbrace$, which we denote by $S_{\mathrm{sm},\pm}(\cdot,t)$, where $+$ corresponds to the limit from the left and $-$ to the limit from the right. These satisfy the following jump conditions$:$

For $z\in \T\cap L_{{\mathrm{sm}}}$, 

\begin{equation}\label{eq:Sjump1}
S_{{\mathrm{sm}},+}(z,t)=S_{{\mathrm{sm}},-}(z,t)\begin{pmatrix}
0 & f_t(z)\\
-f_t(z)^{-1} & 0
\end{pmatrix}.
\end{equation}

For $z\in \T\setminus (L_{{\mathrm{sm}}}\cup \lbrace e^{iu},e^{-iu}\rbrace)$

\begin{equation}\label{eq:Sjump2}
S_{{\mathrm{sm}},+}(z,t)=S_{{\mathrm{sm}},-}(z,t)\begin{pmatrix}
z^N & f_t(z)\\
0 & z^{-N}
\end{pmatrix}.
\end{equation}

For $z\in \partial L_{{\mathrm{sm}}}\cap\lbrace |z|<1\rbrace$

\begin{equation}\label{eq:Sjump3}
S_{{\mathrm{sm}},+}(z,t)=S_{{\mathrm{sm}},-}(z,t)\begin{pmatrix}
1 & 0\\
z^N f_t(z)^{-1} & 1
\end{pmatrix}
\end{equation}

\noindent and for $z\in \partial L_{{\mathrm{sm}}}\cap \lbrace |z|>1\rbrace$

\begin{equation}\label{eq:Sjump4}
S_{{\mathrm{sm}},+}(z,t)=S_{{\mathrm{sm}},-}(z,t)\begin{pmatrix}
1 & 0\\
z^{-N}f_t(z)^{-1} & 1
\end{pmatrix}.
\end{equation}

\item[3.] As $z\to\infty$, $S_{{\mathrm{sm}}}(z,t)=I+\mathcal{O}(z^{-1})$.

\item[4.] As $z\to e^{i u}$

\begin{equation}\label{eq:Ssingasy}
S_{{\mathrm{sm}}}(z,t)=\begin{cases}
\begin{pmatrix}
\mathcal{O}(|z-e^{iu}|^{-\widetilde{\beta}_1}) & \mathcal{O}(1)\\
\mathcal{O}(|z-e^{iu}|^{-\widetilde{\beta}_1}) & \mathcal{O}(1)
\end{pmatrix}, & z\in L_{{\mathrm{sm}}}\setminus \T\\
\mathcal{O}(1), & z\in \C\setminus (\overline{L}_{{\mathrm{sm}}}\cup \T)
\end{cases}
\end{equation}

\noindent and as $z\to e^{-iu}$

\begin{equation}\label{eq:Ssingasy12}
S_{{\mathrm{sm}}}(z,t)=\begin{cases}
\begin{pmatrix}
\mathcal{O}(|z-e^{-iu}|^{-\widetilde{\beta}_2}) & \mathcal{O}(1)\\
\mathcal{O}(|z-e^{-iu}|^{-\widetilde{\beta}_2}) & \mathcal{O}(1)
\end{pmatrix}, & z\in L_{{\mathrm{sm}}}\setminus \T\\
\mathcal{O}(1), & z\in \C\setminus (\overline{L}_{{\mathrm{sm}}}\cup \T)
\end{cases}.
\end{equation}

\end{itemize}
\end{lemma}

We point out here that for example on $\partial L_{{\mathrm{sm}}}\cap \lbrace |z|=1+ \frac{2}{3}\kappa\rbrace$, the jump matrix is $I+\mathcal{O}(|z|^{-N}f_t(z)^{-1})=I+\mathcal{O}(e^{-cN^\delta})$ for some fixed $c>0$ and we should expect that this part of the jump contour is somehow negligible in our RH-analysis. While most of the lemma is a straightforward consequence of the RHP for $Y$ and $T$, we mention that to see uniqueness of the solution to the problem, one needs to use \eqref{eq:Ssingasy} and \eqref{eq:Ssingasy12}, to see that if $S$ and $\widetilde{S}$ are two solutions, then any possible isolated singularities of $S\widetilde{S}^{-1}$ at $e^{\pm iu}$ are removable.

\subsection{The second transformation -- the large \texorpdfstring{$u$}{u} case}

The problem with the previous transformation is that we would like that on most of $\T$, one would have the jump condition \eqref{eq:Sjump1} instead of \eqref{eq:Sjump2}, but when $u$ grows, the portion of $\T$ where one has the jump \eqref{eq:Sjump2} becomes relevant. The reason to prefer \eqref{eq:Sjump1} is that solving the corresponding approximate RHP is much simpler and we get better estimates. We remedy this by instead of considering a $\kappa$-neighborhood of the point $1$, we consider disjoint $\kappa/4$-neighborhoods of the points $e^{\pm iu}$. We'll also need to rescale our enlargement of $\T$ slightly.  More precisely, we define

\begin{align}\label{eq:Ltildefla}
\widetilde{L}_{{\mathrm{la}}}&=\lbrace z\in \C: d(z,e^{\pm iu})>\kappa/4\rbrace\cap \left\lbrace z\in \C:1-\frac{1}{6}\kappa<|z|< 1+\frac{1}{6}\kappa\right\rbrace,
\end{align}

\noindent For an illustration of the set $\widetilde{L}_{{\mathrm{la}}}$, see Figure \ref{fig:L2}.

\begin{figure}
\begin{center}
\begin{tikzpicture}[xscale=0.02,yscale=0.02]
\clip (-135,-135) rectangle (170,135); 

\fill [color=black] (0,0) circle (3);

\draw [black,thick,domain=45:315] plot ({80*cos(\x)}, {80*sin(\x)});
\draw [black,thick,domain=42:318] plot ({120*cos(\x)}, {120*sin(\x)});

\draw [black,thick,domain=-15:15] plot ({80*cos(\x)}, {80*sin(\x)});
\draw [black,thick,domain=-19:19] plot ({120*cos(\x)}, {120*sin(\x)});

\fill [color=black] (86,50) circle (3);
\fill [color=black] (86,-50) circle (3);

\node at (0,-100) {\small $\widetilde{L}_{{\mathrm{la}}}$};
\node at (-10,0) {\small $0$};
\node at (90,65) {\small $e^{iu}$};
\node at (90,-65) {\small $e^{-iu}$};

\draw [black,thick,domain=82:168] plot ({86+30*cos(\x)}, {50+ 30*sin(\x)});
\draw [black,thick,domain=192:278] plot ({86+30*cos(\x)}, {-50+ 30*sin(\x)});

\draw [black,thick,domain=252:338] plot ({86+30*cos(\x)}, {50+ 30*sin(\x)});
\draw [black,thick,domain=22:108] plot ({86+30*cos(\x)}, {-50+ 30*sin(\x)});

\end{tikzpicture}
\begin{tikzpicture}[xscale=0.02,yscale=0.02]
\clip (-135,-135) rectangle (170,135);

\draw [black,thick,domain=45:315] plot ({80*cos(\x)}, {80*sin(\x)});
\draw [black,thick,domain=42:318] plot ({120*cos(\x)}, {120*sin(\x)});

\draw [black,thick,domain=-15:15] plot ({80*cos(\x)}, {80*sin(\x)});
\draw [black,thick,domain=-19:19] plot ({120*cos(\x)}, {120*sin(\x)});

\fill [color=black] (86,50) circle (3);
\fill [color=black] (86,-50) circle (3);
\draw[->,thick] (-80,2) -- (-80,-2);
\draw[->,thick] (-120,2) -- (-120,-2);

\draw[->,thick] (80,-2) -- (80,2);
\draw[->,thick] (120,-2) -- (120,2);

\node at (-70,0) {\small $+$};
\node at (-90,0) {\small $-$};

\node at (-110,0) {\small $+$};
\node at (-130,0) {\small $-$};

\node at (70,0) {\small $+$};
\node at (90,0) {\small $-$};

\node at (110,0) {\small $+$};
\node at (130,0) {\small $-$};

\node at (0,-100) {\small ${L_{{\mathrm{la}}}}$};
\node at (110,65) {\small $e^{iu}$};
\node at (110,-65) {\small $e^{-iu}$};

\draw [black, thick] ({86+30*cos(193)}, {-50+ 30*sin(193)}) -- (86,-50);
\draw [black, thick] ({86+30*cos(275)}, {-50+ 30*sin(275)}) -- (86,-50);

\draw [black, thick] ({86+30*cos(85)}, {50+ 30*sin(85)}) -- (86,50);
\draw [black, thick] ({86+30*cos(167)}, {50+ 30*sin(167)}) -- (86,50);

\draw [black, thick] ({86+30*cos(22)}, {-50+ 30*sin(22)}) -- (86,-50);
\draw [black, thick] ({86+30*cos(108)}, {-50+ 30*sin(108)}) -- (86,-50);

\draw [black, thick] ({86+30*cos(252)}, {50+ 30*sin(252)}) -- (86,50);
\draw [black, thick] ({86+30*cos(338)}, {50+ 30*sin(338)}) -- (86,50);

\end{tikzpicture}
\end{center}
\vspace{-0.3cm}
\caption{Left: A caricature  of the set $\widetilde{L}_{{\mathrm{la}}}$. Right: a caricature of the set $L_{{\mathrm{la}}}$ along with the orientation of its boundary contour for the $S$-RHP.}
\label{fig:L2}
\end{figure}

Again, to obtain a suitable jump contour for our deformed RHP, we connect $\widetilde{L}_{{\mathrm{la}}}$ to $e^{iu}$ and $e^{-iu}$ in a suitable way. We'll be more precise later, but we take four curves going from $e^{iu}$ to the four closest points on $\overline{\widetilde{L}}_{{\mathrm{la}}}\cap \lbrace |z|=1\pm \kappa/6\rbrace$, and similarly for $e^{-iu}$. Again the contour won't intersect the cut of $f$ apart from at the points $e^{\pm i u}$. We add the (interior of the) domain bounded by these curves to $\widetilde{L}_{{\mathrm{la}}}$ and call the new set $L_{{\mathrm{la}}}$ (see Figure \ref{fig:L2}).

In the large $u$ case the next transformation is then

\begin{equation}
S_{{\mathrm{la}}}(z,t)=\begin{cases}
T(z,t), & z\notin L_{{\mathrm{la}}}\\
T(z,t)\begin{pmatrix}
1 & 0\\
z^{-N}f_t(z)^{-1} & 1
\end{pmatrix}, & z\in L_{{\mathrm{la}}}\ \mathrm{and} \ |z|>1\\
T(z,t)\begin{pmatrix}
1 & 0\\
-z^{N}f_t(z)^{-1} & 1
\end{pmatrix}, & z\in L_{{\mathrm{la}}}\ \mathrm{and} \ |z|<1
\end{cases}.
\end{equation}

For the RHP,  we again choose the orientation of $\partial L_{{\mathrm{la}}}$ in such a way that the side of the origin is always the $+$-side (Figure \ref{fig:L2}) and for $\T$  we choose the origin to be on the $+$-side. The relevant RHP is now the following.

\begin{lemma}\label{le:Srhpla}
For $t\in[0,1]$, $S_{{\mathrm{la}}}(\cdot,t)$ is the unique solution to the following problem.

\begin{itemize}[leftmargin=0.5cm]
\item[1.] $S_{{\mathrm{la}}}(\cdot,t):\C \setminus (\partial L_{{\mathrm{la}}}\cup \T)\to \C^{2\times 2}$ is analytic.
\item[2.] For $z\in \T\cap L_{{\mathrm{la}}}$, 

\begin{equation}\label{eq:Sjump1la}
S_{{\mathrm{la}},+}(z,t)=S_{{\mathrm{la}},-}(z,t)\begin{pmatrix}
0 & f_t(z)\\
-f_t(z)^{-1} & 0
\end{pmatrix}.
\end{equation}

For $z\in \T\setminus (L_{{\mathrm{la}}}\cup \lbrace e^{iu},e^{-iu}\rbrace)$

\begin{equation}\label{eq:Sjump2la}
S_{{\mathrm{la}},+}(z,t)=S_{{\mathrm{la}},-}(z,t)\begin{pmatrix}
z^N & f_t(z)\\
0 & z^{-N}
\end{pmatrix}.
\end{equation}

For $z\in \partial L_{{\mathrm{la}}}\cap\lbrace |z|<1\rbrace$

\begin{equation}\label{eq:Sjump3la}
S_{{\mathrm{la}},+}(z,t)=S_{{\mathrm{la}},-}(z,t)\begin{pmatrix}
1 & 0\\
z^N f_t(z)^{-1} & 1
\end{pmatrix}
\end{equation}

\noindent and for $z\in \partial L_{{\mathrm{la}}}\cap \lbrace |z|>1\rbrace$

\begin{equation}\label{eq:Sjump4la}
S_{{\mathrm{la}},+}(z,t)=S_{{\mathrm{la}},-}(z,t)\begin{pmatrix}
1 & 0\\
z^{-N}f_t(z)^{-1} & 1
\end{pmatrix}.
\end{equation}

\item[3.] As $z\to\infty$, $S_{{\mathrm{la}}}(z,t)=I+\mathcal{O}(z^{-1})$.

\item[4.] As $z\to e^{iu}$,

\begin{equation}\label{eq:Ssingasyla}
S_{{\mathrm{sm}}}(z,t)=\begin{cases}
\begin{pmatrix}
\mathcal{O}(|z-e^{iu}|^{-\widetilde{\beta}_1}) & \mathcal{O}(1)\\
\mathcal{O}(|z-e^{iu}|^{-\widetilde{\beta}_1}) & \mathcal{O}(1)
\end{pmatrix}, & z\in L_{{\mathrm{sm}}}\setminus \T\\
\mathcal{O}(1), & z\in \C\setminus (\overline{L}_{{\mathrm{sm}}}\cup \T)
\end{cases}
\end{equation}

\noindent and as $z\to e^{-iu}$

\begin{equation}\label{eq:Ssingasyla12}
S_{{\mathrm{sm}}}(z,t)=\begin{cases}
\begin{pmatrix}
\mathcal{O}(|z-e^{-iu}|^{-\widetilde{\beta}_2}) & \mathcal{O}(1)\\
\mathcal{O}(|z-e^{-iu}|^{-\widetilde{\beta}_2}) & \mathcal{O}(1)
\end{pmatrix}, & z\in L_{{\mathrm{sm}}}\setminus \T\\
\mathcal{O}(1), & z\in \C\setminus (\overline{L}_{{\mathrm{sm}}}\cup \T)
\end{cases}.
\end{equation}

\end{itemize}
\end{lemma}

The next step is to find approximate solutions to the RHPs satisfied by $S_{{\mathrm{sm}}}$ and $S_{{\mathrm{la}}}$. These approximate solutions are often called parametrices.

\section{Parametrices}\label{sec:param}

We will search for different approximate solutions in different domains. When we are not too close to the singularities, we'll simply approximate our RHPs by the jump condition \eqref{eq:Sjump1}. The solution to this RHP is known as the global parametrix. On the other hand, when we are very close to the singularities, we'll focus on solving an approximate problem here -- the solution to this problem is called the local parametrix. As the jump contours are different in the small $u$ and large $u$ cases, we'll have to treat them separately when searching for the local parametrix.

We begin with defining the global parametrix.

\subsection{The global parametrix}

Define the function (see \eqref{eq:dint} and \eqref{eq:doutt})

\begin{equation}\label{eq:globaldef}
\mathcal{N}(z,t)=\begin{cases}
\mathcal{D}_{t,in}(z)^{\sigma_3}\begin{pmatrix}
0 & 1\\
-1 & 0
\end{pmatrix}, & |z|<1\\
\mathcal{D}_{t,out}(z)^{\sigma_3}, & |z|>1
\end{cases}.
\end{equation}

Again, it is standard that $\mathcal{N}$ satisfies a Riemann-Hilbert problem. We leave the proof of the following lemma to the reader. 

\begin{lemma}\label{le:globalrhp}

$\mathcal{N}(\cdot,t)$ solves the following problem.

\begin{itemize}[leftmargin=0.5cm]
\item[1.] $\mathcal{N}(\cdot,t):\C\setminus \T\to \C^{2\times 2}$ is analytic.
\item[2.] For $z\in \T\setminus \lbrace e^{iu},e^{-iu}\rbrace$,

\begin{equation}\label{eq:globaljump}
\mathcal{N}_+(z,t)=\mathcal{N}_-(z,t)\begin{pmatrix}
0 & f_t(z)\\
-f_t(z)^{-1} & 0
\end{pmatrix}.
\end{equation}

\item[3.] As $z\to\infty$

\begin{equation}\label{eq:globalasy}
\mathcal{N}(z,t)=I+\mathcal{O}(z^{-1}).
\end{equation}
\end{itemize}
\end{lemma}

Note that the jump condition of $\mathcal{N}$ is the same as \eqref{eq:Sjump1}.

\subsection{The local parametrix in the small \texorpdfstring{$u$}{u} case}

Our local parametrix will rely heavily on work in \cite{ck}. It will be built using functions satisfying certain model Riemann-Hilbert problems. For a detailed analysis, we refer to \cite{ck}, but we review the definitions and some properties of them in Appendix \ref{app:A}. 

Let us first introduce a change of coordinates that will zoom into a neighborhood of the singularities:

\begin{equation}\label{eq:zetadef}
\zeta(z):=\frac{1}{u}\log z,
\end{equation}

\noindent where we take the principal branch of the logarithm. Note that $\zeta$ maps the arc $\lbrace e^{is},|s|\leq u\rbrace$ into $[-i,i]$. 

We can now be more concrete about how we choose $\partial L_{{\mathrm{sm}}}$ near $e^{\pm iu}$. We choose it so that under $\zeta$ it is mapped onto parts of the rays $i+e^{i\frac{\pi}{2}\pm i\frac{\pi}{4}}\times \R_+$ and $-i+e^{-i\frac{\pi}{2}\pm i\frac{\pi}{4}}\times \R_+$. As $\zeta$ is conformal, we see with this choice that e.g. $\partial L_{{\mathrm{sm}}}$ does not intersect the cut of $f$ except at $e^{\pm iu}$.

We'll also introduce an analytic function whose role is to ensure that near the boundary of a $\kappa$-neighborhood of the point $1$, our local parametrix looks like the global one. Let 

\begin{equation}\label{eq:Edef}
E(z,t):=\begin{pmatrix}
0 & 1\\
1& 0
\end{pmatrix}\left[\mathcal{D}_{t,in}(z)\mathcal{D}_{t,out}(z)\right]^{-\frac{1}{2}\sigma_3}\widehat{P}^{(\infty)}(\zeta(z))^{-1},
\end{equation}

\noindent where $\widehat{P}^{(\infty)}$ is defined in \eqref{eq:Phidef} and for $\mathcal{D}_{t,in/out}$ see \eqref{eq:dint} and \eqref{eq:doutt}. Moreover, the branch of the root is chosen so that the cuts of $[\mathcal{D}_{t,in}(z)\mathcal{D}_{t,out}(z)]^{-1/2}$ are $e^{\pm iu}\times(0,\infty)$. Using the definition of $\mathcal{D}_{t,in/out}$ and $\widehat{P}^{(\infty)}$ one can easily check that this function is analytic in a $\kappa$-neighborhood of $1$. We omit the proof.

In addition to the function $\Phi$ from Appendix \ref{sec:phi}, we'll also need the following function to build our local parametrix: for $z\notin e^{\pm iu}\times[0,\infty)$ and $z\notin(-\infty,0]$ let

\begin{equation}\label{eq:Wdef}
W(z,t):=\begin{cases}
-z^{\frac{N}{2}\sigma_3}f_t(z)^{-\frac{\sigma_3}{2}}\sigma_3, & |z|<1\\
z^{\frac{N}{2}\sigma_3}f_t(z)^{\frac{\sigma_3}{2}}\begin{pmatrix}
0 & 1\\
1 & 0
\end{pmatrix},& |z|>1
\end{cases}
\end{equation}
where $z^{N/2}=e^{\frac{N}{2}\log z}$ where the branch of the log is the principal one, and the branch of $f_t(z)^{\pm 1/2}$ is chosen so that the cuts are on $e^{\pm iu}(0,\infty)$.

Finally we define our local parametrix:

\begin{equation}\label{eq:localdef}
P(z,t)=E(z,t)\Phi\left(\zeta(z),-2iNu\right)W(z,t),
\end{equation}

\noindent where as mentioned, $\Phi$ is defined in Appendix \ref{sec:phi}.

Using these definitions, and our definition of the set $L_{\mathrm{sm}}$, one can check that in a $\kappa$-neighborhood of $1$, $P$ has the same jumps as $S_{{\mathrm{sm}}}$ (namely \eqref{eq:Sjump2}, \eqref{eq:Sjump3}, and \eqref{eq:Sjump4}). Moreover, they have same behavior near $e^{\pm iu}$ -- \eqref{eq:Ssingasy}. We omit the details.

\subsection{The local parametrix in the large \texorpdfstring{$u$}{u} case}

The construction is now similar to that in the small $u$ case, but we essentially treat the singularities separately. We again use the same change of coordinates $z\mapsto \zeta(z)$. The role of $\widehat{P}^{(\infty)}$ is played by the functions

\begin{equation*}
\Omega_1(z)=\begin{cases}
e^{i\frac{\widetilde{\beta}_1}{4}\pi\sigma_3}, & \mathrm{Im}[\zeta(z)]>1\\
e^{-i\frac{\widetilde{\beta}_1}{4}\pi\sigma_3}, & \mathrm{Im}[\zeta(z)]<1
\end{cases} \qquad \mathrm{and} \qquad \Omega_2(z)=\begin{cases}
e^{i\frac{\widetilde{\beta}_2}{4}\pi\sigma_3}, & \mathrm{Im}[\zeta(z)]>-1\\
e^{-i\frac{\widetilde{\beta}_2}{4}\pi\sigma_3}, & \mathrm{Im}[\zeta(z)]<-1
\end{cases}.
\end{equation*}

Note that e.g. $\mathrm{Im}[\zeta(z)]>\pm 1$ is equivalent to $\mathrm{arg}(z)>\pm u$. The role of the function $E$ is now played by the functions (for $j=1,2$)

\begin{equation*}
\widetilde{E}_j(z,t)=\begin{pmatrix}
0 & 1\\
1 & 0
\end{pmatrix}(\mathcal{D}_{t,in}(z)\mathcal{D}_{t,out}(z))^{-\frac{\sigma_3}{2}}\Omega_j(z)^{-1}e^{-\frac{Nu i}{2}\sigma_3},
\end{equation*}
where the branch of the root is chosen as for \eqref{eq:Edef}.

In reference to Appendix \ref{app:Am}, we define the function $\widetilde{M}_j(\lambda)$ to be the function $M(\lambda)=M(\lambda,\beta)$ from Appendix \ref{app:Am}, with the difference that we replace the $\beta$ by $\beta_j/2$, so $\widetilde{M}_j(\lambda)=M(\lambda,\beta_j/2)$. Finally we define for $j=1,2$

\begin{equation*}
\widetilde{P}_j(z,t)=\widetilde{E}_j(z,t)\widetilde{M}_j(Nu (\zeta(z)\mp i))\Omega_j(z)W(z,t),
\end{equation*}

\noindent where $W$ is as in \eqref{eq:Wdef} and the sign in $\mp$ is such that for $j=1$, we choose the $-$ sign and for $j=2$, we chose the $+$ sign. 

Again we can now be more precise about what $\partial L_{{\mathrm{la}}}$ looks like near $e^{\pm iu}$ --  we choose it so that in a $\kappa/4$-neighborhood of $e^{\pm iu}$, the different parts of $\partial L_{{\mathrm{la}}}$ are mapped onto parts of the rays $e^{k\pi i/4}\times \R_+$ with $k=1,3,5,7$. One can check that with this choice, $\widetilde{P}_1$ has the same jump structure as $S_{{\mathrm{la}}}$ in a $\kappa/4$-neighborhood of $e^{iu}$ and $\widetilde{P}_2$ has the same jump structure as $S_{{\mathrm{la}}}$ in a $\kappa/4$-neighborhood of $e^{-iu}$. Finally we mention that using the explicit form of $\widetilde{M}$ from \cite[Section 4.2.1]{cik}, one can check that $\widetilde{P}_1(z,t)$ has the same asymptotic behavior as $S(z,t)$ as $z\to e^{iu}$ (and similarly $\widetilde{P}_2$ as $z\to e^{-iu}$). We omit the details.

\section{The final transformation}\label{sec:final}

In our final transformation we make use of these approximate solutions. To solve the final RHP asymptotically, we'll need estimates for its jump matrices and we'll derive these estimates in this section. Again we'll need to discuss the small and large $u$ situations separately. 

\subsection{The small \texorpdfstring{$u$}{u} case}

Let $U=\lbrace z: |z-1|<\kappa\rbrace$. We then define 

\begin{equation}
R_{{\mathrm{sm}}}(z,t)=\begin{cases}
S_{{\mathrm{sm}}}(z,t)\mathcal{N}(z,t)^{-1}, & z\in \C\setminus \overline{U}\\
S_{{\mathrm{sm}}}(z,t) P(z,t)^{-1}, & z\in U
\end{cases}.
\end{equation}

By construction, $P$ has the same jump contours and jump matrices as $S_{{\mathrm{sm}}}$ in $U$. It follows that $R_{{\mathrm{sm}}}$ only has jumps across $\partial U$ and $\partial L_{{\mathrm{sm}}}\setminus \overline{U}$. Moreover, from the asymptotic behavior of $S_{{\mathrm{sm}}}$ and $P$ near $e^{\pm iu}$, it follows that the possible isolated singularities of $R_{{\mathrm{sm}}}$ are not strong enough to be poles or essential singularities, so $R_{{\mathrm{sm}}}$ is analytic in $\C\setminus (\partial U\cup (\partial L_{{\mathrm{sm}}}\setminus \overline{U}))$.

We orient $\partial U$ so that the inside of the disk is the $-$ side and the outside of it is the $+$ side. We orient $\partial L_{{\mathrm{sm}}}\setminus \overline{U}$ the same way as before. See Figure \ref{fig:smjump} for a sketch of the jump contour. The RHP associated to $R_{{\mathrm{sm}}}$ is then the following one.

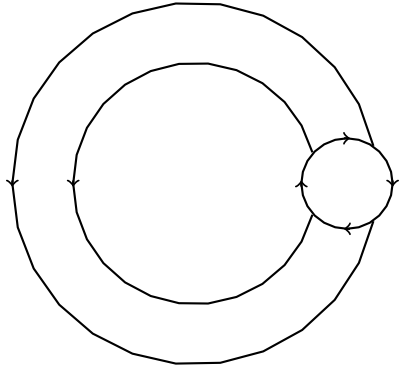
\begin{figure}
\begin{center}
\begin{tikzpicture}[xscale=0.02,yscale=0.02]
\clip (-135,-135) rectangle (170,135);

\draw [black,thick,domain=15:345] plot ({80*cos(\x)}, {80*sin(\x)});
\draw [black,thick,domain=12:348] plot ({120*cos(\x)}, {120*sin(\x)});

\draw[->,thick] (-80,2) -- (-80,-2);
\draw[->,thick] (-120,2) -- (-120,-2);

\draw[->,thick] (70,-2) -- (70,2);
\draw[->,thick] (130,2) -- (130,-2);

\draw[->,thick] (98,30) -- (102,30);
\draw[->,thick] (102,-30) -- (98,-30);

\draw [black,thick,domain=0:360] plot ({100+30*cos(\x)}, {30*sin(\x)});

\end{tikzpicture}
\end{center}
\vspace{-0.3cm}
\caption{The jump contour and orientation of $R_{{\mathrm{sm}}}$. The left side of the contour is the $+$ side.}
\label{fig:smjump}
\end{figure}

\begin{lemma}\label{le:Rrhpsm}
$R_{{\mathrm{sm}}}$ is the unique solution to the following RHP.

\begin{itemize}[leftmargin=0.5cm]
\item[1.] $R_{{\mathrm{sm}}}:\C\setminus (\partial U\cup (\partial L_{{\mathrm{sm}}}\setminus \overline{U}))\to \C^{2\times 2}$ is analytic.
\item[2.] $R_{{\mathrm{sm}}}$ has continuous boundary values on $(\partial U\cup (\partial L_{{\mathrm{sm}}}\setminus \overline{U}))\setminus (\partial U\cap \partial L_{\mathrm{sm}})$ and these satisfy the following jump conditions: for $z\in \partial U\setminus \partial L_{{\mathrm{sm}}}$:

\begin{equation}\label{eq:Rsmjump1}
R_{{\mathrm{sm}},+}(z,t)=R_{{\mathrm{sm}},-}(z,t)P(z,t)\mathcal{N}(z,t)^{-1},
\end{equation}

\noindent for $z\in \partial L_{{\mathrm{sm}}}\lbrace |z|<1\rbrace\setminus\overline{U}$:

\begin{equation}\label{eq:Rsmjump2}
R_{{\mathrm{sm}},+}(z,t)=R_{{\mathrm{sm}},-}(z,t)\mathcal{N}(z,t)\begin{pmatrix}
1 & 0\\
z^N f_t(z)^{-1} & 1
\end{pmatrix}\mathcal{N}(z,t)^{-1},
\end{equation}

\noindent and for $z\in \partial L_{{\mathrm{sm}}}\lbrace |z|>1\rbrace\setminus\overline{U}$:

\begin{equation}\label{eq:Rsmjump3}
R_{{\mathrm{sm}},+}(z,t)=R_{{\mathrm{sm}},-}(z,t)\mathcal{N}(z,t)\begin{pmatrix}
1 & 0\\
z^{-N} f_t(z)^{-1} & 1
\end{pmatrix}\mathcal{N}(z,t)^{-1}.
\end{equation}

\item[3.] As $z\to \infty$, $R_{{\mathrm{sm}}}(z,t)=I+\mathcal{O}(z^{-1})$.
\item[4.] $R_{\mathrm{sm}}(z,t)$ remains bounded as $z\to w\in\partial U\cap \partial L_{\mathrm{sm}}$.
\end{itemize}
\end{lemma}

Our goal is to show that the jump matrix of $R_{{\mathrm{sm}}}$ is uniformly small so we can solve the problem asymptotically through the standard small norm machinery. Let us consider first the jump across $\partial L_{{\mathrm{sm}}}\setminus \overline{U}$.

\begin{lemma}\label{le:smjumpl}
There exists a $c>0$ which is independent of $N$, $z$, $t$, and $\theta,\theta'$ such that uniformly in $z\in \partial L_{{\mathrm{sm}}}\setminus \overline{U}$, $t\in[0,1]$, $\theta,\theta'$

\begin{equation*}
R_{{\mathrm{sm}},-}(z,t)^{-1}R_{{\mathrm{sm}},+}(z,t)=I+\mathcal{O}(e^{-cN^\delta})
\end{equation*}

\noindent uniformly as $N\to\infty$.
\end{lemma}

\begin{proof}
Let us write the jump matrix of $R_{{\mathrm{sm}}}$ across this contour as

\begin{align*}
R_{{\mathrm{sm}},-}(z,t)^{-1}R_{{\mathrm{sm}},+}(z,t)=I+\mathcal{N}(z,t)\begin{pmatrix}
0 & 0\\
z^{\pm N} f_t(z)^{-1} & 0
\end{pmatrix}\mathcal{N}(z,t)^{-1}.
\end{align*}

\noindent where we have the $+$ sign inside the unit disk and the $-$ sign outside of the disk. For simplicity, let's focus on the case where $|z|=1+\frac{2}{3}\kappa$, the $|z|<1$-case is dealt with similarly. Here we have 

\begin{align*}
\mathcal{N}(z,t)\begin{pmatrix}
0 & 0\\
z^{-N} f_t(z)^{-1} & 0
\end{pmatrix}\mathcal{N}(z,t)^{-1}&=\begin{pmatrix}
0 & 0\\
\mathcal{D}_{t,out}(z)^{-2}z^{-N} f_t(z)^{-1} & 0
\end{pmatrix}\\
&=\begin{pmatrix}
0 & 0\\
\mathcal{D}_{t,in}(z)^{-1}\mathcal{D}_{t,out}(z)^{-1}z^{-N}  & 0
\end{pmatrix}.
\end{align*}

As $|z|=1+\frac{2}{3}\kappa\geq 1+C N^{\delta-1}$, we have for such $z$, $|z|^{-N}=\mathcal{O}(e^{-c N^\delta})$ for some suitable $c$ (which is independent of everything relevant). Thus from Lemma \ref{le:Eest}, we conclude the claim for $|z|>1$. As mentioned, the argument for $|z|<1$ is similar.
\end{proof}

Consider now the jump across $\partial U$. Here $S_{{\mathrm{sm}}}$ has no jumps and jump matrix describes how well the parametrices match on $\partial U$.

\begin{lemma}\label{le:smjumpu}
Uniformly in $z\in \partial U$, $t\in[0,1]$, $\theta,\theta'$

\begin{equation*}
R_{{\mathrm{sm}},-}(z,t)^{-1}R_{{\mathrm{sm}},+}(z,t)=I+\mathcal{O}(N^{-\delta})
\end{equation*}

\noindent as $N\to\infty$.
\end{lemma}

\begin{proof}
Let us first note that from the definition of the global parametrix $\mathcal{N}$ and local parametrix $P$, we can write the jump matrix across $\partial U$ as 

\begin{equation*}
P(z,t)\mathcal{N}(z,t)^{-1}=E(z,t)\Phi(\zeta(z),-2iNu) e^{\frac{Nu}{2}\zeta(z)\sigma_3}\widehat{P}^{(\infty)}(\zeta(z))^{-1} E(z,t)^{-1}.
\end{equation*}

We now point out that from Lemma \ref{le:Eest}, $E$ and $E^{-1}$ are uniformly bounded on $\partial U$, so if we can show that uniformly on $\partial U$ (and in everything else) $\Phi(\zeta(z),-2iNu) e^{\frac{Nu}{2}\zeta(z)\sigma_3}\widehat{P}^{(\infty)}(\zeta(z))^{-1}=I+\mathcal{O}(N^{-\delta})$, we are done. To do this we wish to use the asymptotics of $\Phi$ (described in detail in \cite{ck} and summarized in our Appendix \ref{app:A}), which depend on whether $s=-2iNu$ tends to zero, is bounded, or tends to infinity as $N\to\infty$. 

This is slightly  lengthy, but follows quite directly from the results quoted in Appendix \ref{app:A}. Let us fix $c$ small enough and $C$ large enough (small enough and large enough being chosen so that we can use the small $|s|$ asymptotics of $\Phi(\zeta,s)$ for $|s|<c$ and large $|s|$ asymptotics of $\Phi(\zeta,s)$ for $|s|>C$). Consider first the case when $2Nu<c$. As for $z\in \partial U$, $|z-1|=\kappa$ we see that $|\zeta(z)|\asymp \kappa/u$ so in particular, $N^{\delta}\lesssim |\zeta(z)|$ (recall we are considering $u=\mathcal{O}(N^{-1})$ and always $N^{\delta-1}\lesssim \kappa$) where the implied constant is uniform in everything relevant. We thus are interested in the large $\zeta$ and small $s$ asymptotics of $\Phi(\zeta,s)$. We find from \eqref{eq:phiasy3} (writing $\lambda=(\zeta-i)|s|/2$ and noting that for $z\in\partial U$, $|(\zeta(z)-i)s|\asymp N\kappa)$ that for $Nu<c/2$ and $z\in \partial U$

\begin{align*}
\Phi(\zeta(z),-2iNu) e^{\frac{Nu}{2}\zeta(z)\sigma_3}\widehat{P}^{(\infty)}(\zeta(z))^{-1}&=e^{\frac{s}{4}\sigma_3}(I+\mathcal{O}((N\kappa)^{-1}))e^{-\frac{(\zeta(z)-i)|s|}{4}\sigma_3} e^{\frac{Nu}{2}\zeta(z)\sigma_3}\\
&=I+\mathcal{O}((N\kappa)^{-1})=I+\mathcal{O}(N^{-\delta})
\end{align*}

\noindent uniformly in everything relevant.

Consider next the situation where $c\leq Nu\leq C$. With similar reasoning as above,  we have $N^\delta\lesssim|\zeta(z)|$ for $z\in \partial U$. Now from \eqref{eq:phiasy1} we see that on $\partial U$

\begin{equation*}
\Phi(\zeta(z),-2iNu) e^{\frac{Nu}{2}\zeta(z)\sigma_3}\widehat{P}^{(\infty)}(\zeta(z))^{-1}=I+\mathcal{O}(|\zeta(z)|^{-1})=I+\mathcal{O}(N^{-\delta})
\end{equation*}

\noindent uniformly in everything relevant.

Consider finally $C\leq Nu\leq \kappa N/2$. To use the asymptotics of Appendix \ref{app:A}, we need to determine whether or not $\zeta(\partial U)$ intersects the region delimited by $\Gamma_5'$ and $\Gamma_5''$ (see Appendix \ref{app:A} for the description of these contours). To do this, let us deduce the inverse image of $\Gamma_5''$ under $\zeta$ ($\Gamma_5'$ is similar). We thus want to find those $z$ for which $u^{-1}\log z=(1-s)i+s$ for some $s\in[0,1]$ or $u^{-1}\log z=(1-s)(-i)+s$ for some $s\in[0,1]$. We note that this is true for $z=e^{\pm iu} e^{su(1\mp i)}$ so we see that $|z-1|=u\sqrt{s^{2}+(1-s)^{2}}+\mathcal{O}(u^{2})$ where the $\mathcal{O}(u^{2})$ is uniform in $s$. As $u\leq \kappa/2$, we see by choosing $\varepsilon$ in the definition of $\kappa$ small enough that for such $z$, $|z-1|<\kappa$ and $\partial U$ is outside the region delimited by $\Gamma_5'$ and $\Gamma_5''$.

Consider now the relevant asymptotics. From \eqref{eq:phiasy2} we find that for $z\in \partial U$

\begin{equation*}
\Phi(\zeta(z),-2iNu) e^{\frac{Nu}{2}\zeta(z)\sigma_3}\widehat{P}^{(\infty)}(\zeta(z))^{-1}=I+\mathcal{O}(|s\zeta(z)|^{-1}). 
\end{equation*}

Again we have for $z\in \partial U$ $|\zeta(z)|\asymp \kappa/u$ so $N^\delta\lesssim|s\zeta(z)|$ which yields the claim.
\end{proof}

\subsection{The large \texorpdfstring{$u$}{u} case} The basic idea is similar to the small $u$ case, but the differences in the local parametrices change things slightly. We now define $\widetilde{U}=\widetilde{U}_1\cup \widetilde{U}_2$, where $\widetilde{U}_1$ is the $\kappa/4$-neighborhood of $e^{iu}$ and $\widetilde{U}_2$ is the $\kappa/4$-neighborhood of $e^{-iu}$. We then go on and define 

\begin{equation}\label{eq:Rladef}
R_{{\mathrm{la}}}(z,t)=\begin{cases}
S_{{\mathrm{la}}}(z,t)\mathcal{N}(z,t)^{-1}, & z\in \C\setminus \overline{\widetilde{U}}\\
S_{{\mathrm{la}}}(z,t)\widetilde{P}_1(z,t)^{-1}, & z\in \overline{\widetilde{U}_1}\\
S_{{\mathrm{la}}}(z,t)\widetilde{P}_2(z,t)^{-1}, & z\in \overline{\widetilde{U}_2}
\end{cases}.
\end{equation}

\begin{figure}
\begin{center}
\begin{tikzpicture}[xscale=0.02,yscale=0.02]
\clip (-135,-135) rectangle (170,135);

\draw [black,thick,domain=45:315] plot ({80*cos(\x)}, {80*sin(\x)});
\draw [black,thick,domain=42:318] plot ({120*cos(\x)}, {120*sin(\x)});

\draw [black,thick,domain=-15:15] plot ({80*cos(\x)}, {80*sin(\x)});
\draw [black,thick,domain=-19:19] plot ({120*cos(\x)}, {120*sin(\x)});

\draw[->,thick] (-80,2) -- (-80,-2);
\draw[->,thick] (-120,2) -- (-120,-2);

\draw[->,thick] (107,71) -- (108,70);
\draw[->,thick] (65,70) -- (66,71);
\draw[->,thick] (108,30) -- (107,29);
\draw[->,thick] (66,29) -- (65,30);

\draw[->,thick] (108,-70)--(107,-71) ;
\draw[->,thick] (66,-71)--(65,-70) ;
\draw[->,thick] (107,-29)--(108,-30) ;
\draw[->,thick]  (65,-30)--(66,-29) ;

\draw [black,thick,domain=0:360] plot ({86+30*cos(\x)}, {50+ 30*sin(\x)});
\draw [black,thick,domain=0:360] plot ({86+30*cos(\x)}, {-50+ 30*sin(\x)});

\end{tikzpicture}
\end{center}
\vspace{-0.3cm}
\caption{The jump contour and its orientation for $R_{{\mathrm{la}}}$. The left side of the contour is the $+$ side.}
\label{fig:Rlajump}
\end{figure}
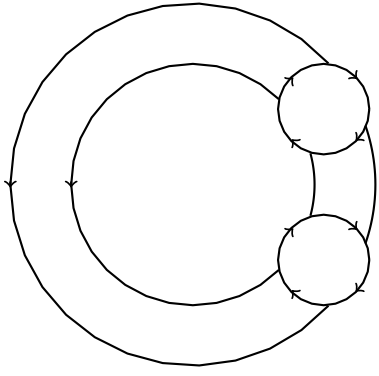

Again $R_{{\mathrm{la}}}$ is analytic off of $(\partial L_{{\mathrm{la}}}\setminus \overline{\widetilde{U}})\cup \partial \widetilde{U}$ and if we choose the orientation of $\partial L_{{\mathrm{la}}}\setminus \overline{\widetilde{U}}$ as before and the orientation of $\partial \widetilde{U}$ such that the inside of the disks is the $-$ side, $R_{{\mathrm{la}}}$ satisfies the following RHP:

\begin{lemma}\label{le:Rrhpla}
$R_{{\mathrm{la}}}$ is characterized by the following RHP.

\begin{itemize}[leftmargin=0.5cm]
\item[1.] $R_{{\mathrm{la}}}:\C\setminus (\partial \widetilde{U}\cup (\partial L_{{\mathrm{la}}}\setminus \overline{\widetilde{U}}))\to \C^{2\times 2}$ is analytic.
\item[2.] $R_{{\mathrm{la}}}$ has continuous boundary values on $(\partial \widetilde{U}\cup (\partial L_{{\mathrm{la}}}\setminus \overline{\widetilde{U}}))\setminus (\partial L_{\mathrm{la}}\cap \partial \widetilde{U})$ and they satisfy the following jump conditions: for $z\in \partial \widetilde{U}_j\setminus \partial L_{{\mathrm{la}}}$

\begin{equation}\label{eq:Rlajump1}
R_{{\mathrm{la}},+}(z,t)=R_{{\mathrm{la}},-}(z,t)\widetilde{P}_j(z,t)\mathcal{N}(z,t)^{-1}
\end{equation}

\noindent and for $z\in \partial L_{{\mathrm{la}}}\setminus \overline{\widetilde{U}}$

\begin{equation}\label{eq:Rlajump2}
R_{{\mathrm{la}},+}(z,t)=R_{{\mathrm{la}},-}(z,t)\mathcal{N}(z,t)\begin{pmatrix}
1 & 0\\
z^{-N} f_t(z)^{-1} & 1
\end{pmatrix}\mathcal{N}(z,t)^{-1}.
\end{equation}

\item[3.] As $z\to \infty$, $R_{{\mathrm{la}}}(z,t)=I+\mathcal{O}(z^{-1})$.
\item[4.] As $z\to w\in \partial L_{\mathrm{la}}\cap \partial \widetilde{U}$, $R_{\mathrm{la}}(z,t)$ remains bounded.
\end{itemize}
\end{lemma}

Let us now estimate the jump matrices. The asymptotics related to the jump across $\partial L_{{\mathrm{la}}}\setminus \overline{\widetilde{U}}$ are obtained as in the small $u$-case so we omit the proof of the following estimate.

\begin{lemma}\label{le:lajumpl}
There exists a $c>0$ which is independent of $N$, $z$, $t$, and $\theta,\theta'$ such that uniformly in $z\in \partial L_{{\mathrm{la}}}\setminus \overline{\widetilde{U}}$, $t\in[0,1]$, $\theta,\theta'$

\begin{equation*}
R_{{\mathrm{sm}},-}(z,t)^{-1}R_{{\mathrm{sm}},+}(z,t)=I+\mathcal{O}(e^{-cN^\delta})
\end{equation*}

\noindent as $N\to\infty$.
\end{lemma}

Again the harder thing to prove is the estimate for the jump matrix across $\partial \widetilde{U}_j$. The claim now is the following.

\begin{lemma}\label{le:lajumpu}
Uniformly in $z\in \partial \widetilde{U}$, $t\in[0,1]$, $\theta,\theta'$

\begin{equation*}
R_{{\mathrm{la}},-}(z,t)^{-1}R_{{\mathrm{la}},+}(z,t)=I+\mathcal{O}(N^{-\delta})
\end{equation*}

\noindent as $N\to\infty$.
\end{lemma}

\begin{proof}
Let us focus on the jump across $\partial\widetilde{U}_1$ for simplicity, the jump across $\partial\widetilde{U}_2$ is similar. Using the definitions of $\widetilde{P}_1$, $W$, $\mathcal{N}$, and $\widetilde{E}_1$, one can check that the jump matrix is

\begin{equation*}
\widetilde{P}_1(z,t)\mathcal{N}(z,t)^{-1}=\widetilde{E}_1(z)\widetilde{M}(Nu(\zeta(z)-i)) e^{\frac{Nu}{2}(\zeta(z)-i)\sigma_3}\widetilde{E}_1(z)^{-1}.
\end{equation*}

Now for $z\in\partial \widetilde{U}_1$, we have $z=e^{iu}+\kappa e^{i\varphi}$ for some $\varphi\in[0,2\pi]$, so $|\zeta(z)-i|=\frac{1}{u}|\log(1+\kappa e^{i(\varphi-u)}|\asymp \kappa/u$. Thus on  $\partial \widetilde{U}_1$, $Nu|\zeta(z)-i|\geq C N^\delta$ for some fixed $C>0$. We conclude from \eqref{eq:Masy} that 

\begin{equation*}
\widetilde{M}_1(Nu(\zeta(z)-i))=\left(I+\mathcal{O}\left(N^{-\delta}\right)\right)e^{-\frac{1}{2}(Nu(\zeta(z)-i))\sigma_3}.
\end{equation*}

From Lemma \ref{le:Eest}, $\widetilde{E}_1$ and $\widetilde{E}_1^{-1}$ are uniformly bounded on $\partial \widetilde{U}_1$, so putting things together, we see that on $\partial \widetilde{U}_1$,

\begin{equation*}
\widetilde{P}_1(z,t)\mathcal{N}(z,t)^{-1}=I+\mathcal{O}(N^{-\delta}),
\end{equation*}

\noindent uniformly in everything relevant. This yields the claim.
\end{proof}

Armed with these estimates, we are in a position to solve the RHPs for $R_{{\mathrm{sm}}}$ and $R_{{\mathrm{la}}}$ asymptotically as $N\to\infty$.
\section{Solving the small norm Riemann-Hilbert problem asymptotically}

The goal of this section is to prove the following result.

\begin{proposition}\label{pr:Rasy}
Let us write $\Sigma_{\mathrm{sm}}$ for the jump contour of $R_{\mathrm{sm}}$ and $\Sigma_{\mathrm{la}}$ for the jump contour of $R_{\mathrm{la}}$. Uniformly in $z$ in compact subsets of $\lbrace z\in \C: d(z,\Sigma_{\mathrm{sm}})\geq \frac{1}{2}\kappa, |z-1|\geq \kappa \rbrace$, uniformly in $t\in[0,1]$ as well as uniformly in $\theta,\theta'$, we have the following asymptotics: 

\begin{equation}\label{eq:rsmasy}
R_{{\mathrm{sm}}}(z,t)=I+\mathcal{O}\left(\kappa N^{-\delta} |z-1|^{-1}\right), \qquad and \qquad \frac{d}{dz}R_{{\mathrm{sm}}}(z,t)=\mathcal{O}(N^{-\delta}|z-1|^{-1})
\end{equation}

\noindent as $N\to\infty$. 

\vspace{0.3cm}

Uniformly in $z$ in compact subsets of  $\lbrace z\in \C: d(z,\Sigma_{\mathrm{la}})\geq \frac{1}{2}\kappa, |z-e^{iu}|,|z-e^{-iu}|\geq \kappa\rbrace$, uniformly in $t\in[0,1]$, as well as uniformly in $\theta,\theta'$, we have the following asymptotics: 

\begin{equation}\label{eq:rlaasy1}
R_{{\mathrm{la}}}(z,t)=I+\mathcal{O}\left(\kappa N^{-\delta}(|z-e^{-iu}|^{-1}+|z-e^{iu}|^{-1})\right) 
\end{equation}

\noindent and 

\begin{equation}\label{eq:rlaasy2}
\frac{d}{dz}R_{{\mathrm{la}}}(z,t)=\mathcal{O}(N^{-\delta}(|z-e^{-iu}|^{-1}+|z-e^{iu}|^{-1}))
\end{equation}

\noindent as $N\to\infty$. 
\end{proposition}

The argument is essentially standard in the RHP literature, but for readers unfamiliar with it, we'll sketch a proof. This being said, as the argument in the large $u$ case is so close to the one in the small $u$ case, we skip the proof in the large $u$ case. For notational simplicity, we'll suppress the ${\mathrm{sm}}$ subscript in $R$ and $L$. 

While the proof is indeed largely standard, we do need to know some bounds on how the norm of the Cauchy (or Hilbert) transform as an operator on $L^2(\Gamma)$ can depend on the contour $\Gamma$. For this, we record the following result which is a simple corollary of very general and strong results. 

\begin{lemma}\label{le:cauchy} Assume that $\Gamma\subset\C$ is a compact and connected set that is $1$-regular, i.e.
\begin{equation}\label{eq:reg}
{\mathcal{H}}^1(B(z,r)\cap\Gamma)\leq Ar
\end{equation}
for every $z\in\Gamma$ and $r\leq {\rm diam}(\Gamma)$ $($here $\mathcal{H}^1$ denotes the one-dimensional Hausdorff measure and $B(z,r)$ the open $r$-disk surrounding $z)$. Define the Cauchy integral operator on $\Gamma$ by 

$$
(C^\Gamma f)(z)=\lim_{\epsilon\to 0^+}\int_{\Gamma\setminus B(z,\epsilon)}\frac{f(s)}{z-s}\frac{d\mathcal{H}^1(s)}{2\pi i}.
$$

\noindent Then $C^\Gamma:L^2(\Gamma, {d{\mathcal{H}}^1})\to L^2(\Gamma,{d{\mathcal{H}}^1})$ is bounded and for the operator norm, we have the estimate

$$
\big\| C^\Gamma\big\|_{L^2(\Gamma,{d{\mathcal{H}}^1})\to L^2(\Gamma,{d{\mathcal{H}}^1})} \leq f(A),
$$
where the quantity $f(A)$ depends only on the constant $A$.
\end{lemma}
\begin{proof}
The boundedness of the Cauchy transform $C^\Gamma$ under the condition \eqref{eq:reg} follows from the well known theorem of G. David \cite{D}. One may run through the proof and dig out  a quantitative bound for how the norm depends on $A$\footnote{The 'digging' is considerably easier by using later developments on the topic as described in  \cite{T}. We are grateful for  Tuomas Orponen for pointing this out.}. However, the existence of a bound that depends just on $A$ can  be verified directly. Assume to the contrary that there exist $A>0$ and compact  connected subsets $\Gamma_k\subset \C$ ($k=1,2,\ldots$) such that \eqref{eq:reg} is valid for all the sets $\Gamma_k$ but with
$$
\big\| C^{\Gamma_k}\big\|_{L^2(\Gamma_k,{d{\mathcal{H}}^1})\to L^2(\Gamma_k,{d{\mathcal{H}}^1})} \geq k,\qquad k\geq 1.
$$
By the translation and dilation invariance of the Cauchy transform we may assume that $2^{-k}\in\Gamma_k$ and ${\rm diam}(\Gamma_k)\leq 2^{-k-8}.$  Denote 
$$
\Gamma := [0,1]\bigcup \Big(\bigcup_{k=1}^\infty \Gamma_k\Big).
$$
Then clearly $\Gamma$ is compact, connected and satisfies \eqref{eq:reg} with some constant $A'<\infty$, but one can check easily that the Cauchy-transform is not bounded on $\Gamma$, which contradicts David's theorem.

\end{proof}

\begin{remark}\label{rem:path}
Note that every $1$-regular connected compact set is the graph of a curve -- see e.g. \rm{\cite{New}}.
\end{remark}

The starting point of the proof of Proposition \ref{pr:Rasy} is characterizing $R$ in terms of the Cauchy transform of a solution to a suitable integral equation. For a proof of a variant of the following result, we refer to \cite[Theorem 7.8]{dkmlvz}. This being said,  there is one subtlety in our case, one needs to use Lemma \ref{le:cauchy} to justify that in the notation introduced below, the norm of $(I-C_\Delta)^{-1}$ on $L^2(\Sigma)$ is bounded in $N$. To justify this, one can easily check that $\Sigma$ satisfies the conditions of Lemma \ref{le:cauchy} so $C^\Sigma$ is bounded on $L^2(\Sigma)$. From this, it follos from e.g. Sokhotski-Plemelj, that $C_-$ is bounded on $L^2(\Sigma)$, and its norm can be bounded above by a quantity independent of $N$. This then implies that the norm of $C_\Delta$ is small for large enough $N$, and $(I-C_\Delta)$ is invertible as a Neumann series for large enough $N$. Moreover, its norm is bounded from above by a quantity independent of $N$. We omit further details about the proof.

\begin{lemma}\label{le:inteq}
Let us orient $\Sigma$ as before, and write $J_R$ for the jump matrix of $R$, and $\Delta=J_R-I$. Then for large enough $N$, and $z\notin \Sigma$, we can write 

\begin{equation*}
R=I+C\big(\Delta+\Delta\big[(I-C_\Delta)^{-1}C_-(\Delta)\big]\big),
\end{equation*}

\noindent where $C$ denotes the Cauchy-transform associated to $\Sigma$ $($for notational simplicity, we drop the dependence on $\Sigma$ in our notation now$):$ 

\begin{equation*}
(Cf)(z)=\int_\Sigma\frac{f(w)}{w-z}\frac{dw}{2\pi i},
\end{equation*}

\noindent $C_-$ denotes its limit from the $-$ side of $\Sigma$: for $z\in \Sigma$, $(C_-f)(z)=\lim_{\xi\to z^-} (Cf)(\xi)$, and $C_\Delta(f)=C_-(\Delta f)$.
\end{lemma}

This result is enough for us to prove Proposition \ref{pr:Rasy}.

\begin{proof}[Sketch of a proof of Proposition \ref{pr:Rasy}]
Consider a point $z\in \C\setminus \Sigma$ with $d(z,\Sigma)\geq \kappa/2$. Let us first point out that by Lemma \ref{le:smjumpl} and Lemma \ref{le:smjumpu}

\begin{align*}
|[C(\Delta)](z)|&\leq \int_{\partial U}\frac{|\Delta(w)|}{|w-z|}\frac{dw}{2\pi i}+\int_{\partial L\setminus \overline{U}}\frac{|\Delta(w)|}{|w-z|}\frac{dw}{2\pi i}\\
&\leq \mathcal{O}(N^{-\delta}|\partial U|d(z,\partial U)^{-1})+\mathcal{O}(e^{-cN^\delta}d(z,\partial L\setminus\overline{U})^{-1})\\
&\leq \mathcal{O}\left(\kappa N^{-\delta}|z-1|^{-1}\right)+\mathcal{O}(e^{-cN^\delta} \kappa^{-1})\\
&=\mathcal{O}\left(\kappa N^{-\delta}|z-1|^{-1}\right)
\end{align*}

\noindent with the relevant uniformity.

On the other hand, by Cauchy-Schwarz, and performing a similar argument splitting $\Sigma$ into two parts, as well as making use of the fact that the norm of $(I-C_\Delta)^{-1}$ is bounded in $N$  (which follows from Lemma \ref{le:cauchy} as pointed out earlier)

\begin{align*}
\left|C\left(\Delta(1-C_\Delta)^{-1}[C_-(\Delta)]\right)(z)\right|&\leq 
\left(\int_\Sigma \frac{|\Delta(w)|^2}{|w-z|^2}|dw|\right)^{1/2}
||(1-C_\Delta)^{-1}[C_-(\Delta)]||_{L^2(\Sigma)}\\
&\lesssim \frac{\mathcal{O}(\sqrt{\kappa} N^{-\delta})}{|z-1|}||C_-(\Delta)||_{L^2(\Sigma)}\\
&\leq \mathcal{O}(\kappa N^{-2\delta}|z-1|^{-1}).
\end{align*}

We conclude that for $z\notin \Sigma$ with $d(z,\Sigma)\geq \kappa/2$, 

\begin{equation*}
R(z,t)=I+\mathcal{O}\left(\kappa N^{-\delta}|z-1|^{-1}\right),
\end{equation*}

\noindent where the implied constant is uniform in everything relevant.

If it were of use, one could make use of a standard contour deformation argument -- see e.g. \cite[Proof of Corollary 7.9]{dkmlvz} -- which would allow one to extend similar asymptotics onto $\Sigma$, but we have no use of it here.

For the derivative, note that for $d(z,\Sigma)>\kappa/2$  Cauchy's integral formula implies that

\begin{equation*}
\left|\frac{d}{dz}R(z,t)\right|\leq \int_{|w-z|=\kappa/4}\frac{|R(w,t)-I|}{|w-z|^2}\frac{|dw|}{2\pi }=\mathcal{O}\left(\frac{\kappa}{\kappa^2}\kappa N^{-\delta}|z-1|^{-1}\right)=\mathcal{O}(N^{-\delta}|z-1|^{-1})
\end{equation*}

\noindent uniformly in everything relevant.
\end{proof}

\section{Integrating the differential identity}\label{sec:intdi}

Let us now go back to our differential identity. Our argument is essentially the same as in \cite[Section 5.3]{dik2} and in \cite{abb}. Again, as the $R_{\mathrm{la}}$-case is similar, we'll present things for $R_{{\mathrm{sm}}}$ and drop the subscript $\mathrm{sm}$.

Our goal is to deform our integration contour into $\lbrace z\in \C: |z|=1\pm 2\kappa\rbrace$ and express here $Y$ in terms of $R$ and the global parametrix as well as make use of the asymptotics of $R$. It will turn out that the leading order asymptotics of $R$ are sufficient for our purposes and that the global parametrix on the other hand is something for which the relevant integrals can be performed explicitly. Let us begin with expressing our integrand in a way which is more suitable for our contour deformation argument. 

From now on, let us write $'$ for differentiation with respect to $z$ and $\dot{}$ for differentiation with respect to $t$.

\begin{lemma}\label{le:direform}
For $z\in \T$, 

\begin{equation*}
\left(Y_\pm (z,t)^{-1}Y_\pm'(z,t)\right)_{11}=Y_{11}'(z,t)Y_{22,\pm}(z,t)-Y_{21}'(z,t)Y_{12,\pm}(z,t)
\end{equation*}

\noindent and 

\begin{align*}
z^{-N}&\left(Y_{11}(z,t)Y_{21}'(z,t)-Y_{21}(z,t)Y_{11}'(z,t)\right)\dot{f}_t(z)\\
&=-\left[\left(Y_+(z,t)^{-1}Y_+'(z,t)\right)_{11}-
\left(Y_-(z,t)^{-1}Y_-'(z,t)\right)_{11}\right]\frac{\dot{f}_t(z)}{f_t(z)},
\end{align*}

\noindent where $Y_\pm '$ denotes the boundary values of $Y'$.
\end{lemma}

\begin{proof}
The first claim is a simple calculation, where we use the fact that $\det Y=1$ and point out that $Y_{11}$ and $Y_{21}$ are polynomials so their derivatives exist and have no jumps on $\T$. $Y_{22}$ and $Y_{12}$ have continuous boundary values as we have mentioned. Thus the functions $(Y_\pm^{-1}Y_\pm')_{11}$ are continuous on $\T$.

For the second claim, we note that recalling \eqref{eq:Yjump}, one can check with a direct calculation that on $\T$

\begin{align*}
\left(Y_+^{-1}Y_+'\right)_{11}&=\left(Y_-^{-1}Y_-'\right)_{11} -\left(Y_-^{-1}Y_-'\right)_{21} f z^{-N}\\
&=\left(Y_-^{-1}Y_-'\right)_{11}-z^{-N}\left(Y_{11}Y_{21}'-Y_{21}Y_{11}'\right)f
\end{align*}

\noindent which yields the claim.
\end{proof}

We next deform $\T$ into a suitable contour. 

\begin{lemma}\label{le:dicontdefo}
Let 

\begin{equation*}
\Gamma_\pm =\lbrace |z|=1\pm 2\kappa\rbrace
\end{equation*}

\noindent oriented in the counter-clockwise direction.

Then 

\begin{align*}
\int_\T &z^{-N}\left(Y_{11}(z,t)Y_{21}'(z,t)-Y_{21}(z,t)Y_{11}'(z,t)\right)\dot{f}_t(z)\frac{dz}{2\pi i}\\
&=\int_{\Gamma_+}\left(Y(z,t)^{-1}Y'(z,t)\right)_{11}\frac{\dot{f}_t(z)}{f_t(z)}\frac{dz}{2\pi i}-\int_{\Gamma_-}\left(Y(z,t)^{-1}Y'(z,t)\right)_{11}\frac{\dot{f}_t(z)}{f_t(z)}\frac{dz}{2\pi i}
\end{align*}
\end{lemma}

\begin{proof}
By Lemma \ref{le:direform} and the proof of Lemma \ref{le:vtholo}, $(Y(z,t)^{-1}Y'(z,t))_{11}\frac{\dot{f}_t(z)}{f_t(z)}$ is holomorphic in $\lbrace 1-3\kappa <|z|<1\rbrace$ and $\lbrace 1<|z|<1+3\kappa\rbrace$ so the claim follows from Cauchy's integral theorem. 
\end{proof}

We can now express the values of $Y$ on on $\Gamma_\pm$ in terms of the global parametrix and $R$.

\begin{lemma}\label{le:diglobal}
On $\Gamma_-$, 

\begin{align*}
\left(Y(z,t)^{-1}Y'(z,t)\right)_{11}\frac{\dot{f}_t(z)}{f_t(z)}&=(R(z,t)^{-1}R'(z,t))_{22}\frac{\dot{f}_t(z)}{f_t(z)}-\frac{\mathcal{D}_{t,in}'(z)}{\mathcal{D}_{t,in}(z)}\frac{\dot{f}_t(z)}{f_t(z)}\\
&=-\left(\sum_{j=1}^\infty j V_j(t) z^{j-1}+\frac{\widetilde{\beta}_1}{2}\frac{1}{z-e^{iu}}+\frac{\widetilde{\beta}_2}{2}\frac{1}{z-e^{-iu}}\right)\frac{\dot{f}_t(z)}{f_t(z)}\\
&\quad +\mathcal{O}(N^{-\delta}|z-1|^{-1})\times \frac{\dot{f}_t(z)}{f_t(z)}
\end{align*}

\noindent and on $\Gamma_+$

\begin{align*}
&\left(Y(z,t)^{-1}Y'(z,t)\right)_{11}\frac{\dot{f}_t(z)}{f_t(z)}\\
&=(R(z,t)^{-1}R'(z,t))_{11}\frac{\dot{f}_t(z)}{f_t(z)}+\frac{\mathcal{D}_{t,out}'(z)}{\mathcal{D}_{t,out}(z)}\frac{\dot{f}_t(z)}{f_t(z)}+\frac{N}{z}\frac{\dot{f}_t(z)}{f_t(z)}\\
&=\left(\sum_{j=1}^\infty j V_{-j}(t) z^{-j-1}-\frac{\widetilde{\beta}_1}{2}\frac{1}{z-e^{iu}}-\frac{\widetilde{\beta}_2}{2}\frac{1}{z-e^{-iu}}+\frac{2N+\widetilde{\beta}_1+\widetilde{\beta}_2}{2z}+\mathcal{O}(N^{-\delta}|z-1|^{-1})\right)\frac{\dot{f}_t(z)}{f_t(z)},
\end{align*}

\noindent where the $\mathcal{O}(N^{-\delta}|z-1|^{-1})$-terms are uniform in everything relevant.

\end{lemma}

\begin{proof}
The formulas where $R$ appears follow from a direct calculation which simply expresses $Y$ in terms of $R$ and $\mathcal{N}$ (which is then expressed in terms of $\mathcal{D}_{t,in/out}$). The second set of identities then make use of the fact that on $\Gamma_\pm$ $R'=\mathcal{O}(N^{-\delta}|z-1|^{-1})$ and $R$ is bounded, as well as the definition of $\mathcal{D}_{t,in/out}$.
\end{proof}

We are now in a position where we can start calculating our integral over $t$. Let us first consider the $V_j(t)$ terms. For these, we note that if we define $g(z,t)=\int_{\T}\frac{V_t(w)}{w-z}\frac{dw}{2\pi i}$, then by Cauchy's integral theorem

\begin{align*}
\int_{\Gamma_-}\sum_{j=1}^\infty j V_j(t) z^{j-1} \frac{\dot{f}_t(z)}{f_t(z)}\frac{dz}{2\pi i }&+\int_{\Gamma_+}\sum_{j=1}^\infty j V_{-j}(t) z^{-j-1} \frac{\dot{f}_t(z)}{f_t(z)}\frac{dz}{2\pi i}\\
&=\int_{\T}(g_+'(z,t)+g_-'(z,t)) \frac{\dot{f}_t(z)}{f_t(z)}\frac{dz}{2\pi i}.
\end{align*}

\noindent The $t$-integral of this has in turn been calculated by Deift \cite[equations (86) and (87)]{deift2}:

\begin{equation*}
\int_0^1\int_{\T}(g_+'(z,t)+g_-'(z,t)) \frac{\dot{f}_t(z)}{f_t(z)}\frac{dz}{2\pi i}=\sum_{k=1}^\infty k V_k V_{-k}.
\end{equation*}

We record this argument as a lemma. 

\begin{lemma}[Deift]\label{le:deift}
\begin{equation*}
\int_0^1\int_{\Gamma_-}\sum_{j=1}^\infty j V_j(t) z^{j-1} \frac{\dot{f}_t(z)}{f_t(z)}\frac{dz}{2\pi i }+\int_{\Gamma_+}\sum_{j=1}^\infty j V_{-j}(t) z^{-j-1} \frac{\dot{f}_t(z)}{f_t(z)}\frac{dz}{2\pi i}dt=\sum_{j=1}^\infty j V_jV_{-j}.
\end{equation*}

\end{lemma}

Let us then consider the terms containing $\pm u$ as well as the $z^{-1}$-terms. This is precisely as in \cite{dik2}. The point being that for these terms the $t$-integral is easy to perform as $\int_0^1 \frac{\dot{f}}{f}dt=V$. Then by contour deformation, we can reduce the integrals into ones over $\T$.

\begin{lemma}\label{le:FH}
\begin{align*}
\int_0^1\int_{\Gamma_-}&\left[\frac{\widetilde{\beta}_1}{2}\frac{1}{z-e^{iu}}+\frac{\widetilde{\beta}_2}{2}\frac{1}{z-e^{-iu}}\right]\frac{\dot{f}_t(z)}{f_t(z)}\frac{dz}{2\pi i}dt\\
&+\int_0^1\int_{\Gamma_+}\left[-\frac{\widetilde{\beta}_1}{2}\frac{1}{z-e^{iu}}-\frac{\widetilde{\beta}_2}{2}\frac{1}{z-e^{-iu}}+\frac{\widetilde{\beta}_1+\widetilde{\beta}_2}{2z}+\frac{N}{z}\right]\frac{\dot{f}_t(z)}{f_t(z)}\frac{dz}{2\pi i}dt\\
&=-\frac{\widetilde{\beta}_1}{2}V(e^{iu})-\frac{\widetilde{\beta}_2}{2}V(e^{-iu}).
\end{align*}
\end{lemma}

\begin{proof}
As mentioned, we start with the fact that $\int_0^1 \frac{\dot{f}_t(z)}{f_t(z)}dt=\int_0^1\partial_t \log(1-t+te^{V(z)})dt=V(z)$. By contour deformation, we thus have

\begin{align*}
\int_0^1\int_{\Gamma_-}\frac{1}{z-e^{iu}}\frac{\dot{f}_t(z)}{f_t(z)}\frac{dz}{2\pi i}dt&=-e^{-iu}\sum_{k=0}^\infty e^{-iku}\int_{\Gamma_-} V(z) z^k\frac{dz}{2\pi i}\\
&=-\sum_{k=0}^\infty e^{-i(k+1)u}\int_{\T}V(z) z^{k+1}\frac{dz}{2\pi i z}\\
&=-\sum_{k=1}^\infty V_{-k}e^{-iku}  
\end{align*}

\noindent and 

\begin{align*}
\int_0^1\int_{\Gamma_+}\frac{1}{z-e^{iu}}\frac{\dot{f}_t(z)}{f_t(z)}\frac{dz}{2\pi i}dt&=\sum_{k=0}^\infty \int_{\Gamma_+} e^{iku}z^{-k-1}V(z)\frac{dz}{2\pi i}\\
&=\sum_{k=0}^\infty V_k e^{iku}.
\end{align*}

From this, we see also that the $z^{-1}$-terms yield something proportional to $V_0=0$. Combining these facts, we see the claim.
\end{proof}

Our remaining task is to estimate the integrals of $\mathcal{O}(N^{-\delta}|z-1|^{-1})\frac{\dot{f}}{f}$. This is essentially identical to the corresponding estimate in \cite{abb}. 

\begin{lemma}\label{le:dierror}
As $N\to\infty$, 

\begin{equation*}
\int_{\Gamma_\pm}\int_0^1\mathcal{O}(N^{-\delta}|z-1|^{-1})\left|\frac{\dot{f}_t(z)}{f_t(z)}\right| |dz|dt=\mathcal{O}(N^{-\delta}(\log N)^3).
\end{equation*}
\end{lemma}

\begin{proof}
Consider first the situation where $t\leq e^{-(\log N)^2}$. For such $t$, we note from the proof of Lemma \ref{le:vtholo} and the fact that on $\T$, $V(z)=\mathcal{O}(\log N)$,  that also on $\Gamma_\pm$, $V(z)=\mathcal{O}(\log N)$ (uniformly in $z$). Thus for $t\leq e^{-(\log N)^2}$,

\begin{equation*}
\frac{\dot{f}_t(z)}{f_t(z)}=\frac{e^{V(z)}-1}{1-t+t e^{V(z)}}=\mathcal{O}(N^{\mathcal{O}(1)}),
\end{equation*}

\noindent where $\mathcal{O}(1)$ is uniform in $t$ and $z$. Next we note that writing 

\begin{equation*}
\frac{\dot{f}_t(z)}{f_t(z)}=\frac{1-e^{-V(z)}}{(1-t)e^{-V(z)}+t },
\end{equation*}

\noindent we find the same bound for $t\geq 1-e^{-(\log N)^2}$.

Now for $e^{-(\log N)^2}\leq t\leq 1-e^{-(\log N)^2}$, we make use of the following reasoning: depending on the values of $\alpha_1,\alpha_2$, Lemma \ref{le:vtholo} and Lemma \ref{le:logsum} imply that either $|e^{V(z)}|$ is bounded from below by some positive constant (independent of $N$ and $u$), or $|e^{-V(z)}|$ is bounded from below by some similar positive constant. Let us consider the first case, we then have again from the proof of Lemma \ref{le:vtholo} that $\mathrm{Re}\ e^{V(z)}>0$ on $\Gamma_\pm$, so 

\begin{equation*}
\left|\frac{e^{V(z)}-1}{1-t+t e^{V(z)}}\right|\leq \frac{1}{t}\left|\frac{e^{V(z)}-1}{e^{V(z)}}\right|\leq \frac{C}{t}
\end{equation*}

\noindent for some constant $C$ which is uniform in everything relevant. If on the other hand $|e^{-V(z)}|$ is bounded from below by some positive constant, we write 

\begin{equation*}
\left|\frac{e^{V(z)}-1}{1-t+t e^{V(z)}}\right|= \left|\frac{1-e^{-V(z)}}{(1-t)e^{-V(z)}+t}\right|\leq \frac{1}{1-t}\left|\frac{1-e^{-V(z)}}{e^{-V(z)}}\right|\leq \frac{C}{1-t}.
\end{equation*}

 We thus find

\begin{align*}
\int_0^1\left|\frac{\dot{f}_t(z)}{f_t(z)}\right|dt&\leq \left(\int_0^{e^{-(\log N)^2}}+\int_{1-e^{-(\log N)^2}}^1\right)N^{\mathcal{O}(1)}dt+C\int_{e^{-(\log N)^2}}^{1-e^{-(\log N)^2}} \frac{dt}{\min(t,1-t)}\\
&=\mathcal{O}((\log N)^2).
\end{align*}

\noindent uniformly in $z$ as well as $\theta,\theta'$. We thus have to estimate the integral

\begin{align*}
\int_{\Gamma^\pm}|z-1|^{-1}|dz|&\asymp \int_{0}^1 \frac{1}{x+\kappa}dx\asymp \log \kappa^{-1}=\mathcal{O}(\log N).
\end{align*}

\end{proof}

This lets us finally prove Proposition \ref{pr:testimate}.

\begin{proof}[Proof of Proposition $\ref{pr:testimate}$]
Combining Proposition \ref{pr:di} with Lemma \ref{le:direform}, Lemma \ref{le:dicontdefo}, Lemma \ref{le:diglobal}, Lemma \ref{le:deift}, Lemma \ref{le:FH}, and Lemma \ref{le:dierror}, we see that 

\begin{align*}
&\frac{\E e^{\beta_1 X_N(\theta)+\beta_2 X_N(\theta')+\alpha_1 X_{N,K_1}(\theta)+\alpha_2 X_{N,K_2}(\theta)+\mathrm{Tr}\mathcal{T}(U_N)}}{\E e^{\beta X_N(\theta)+\beta X_N(\theta')}}\\
\quad &=\frac{D_{N-1}(f_1)}{D_{N-1}(f_0)}\\
\quad &=e^{\sum_{j=1}^\infty jV_j V_{-j}-\frac{\widetilde{\beta}_1}{2}V(e^{iu})-\frac{\widetilde{\beta}_2}{2}V(e^{-iu})}(1+o(1))\\
\end{align*}

\noindent where $o(1)$ is uniform in everything relevant. To see that this is precisely the claim, note that from the definition of $V$, \eqref{eq:Vdef}, as well as the definition of $\phi$, $\widetilde{\beta}_i$, and $u$ above and in Definition \ref{def:f}, we have 

\begin{align*}
\frac{\widetilde{\beta}_1 V(e^{iu})}{2}+\frac{\widetilde{\beta}_2V(e^{-iu}) }{2}&=\frac{\beta_1 \mathcal{T}(e^{i\theta})}{2}+\frac{\beta_2 \mathcal{T}(e^{i\theta'})}{2}-\alpha_1\sum_{j=1}^{K_1}\frac{\beta_1+\beta_2\cos 2j u}{2j}\\
&\qquad -\alpha_2\sum_{j=1}^{K_2}\frac{\beta_1+\beta_2\cos 2j u}{2j}\\
&=\frac{\beta_1 \mathcal{T}(e^{i\theta})}{2}+\frac{\beta_2 \mathcal{T}(e^{i\theta'})}{2}-\alpha_1\sum_{j=1}^{K_1}\frac{\beta_1+\beta_2\cos j(\theta-\theta')}{2j}\\
&\qquad -\alpha_2\sum_{j=1}^{K_2}\frac{\beta_1+\beta_2\cos j(\theta-\theta')}{2j}
\end{align*}

\noindent and 

\begin{align*}
\sum_{j=1}^\infty jV_j V_{-j}&=\sum_{j=1}^M j\mathcal{T}_j\mathcal{T}_{-j}+\frac{\alpha_1^2}{4}\sum_{j=1}^{K_1}\frac{1}{j}+\frac{\alpha_2^2}{4}\sum_{j=1}^{K_2}\frac{1}{j}+\frac{\alpha_1\alpha_2}{2}\sum_{j=1}^{K_1}\frac{1}{j}\\
&\quad -\frac{\alpha_1}{2}\sum_{j=1}^{\min(M,K_1)}\left(\mathcal{T}_j e^{ij\phi}e^{\pm iju}+\mathcal{T}_{-j}e^{-ij\phi}e^{\mp ij u}\right)\\
&\quad -\frac{\alpha_2}{2}\sum_{j=1}^{\min(M,K_2)}\left(\mathcal{T}_j e^{ij\phi}e^{\pm iju}+\mathcal{T}_{-j}e^{-ij\phi}e^{\mp ij u}\right)\\
&=\sum_{j=1}^M j\mathcal{T}_j\mathcal{T}_{-j}+\frac{\alpha_1^2}{4}\sum_{j=1}^{K_1}\frac{1}{j}+\frac{\alpha_2^2}{4}\sum_{j=1}^{K_2}\frac{1}{j}+\frac{\alpha_1\alpha_2}{2}\sum_{j=1}^{K_1}\frac{1}{j}\\
&\quad -\frac{\alpha_1}{2}\sum_{j=1}^{\min(M,K_1)}\left(\mathcal{T}_j e^{ij\theta}+\mathcal{T}_{-j}e^{-ij\theta}\right)-\frac{\alpha_2}{2}\sum_{j=1}^{\min(M,K_2)}\left(\mathcal{T}_j e^{ij\theta}+\mathcal{T}_{-j}e^{-ij\theta}\right),
\end{align*}

\noindent where the sign in $\pm u$ depends on $\theta,\theta'$ as in Definition \ref{def:f}. Combining all these facts finally yields Proposition \ref{pr:testimate}.
\end{proof}

\appendix

\section{The local parametrix \texorpdfstring{-- results from \cite{ck}}{}}\label{app:A}

In this appendix we give a very brief review of some of the results from \cite{ck} which are relevant to the local parametrix. We do not give any proofs but simply refer to the relevant parts of \cite{ck}. We first give definitions for the objects we are interested in and then describe the asymptotics we are interested in.

\subsection{Definition \texorpdfstring{of $\Psi$}{}}
A fundamental model Riemann-Hilbert problem underlying the analysis in \cite{ck} is introduced in \cite[Section 3]{ck}. We will consider a simplified version of it. The connection between the notational conventions is the following:
 $\beta_1^{\mathrm{CK}}=\beta_2^{\mathrm{CK}}=0$, $2\alpha_1^{\mathrm{CK}}=\beta_1^{\mathrm{ours}}$, and  $2\alpha_2^{\mathrm{CK}}=\beta_2^{\mathrm{ours}}$.  We now describe the RHP we are interested in.

\begin{definition}\label{def:psi}
For each $s\in(-i \R_+)$, let $z\mapsto \Psi(z,s)$ be the unique solution to the following Riemann-Hilbert problem$:$

Find a function $\Psi=\Psi(\zeta,s)$ such that

\begin{itemize}[leftmargin=0.5cm]

\item[1.] $\Psi(\cdot,s):\C\setminus \Gamma\to \C^{2\times 2}$ is analytic, where 

\begin{equation}\label{eq:gamma}
\begin{array}{lll}
\Gamma=\cup_{k=1}^{5}\Gamma_k, & \Gamma_1=i+e^{i\frac{\pi}{4}}\R_+, & \Gamma_2=i +e^{i\frac{3\pi}{4}}\R_+,\\
\Gamma_3=-i+e^{i\frac{5\pi}{4}}\R_+, & \Gamma_4=-i+e^{i\frac{7\pi}{4}}\R_+, & \Gamma_5=[-i,i]. \\
\end{array}
\end{equation}

\item[2.] $\Psi$ has continuous boundary values on $\Gamma\setminus\lbrace -i,i\rbrace$ and these satisfy the jump conditions$:$ for $\zeta\in \Gamma_k\setminus\lbrace -i,i\rbrace$,

\begin{equation}\label{eq:psijump}
\Psi_+(\zeta,s)=\Psi_-(\zeta,s)J_k,
\end{equation}

\noindent where $\Psi_+$ $(\Psi_-)$ denotes the limit of $\Psi$ from the left $($right$)$ of the contour $($the arrows in Figure $\ref{figure: Gamma}$ determine the orientation of the curves$)$, and 

\begin{equation}\label{eq:psij}
\begin{array}{ll}
 J_1=\begin{pmatrix}1&e^{\pi i\beta_1}\\0&1\end{pmatrix}, & J_2=\begin{pmatrix}1&0\\-e^{-\pi i\beta_1}&1\end{pmatrix}\\
J_3=\begin{pmatrix}1&0\\-e^{\pi i\beta_2}&1\end{pmatrix},&  J_4=\begin{pmatrix}1&e^{-\pi i\beta_2}\\0&1\end{pmatrix}\\
J_5=\begin{pmatrix}0&1\\-1&1\end{pmatrix}. &
\end{array}
\end{equation}
\item[3.] In all regions, 

\begin{equation}\label{eq:psiasy}
\Psi(\zeta,s)=\left(I+\Psi_1(s)\zeta^{-1}+\Psi_2(s)\zeta^{-2}+\mathcal{O}(|\zeta|^{-3})\right)e^{-\frac{is}{4}\zeta \sigma_3}
\end{equation}

\noindent as $\zeta\to \infty$.

\item[4.] For $\beta_1,\beta_2\notin \Z_+$, let 

\begin{equation*}
g=-\frac{1}{2i \sin(\pi \beta_1)}(e^{\pi i \beta_1}-1), \quad h=-\frac{1}{2i \sin(\pi \beta_2)}(1-e^{-i \pi \beta_2}),
\end{equation*}

\begin{equation*}
G_{III}=\begin{pmatrix}
1 & g\\
0 & 1
\end{pmatrix}, \quad G_I=G_{III}J_5^{-1}, \quad G_{II}=G_1 J_1,
\end{equation*}

\noindent and

\begin{equation*}
H_{III}=\begin{pmatrix}
1 & h\\
0 & 1
\end{pmatrix}, \quad H_{IV}=H_{III}J_3^{-1},\quad H_I=H_{IV}J_4^{-1}.
\end{equation*}

Then define $F_1=F_1(\zeta,s)$ in a neighborhood of $i$ by 

\begin{equation}\label{eq:psiati1}
\Psi(\zeta,s)=F_1(\zeta,s)(\zeta-i)^{\frac{\beta_1}{2}\sigma_3} G_j
\end{equation}

\noindent in regions $j=I,II,III$, where the branch cut of $(\zeta-i)^{\frac{\beta_1}{2}\sigma_3}$ is along $i+e^{\frac{3\pi i }{4}}(0,\infty)$ and $\mathrm{arg}(\zeta-i)\in(-5\pi/4, 3\pi /4)$, and $F_2=F_2(\zeta,s)$ in a neighborhood of $-i$ by 

\begin{equation}\label{eq:psiat-i1}
\Psi(\zeta,s)=F_2(\zeta,s)(\zeta+i)^{\frac{\beta_2}{2}\sigma_3} H_j
\end{equation}

\noindent in regions $j=I,III,IV$, where the branch cut now is along $-i+e^{\frac{5\pi i}{4}}(0,\infty)$ and $\mathrm{arg}(\zeta+i)\in(-3\pi/4,5\pi/4)$.

If $\beta\in\Z_+$, define $G_{III}=H_{III}=I$ and the other matrices through the jump matrices as before, and define in the region $j$, $F_1$ and $F_2$ $($with similar branch cuts$)$ by 

\begin{equation}\label{eq:psiati2}
\Psi(\zeta,s)=F_1(\zeta,s)(\zeta-i)^{\frac{\beta_1}{2}\sigma_3}\begin{pmatrix}
1 & \frac{1-e^{\pi i \beta_1}}{2\pi i e^{\pi i \beta_1}}\log(\zeta-i)\\
0 & 1
\end{pmatrix}G_j
\end{equation}

\noindent and

\begin{equation}\label{eq:psiat-i2}
\Psi(\zeta,s)=F_2(\zeta,s)(\zeta+i)^{\frac{\beta_2}{2}\sigma_3}\begin{pmatrix}
1 & \frac{e^{-\pi i \beta_2}-1}{2\pi i e^{-\pi i \beta_2}}\log(\zeta+i)\\
0 & 1
\end{pmatrix}H_j.
\end{equation}

Then these functions $F_1$ and $F_2$  must be analytic functions of $\zeta$ in some neighborhoods of $\pm i$.
\end{itemize}
\end{definition}

\begin{figure}[t]\label{fig:psijumps}
\begin{center}
    \setlength{\unitlength}{0.8truemm}
    \begin{picture}(110,100)(-10,-7.5)
    \put(50,60){\thicklines\circle*{.8}}
    \put(50,30){\thicklines\circle*{.8}}
    \put(51,56){\small $+i$}
    \put(51,31){\small $-i$}
    \put(50,60){\line(1,1){25}}
    \put(50,30){\line(1,-1){25}}
    \put(50,60){\line(-1,1){25}}
    \put(50,30){\line(-1,-1){25}}
    \put(50,30){\line(0,1){30}}
    \put(65,75){\thicklines\vector(1,1){.0001}}
    \put(65,15){\thicklines\vector(-1,1){.0001}}
    \put(35,75){\thicklines\vector(-1,1){.0001}}
    \put(50,47){\thicklines\vector(0,1){.0001}}
    \put(35,15){\thicklines\vector(1,1){.0001}}

    \put(74,88){\small $\begin{pmatrix}1&e^{\pi i\beta_1}\\0&1\end{pmatrix}$}
    \put(-2,88){\small $\begin{pmatrix}1&0\\-e^{-\pi i\beta_1}&1\end{pmatrix}$}
    \put(31,45){\small $\begin{pmatrix}0&1\\-1&1\end{pmatrix}$}
    \put(-2,0){\small $\begin{pmatrix}1&0\\-e^{\pi i\beta_2}&1\end{pmatrix}$}
    \put(74,0){\small $\begin{pmatrix}1&e^{-\pi i\beta_2}\\0&1\end{pmatrix}$}
    \put(8,43){III}
    \put(90,43){I}
    \put(47,7){IV}
    \put(47,75){II}
    \end{picture}
    \caption{The jump contour and jump matrices for $\Psi$ (a modification of  \cite[Figure 1]{ck}).}
    \label{figure: Gamma}
\end{center}
\end{figure}
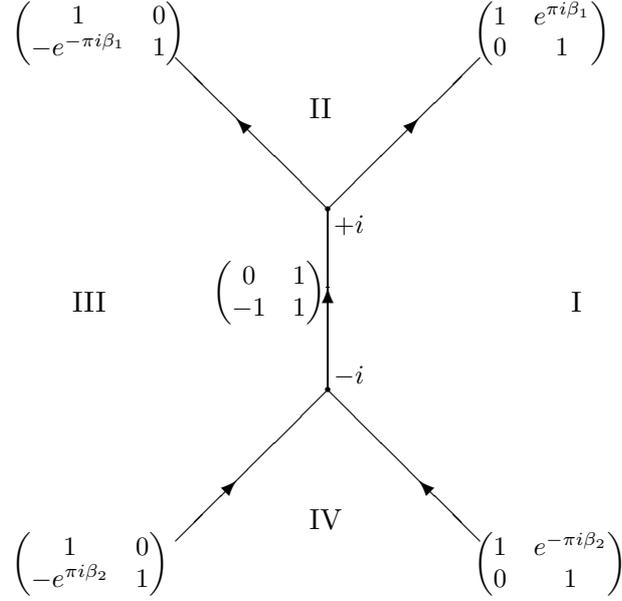

\begin{remark}
It was proven in {\rm{\cite{ck}}} that this problem has a unique solution. Uniqueness of a solution is argued at the end of {\rm{\cite[Section 3.1]{ck}}} and existence is argued in {\rm{\cite[Section 3.4]{ck}}}.
\end{remark}

\subsection{\texorpdfstring{Bounded $s$ and large $\zeta$ asymptotics of $\Psi(\zeta,s)$}{Asymptotics 1}}

For bounded $s$, that is if we assume that there exist fixed $c,C\in(0,\infty)$ and $|s|\in(c,C)$, then the large $|\zeta|$ asymptotics of $\Psi(\zeta,s)$ are described by the following lemma. 

\begin{lemma}\label{le:psiasy1}
As $\zeta\to\infty$, 

\begin{equation*}
\Psi(\zeta,s)=(I+\mathcal{O}(\zeta^{-1}))e^{-i\frac{s}{4}\zeta\sigma_3},
\end{equation*}

\noindent where $\mathcal{O}(\zeta^{-1})$ is uniform in $|s|\in(c,C)$.
\end{lemma}

As mentioned in \cite{ck}, results like this (and much stronger ones) are typical for Painlev\'e RH-problems -- see e.g. \cite[Appendix A]{FIKN} -- and well known for experts. Nevertheless, as we do not know of a reference for a proof directly applicable to the case at hand, we present one in Appendix \ref{app:unif} for the convenience of readers less familiar with such issues.

\subsection{\texorpdfstring{Large $|s|$ asymptotics of $\Psi(\zeta,s)$}{Asymptotics 2}}\label{sec:appu}

Here we consider the asymptotics of $\Psi(\zeta,s)$ as $|s|\to \infty$. The analysis of this has been performed in \cite[Section 5]{ck}.

\begin{figure}
\begin{center}
    \setlength{\unitlength}{0.8truemm}
    \begin{picture}(110,95)(-10,-7.5)
    \put(50,60){\thicklines\circle*{.8}}
    \put(50,30){\thicklines\circle*{.8}}

    \put(65,45){\thicklines\circle*{.8}}
    \put(35,45){\thicklines\circle*{.8}}

    \put(57,44){\small $+1$}
    \put(37,44){\small $-1$}

    \put(43,58){\small $+i$}
    \put(43,29){\small $-i$}
    \put(50,60){\thicklines\circle*{.8}}
    \put(50,30){\thicklines\circle*{.8}}
    \put(50,60){\line(1,1){25}}
    \put(50,30){\line(1,-1){25}}
    \put(50,60){\line(-1,1){25}}
    \put(50,30){\line(-1,-1){25}}

\put(50,30){\line(1,1){15}}\put(50,30){\line(-1,1){15}}
\put(50,60){\line(1,-1){15}}\put(50,60){\line(-1,-1){15}}

\put(60,40){\thicklines\vector(1,1){.0001}}
    \put(57,53){\thicklines\vector(-1,1){.0001}}
\put(40,40){\thicklines\vector(-1,1){.0001}}
    \put(43,53){\thicklines\vector(1,1){.0001}}

\put(23,79){$\Gamma_2$} \put(74,79){$\Gamma_1$}
\put(23,10){$\Gamma_3$} \put(74,10){$\Gamma_4$}
\put(29,44){$\Gamma_5'$} \put(66,44){$\Gamma_5''$}

    \put(65,75){\thicklines\vector(1,1){.0001}}
    \put(65,15){\thicklines\vector(-1,1){.0001}}

    \put(35,75){\thicklines\vector(-1,1){.0001}}
    \put(35,15){\thicklines\vector(1,1){.0001}}


 \end{picture}
    \caption{The jump contour for $U$ - a slight modification of \cite[Figure 3]{ck}}
    \label{fig:Ujump}
\end{center}
\end{figure}
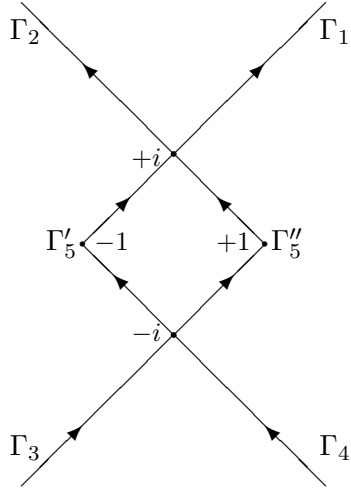

To study the $|s|\to\infty$ asymptotics of $\Psi(\zeta,s)$, one defines

\begin{equation}\label{eq:Udef}
U(\zeta,s)=\begin{cases}
\Psi(\zeta,s)e^{\frac{|s|}{4}\zeta\sigma_3 }, & \mathrm{outside \ the \ region \ delimited \  by \ } \Gamma_5' \ \mathrm{and} \ \Gamma_5''\\
\Psi(\zeta,s)\begin{pmatrix}
1 & 1\\
0 & 1
\end{pmatrix} e^{\frac{|s|}{4}\zeta\sigma_3 }, & \mathrm{in \ the \ right \ part \ of \ the \ region}\\
\Psi(\zeta,s)\begin{pmatrix}
1 & 0\\
1 & 1
\end{pmatrix} e^{\frac{|s|}{4}\zeta\sigma_3 }, & \mathrm{in \ the \ left \ part \ of \ the \ region}
\end{cases}.
\end{equation}

\noindent For the definition of the contours $\Gamma_5'$ and $\Gamma_5''$, see Figure \ref{fig:Ujump}. Then let $\mathcal{U}_\pm$ be small but fixed neighborhoods of $\pm i$. We write $\Gamma_U=\partial \mathcal{U}_+\cup \partial \mathcal{U}_{-} \cup [(\cup_{j=1}^4 \Gamma_j \cup\Gamma_5'\cup \Gamma_5'')\setminus (\overline{\mathcal{U}_+\cup \mathcal{U}_-})]$. Then the asymptotics we will need are the following. For a proof, see the discussion leading up to \cite[(5.25)]{ck}.

\begin{lemma}\label{le:Uasysbig}
As $s\to -i\infty$, for $\zeta\in \C\setminus (\overline{\mathcal{U}_+\cup \mathcal{U}_-}\cup \Gamma_U)$

\begin{equation*}
U(\zeta,s)=I+\mathcal{O}(|s|^{-1}(1+|\zeta|)^{-1}),
\end{equation*}

\noindent where the error is uniform in $\zeta\in \C\setminus (\overline{\mathcal{U}_+\cup \mathcal{U}_-}\cup \Gamma_U)$.
\end{lemma}

\subsection{Definition of \texorpdfstring{$M$}{M}}\label{app:Am}

Before going into the small $|s|$ asymptotics of $\Psi(\zeta,s)$, we need to describe a solution to an auxiliary RHP discussed in \cite[Section 4]{ck} where the discussion relies on \cite[Section 4.2.1]{cik}. The connection between the notation in \cite{ck} and ours is that $\alpha_{CK}=\beta_{ours}$ and $\beta_{CK}=0$. In \cite[Section 4.2.1]{cik}, the authors constructed an explicit function that satisfies the following RHP

{\it \begin{itemize}[leftmargin=0.5cm]
\item[1.] $M:\mathbb C\setminus e^{\pm\frac{\pi i}{4}}\mathbb R \to \mathbb C^{2\times 2}$ is analytic.
\item[2.] $M$ has continuous boundary values on $e^{\pm\frac{\pi i}{4}}\mathbb R\setminus \lbrace 0\rbrace$ and these satisfy the following jump conditions:

\begin{align}
\label{eq:Mjump1}M_+(\lambda)&=M_-(\lambda)\begin{pmatrix}
1 & e^{i\pi \beta}\\
0 & 1
\end{pmatrix}, \quad  \lambda\in e^{i\pi/4}\R_+\\
\label{eq:Mjump2}M_+(\lambda)&=M_-(\lambda)\begin{pmatrix}
1 & 0\\
-e^{-i\pi \beta} & 1
\end{pmatrix}, \quad  \lambda\in e^{3\pi i /4}\R_+\\
\label{eq:Mjump3}M_+(\lambda)&=M_-(\lambda)\begin{pmatrix}
1 & 0\\
e^{i\pi \beta} & 1
\end{pmatrix}, \quad  \lambda\in e^{5\pi i /4}\R_+\\
\label{eq:Mjump4}M_+(\lambda)&=M_-(\lambda)\begin{pmatrix}
1 & -e^{-i\pi \beta} \\
0& 1
\end{pmatrix}, \quad  \lambda\in e^{7\pi i /4}\R_+,
\end{align}

\noindent where all of the rays are oriented away from the origin $($see Figure $\ref{fig:M})$ and the $+$ side is on the left side of the ray.

\item[3.] In all sectors, as $\lambda\to \infty$

\begin{equation}\label{eq:Masy}
M(\lambda)=(I+M_1\lambda^{-1}+\mathcal{O}(\lambda^{-2}))e^{-\frac{1}{2}\lambda\sigma_3},
\end{equation}

\noindent where 

\begin{equation}\label{eq:M1}
M_1=\begin{pmatrix}
\beta^2 & -\beta \\
\beta & -\beta^2
\end{pmatrix}
\end{equation}
\end{itemize}}

We do not need the explicit form of this function -- simply that it exists and its asymptotic expansion is given by \eqref{eq:Masy}.

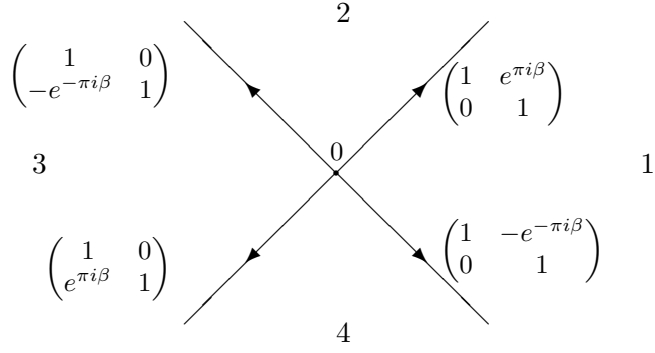
\begin{figure}
\begin{center}
    \setlength{\unitlength}{0.8truemm}
    \begin{picture}(100,48.5)(0,2.5)

    \put(49,27){\small $0$}
    \put(50,25){\thicklines\circle*{.8}}

    \put(50,25){\line(1,1){25}}
    \put(50,25){\line(1,-1){25}}
    \put(50,25){\line(-1,1){25}}
    \put(50,25){\line(-1,-1){25}}
    \put(65,40){\thicklines\vector(1,1){.0001}}
    \put(65,10){\thicklines\vector(1,-1){.0001}}
    \put(35,40){\thicklines\vector(-1,1){.0001}}
    \put(35,10){\thicklines\vector(-1,-1){.0001}}

    \put(67,37){\small $\begin{pmatrix}1&e^{\pi i\beta}\\0&1\end{pmatrix}$}
    \put(-4,40){\small $\begin{pmatrix}1&0\\-e^{-\pi i\beta}&1\end{pmatrix}$}
    \put(2,8){\small $\begin{pmatrix}1&0\\e^{\pi i\beta}&1\end{pmatrix}$}
    \put(67,11){\small $\begin{pmatrix}1&-e^{-\pi i\beta}\\0&1\end{pmatrix}$}

    \put(100,25){1}
    \put(50,50){2}
    \put(0,25){3}
    \put(50,-3){4}
    \end{picture}
    \caption{The jump contour and jump matrices for $M$.}
    \label{fig:M}
\end{center}
\end{figure}

\subsection{\texorpdfstring{Small $|s|$ asymptotics of $\Psi(\zeta,s)$}{Asymptotics 3}}

In this section, we discuss the small $|s|$ asymptotics of $\Psi(\zeta,s)$ as described in \cite[Section 6]{ck}. One begins by defining 

\begin{equation}\label{eq:psihat}
\widehat{\Psi}(\lambda,s)=e^{-\frac{s}{4}\sigma_3}\Psi\left(\frac{2}{|s|}\lambda+i\right).
\end{equation}

We also fix $\mathcal{U}_0$ -- a small neighborhood of origin, though we assume its size fixed and large enough to contain a closed disk of radius $|s|$ around the origin. In \cite[Section 6]{ck} (see in particular around \cite[(6.29)]{ck}) the following result is proven.

\begin{lemma}\label{le:psissmall}
As $s\to 0$

\begin{equation*}
\widehat{\Psi}(\lambda,s)M(\lambda)^{-1}=I+\mathcal{O}(\lambda^{-1}|s\log s|)
\end{equation*}

\noindent uniformly for $\lambda\in \C\setminus \overline{\mathcal{U}_0}$. 
\end{lemma} 

We point out that this result does not appear in \cite{ck} in precisely this form. Perhaps the closest statement is \cite[(6.29) and (6.30)]{ck} which holds for $2\beta\neq 0,1,2,...$. Then using \cite[(6.14)]{ck} and the fact that $L(\lambda)$ and $L(\lambda)^{-1}$ are entire (a standard Liouville's theorem argument implies that $\det L=1$ so the entirety of $L(\lambda)^{-1}$ follows from that of $L$) and independent of $s$, one can easily check that the above asymptotics hold in the case $2\beta \neq 0,1,2...$. For integer values of $2\beta$, the argument is similar but makes use of the discussion at the end of \cite[Section 6.3]{ck}. In particular, the $\log s$ appears only in this case.

\subsection{The function \texorpdfstring{$\Phi$ and its asymptotics}{Phi and its asymptotics}}\label{sec:phi}

The actual function that appears in the construction of our local parametrix is the following one:

\begin{equation}\label{eq:Phidef}
\Phi(\zeta,s)=\Psi(\zeta,s)\times \begin{cases}
I, & \mathrm{Im}(\zeta)\in(-1,1)\\
e^{i\frac{\pi}{2}\beta_1\sigma_3}, & \mathrm{Im}(\zeta)>1\\
e^{-i\frac{\pi}{2}\beta_2\sigma_3}, & \mathrm{Im}(\zeta)<-1
\end{cases}=:\Psi(\zeta,s)\widehat{P}^{(\infty)}(\zeta).
\end{equation}

The fact that $\Psi$ is the unique solution to a RHP translates into $\Phi$ being the unique solution to a slightly different RHP. We state this as a lemma.

\begin{lemma}\label{le:Phirhp}
$\Phi$ is the unique solution to the following RHP$:$

\vspace{0.3cm}

Find a function $\Phi=\Phi(\zeta,s)$ such that

\begin{itemize}[leftmargin=0.5cm]
\item[1.] $\Phi(\cdot,s):\C\setminus \Gamma'\to \C^{2\times 2}$ is analytic. Here

\begin{equation}\label{eq:gammaprime}
\begin{array}{lll}
\Gamma'=\cup_{k=1}^{9}\Gamma_k, & \Gamma_1=i+e^{i\frac{\pi}{4}}\R_+, & \Gamma_2=i +e^{i\frac{3\pi}{4}}\R_+,\\
\Gamma_3=-i+e^{i\frac{5\pi}{4}}\R_+, & \Gamma_4=-i+e^{i\frac{7\pi}{4}}\R_+, & \Gamma_5=[-i,i],\\
 \Gamma_6=i+\R_+, & \Gamma_7=i-\R_+, & \Gamma_8=-i-\R_+,\\
  \Gamma_9=-i+\R_+. & &\\
\end{array}
\end{equation}

\item[2.] The jump conditions are $($the orientation of the contours is according to the arrows in Figure $\ref{fig:phijumps}):$

\begin{equation*}
\Phi_+(\zeta,s)=\Phi_-(\zeta,s)V_k, \quad \zeta\in \Gamma_k
\end{equation*}

\noindent where

\begin{equation}\label{eq:phijumpm}
\begin{array}{lll}
V_1=\begin{pmatrix}
1 & 1\\
0 & 1
\end{pmatrix}, & V_2=\begin{pmatrix}
1 & 0\\
-1 & 1
\end{pmatrix}, & V_3=V_2\\
V_4=V_1, & V_5=\begin{pmatrix}
0 & 1\\
-1 & 1
\end{pmatrix}, &  V_6=e^{\pi i \frac{\beta}{2}\sigma_3}
\\
V_7=V_6, & V_8=V_6, & V_9=V_6.
\end{array}
\end{equation}

\item[3.] As $\zeta\to \infty$, 

\begin{equation}\label{eq:phiasy}
\Phi(\zeta,s)=\left(I+\Psi_1(s)\zeta^{-1}+\Psi_2(s)\zeta^{-2}+\mathcal{O}(|\zeta|^{-3})\right)\widehat{P}^{(\infty)}(\zeta)e^{-\frac{|s|}{4}\zeta\sigma_3}.
\end{equation}

\item[4.] The behavior of $\Phi$ near $\pm i$ is determined by \eqref{eq:Phidef} -- the definition of $\Phi$ and the behavior of $\Psi$ near $\pm i$ $($i.e. \eqref{eq:psiati1} -- \eqref{eq:psiat-i2}$)$.
\end{itemize}
\end{lemma}

\begin{figure}
\begin{center}
    \setlength{\unitlength}{0.8truemm}
    \begin{picture}(100,85)(0,2.5)

    \put(50,60){\thicklines\circle*{.8}}
    \put(50,30){\thicklines\circle*{.8}}
    \put(51,56){\small $i$}
    \put(51,32){\small $-i$}
    \put(50,60){\thicklines\circle*{.8}}
    \put(50,30){\thicklines\circle*{.8}}
    \put(35,60){\thicklines\vector(1,0){.0001}}
    \put(35,30){\thicklines\vector(1,0){.0001}}
    \put(69,60){\thicklines\vector(1,0){.0001}}
    \put(69,30){\thicklines\vector(1,0){.0001}}
    \put(50,60){\line(1,1){25}}
    \put(50,30){\line(1,-1){25}}
    \put(50,60){\line(-1,1){25}}
    \put(50,30){\line(-1,-1){25}}
    \put(50,30){\line(1,0){35}}
    \put(50,60){\line(1,0){35}}
    \put(50,30){\line(-1,0){35}}
    \put(50,30){\line(0,1){30}}
    \put(50,60){\line(-1,0){35}}
    \put(65,75){\thicklines\vector(1,1){.0001}}
    \put(65,15){\thicklines\vector(-1,1){.0001}}
    \put(50,47){\thicklines\vector(0,1){.0001}}
    \put(35,75){\thicklines\vector(-1,1){.0001}}
    \put(35,15){\thicklines\vector(1,1){.0001}}
    \put(74,79){\small $\begin{pmatrix}1&1\\0&1\end{pmatrix}$}
    \put(8,79){\small $\begin{pmatrix}1&0\\-1&1\end{pmatrix}$}
    \put(86,29){\small $e^{\pi i\frac{\beta_2}{2}\sigma_3}$}
    \put(86,59){\small $e^{\pi i\frac{\beta_1}{2}\sigma_3}$}
    \put(-6,29){\small $e^{\pi i\frac{\beta_2}{2}\sigma_3}$}
    \put(31,43){\small $\begin{pmatrix}0&1\\-1&1\end{pmatrix}$}
    \put(-5,59){\small $e^{\pi i\frac{\beta_1}{2}\sigma_3}$}
    \put(8,9){\small $\begin{pmatrix}1&0\\-1&1\end{pmatrix}$}
    \put(74,9){\small $\begin{pmatrix}1&1\\0&1\end{pmatrix}$}
    \end{picture}
    \caption{The jump contour $\Gamma'$ and the jump matrices for $\Phi$.}
    \label{fig:phijumps}
\end{center}
\end{figure}
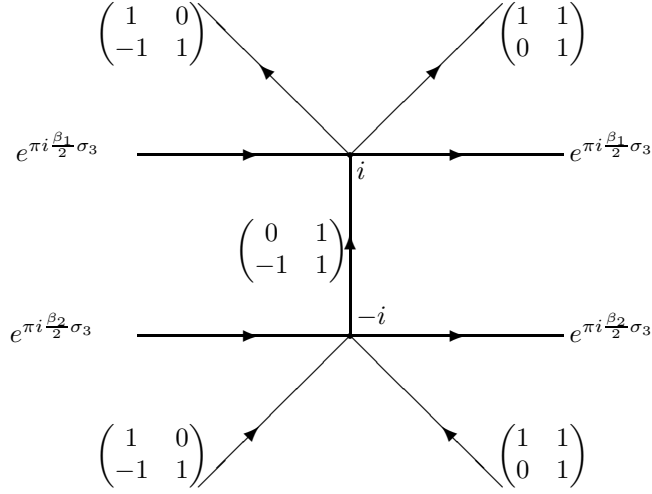

Combining the estimates concerning the asymptotic behavior for $\Phi$, i.e. Lemma \ref{le:psiasy1}, Lemma \ref{le:Uasysbig} (and \eqref{eq:Masy}), and 
Lemma \ref{le:psissmall}, we get the following estimates for the asymptotics of $\Phi$.

\begin{lemma}\label{le:phiasyall}
The function $\Psi(\zeta,s)$ has the following asymptotic behavior. 

\begin{itemize}[leftmargin=0.5cm]
\item[1.] For fixed $c,C\in(0,\infty)$

\begin{equation}\label{eq:phiasy1}
\Phi(\zeta,s)=(I+\mathcal{O}(\zeta^{-1}))e^{-i\frac{s}{4}\zeta\sigma_3}\widehat{P}^{(\infty)}(\zeta)
\end{equation}

\noindent uniformly for $|s|\in(c,C)$.

\item[2.] As $s\to -i\infty$ for $\zeta\in\C\setminus (\overline{\mathcal{U}_+\cup \mathcal{U}_-}\cup \Gamma_U)$

\begin{align}\label{eq:phiasy2}
\Phi(\zeta,s)&=(I+\mathcal{O}(|s|^{-1}(1+|\zeta|)^{-1})e^{-\frac{|s|}{4}\zeta\sigma_3}\\
\notag & \quad \times \begin{cases}
\widehat{P}^{(\infty)}(\zeta), & \mathrm{outside \ the \ region \ delimited \  by \ } \Gamma_5' \ \mathrm{and} \ \Gamma_5''\\
\begin{pmatrix}
1 & -1\\
0 & 1
\end{pmatrix}\widehat{P}^{(\infty)}(\zeta), & \mathrm{in \ the \ right \ part \ of \ the \ region}\\
\begin{pmatrix}
1 & 0\\
-1 & 1
\end{pmatrix}\widehat{P}^{(\infty)}(\zeta), & \mathrm{in \ the \ left \ part \ of \ the \ region}
\end{cases}
\end{align}

\noindent uniformly in $\zeta$.

\item[3.] As $s\to 0$ along the negative imaginary axis, 

\begin{equation}\label{eq:phiasy3}
\Phi\left(\frac{2}{|s|}\lambda+i\right)=e^{\frac{s}{4}\sigma_3}(I+\mathcal{O}(\lambda^{-1}))e^{-\frac{\lambda}{2}\sigma_3}\widehat{P}^{(\infty)}\left(\frac{2}{|s|}\lambda+i\right)
\end{equation}

\noindent uniformly for $\lambda\in \C\setminus \overline{\mathcal{U}_0}$.

\end{itemize}

\end{lemma}

\section{On the uniformity of the asymptotics -- proof of Lemma \ref{le:psiasy1}}\label{app:unif}

In this section we sketch the proof of the uniformity of the asymptotics of $\Psi(\zeta,s)$ for bounded $s$. We don't rely on general theory of Painlev\'e RH-problems, but make use of some explicit results from \cite{ck} as well as an argument mimicking the discussion in \cite[Section 7.5]{deift}. We consider the function $U$ from \eqref{eq:Udef}, though now we won't make any assumption of $|s|$ being large. Note that in terms of $U$, we can state Lemma \ref{le:psiasy1} e.g. as $\sup_{|\zeta|\geq 5,c\leq |s|\leq C}|\zeta(U(\zeta,s)-I)|<\infty$, where $|\cdot|$ denotes e.g. the Hilbert-Schmidt norm on matrices. We also recall the definition of $M=M(\lambda,\beta)$ from Appendix \ref{app:Am}.

We then define the parametrices 

$$
P^{(i)}(\zeta)=e^{-\frac{i|s|}{4}\sigma_3}e^{\frac{\pi i \beta_1}{4}\sigma_3}M\left(\frac{|s|}{2}(\zeta-i),\frac{\beta_1}{2}\right)e^{-\frac{\pi i \beta_1}{4}\sigma_3}e^{\frac{|s|}{4}\zeta\sigma_3}
$$

\noindent and 

$$
P^{(-i)}(\zeta)=e^{\frac{i|s|}{4}\sigma_3}e^{-\frac{\pi i \beta_2}{4}\sigma_3}M\left(\frac{|s|}{2}(\zeta+i),\frac{\beta_2}{2}\right)e^{\frac{\pi i \beta_2}{4}\sigma_3}e^{\frac{|s|}{4}\zeta\sigma_3}.
$$

\noindent Recall also the notation $\mathcal{U}_\pm$ and $\Gamma_U$ from Appendix \ref{sec:appu} and define:

$$
Q(\zeta)=\begin{cases}
U(\zeta), &\zeta\notin \overline{\mathcal{U}_+\cup \mathcal{U}_-}\\
U(\zeta)P^{(i)}(\zeta)^{-1}, & \zeta\in \mathcal{U}_+\\
U(\zeta)P^{(-i)}(\zeta)^{-1}, & \zeta\in \mathcal{U}_-
\end{cases}.
$$

\noindent One can then check from the RHP for $M$ that $Q$ has jumps only on $\Gamma_U$ and its boundary values are piecewise continuous and in $L^2(\Gamma_U)$ (this is in fact the point of this transformation -- it regularizes $U$ at $\pm i$ so that we can use more standard $L^2$-theory). Moreover, using the explicit form of $M$ from \cite[Section 4.2.1]{cik}, one can check that $Q$ has no isolated singularities in $\mathcal{U}_\pm$ so it is analytic in $\C\setminus \Gamma_U$ -- we omit the details. Also, we note that $Q(\zeta)=I+\mathcal{O}(\zeta^{-1})$ as $\zeta\to\infty$. Again, Lemma \ref{le:psiasy1} is equivalent to the uniformity of these asymptotics. Or more precisely, what we want to prove is that e.g. 

$$
\sup_{c\leq |s|\leq C}\sup_{|\zeta|\geq 5}\left|\zeta(Q(\zeta)-I)\right|<\infty,
$$

\noindent where again $|\cdot|$ denotes say the Hilbert-Schmidt (or Frobenius) norm.  To do this, let us look at the jumps of $Q$ in more detail. On $\partial \mathcal{U}_\pm\setminus (\cup_{i=1}^4{\Gamma}_i\cup\Gamma_5'\cup\Gamma_5'')$ (oriented in the counter-clockwise manner), we have 

$$
Q_+(\zeta)=Q_-(\zeta)P^{(\pm i)}(\zeta)^{-1}
$$

\noindent and using the jump conditions of $\Psi$, one can check that on $\Gamma_U\setminus \overline{\mathcal{U}_+\cup\mathcal{U}_-}$ we have

$$
Q_+(\zeta)=Q_-(\zeta)J_Q(\zeta)
$$

\noindent with 

\begin{align*}
J_Q(\zeta,s)&=e^{-\frac{|s|}{4}\zeta\sigma_3}J_k e^{\frac{|s|}{4}\zeta\sigma_3},\qquad \zeta\in \Gamma_k\setminus \overline{\mathcal{U}_+\cup\mathcal{U}_-},\ k=1,...,4\\
&=\begin{cases}
\begin{pmatrix}
1 & e^{\pm \pi i\beta_{1/2}}e^{-\frac{|s|}{2}\zeta}\\
0 & 1
\end{pmatrix}, & \zeta\in \Gamma_1/\Gamma_4\setminus \overline{\mathcal{U}_+\cup\mathcal{U}_-} \\
\begin{pmatrix}
1 & 0\\
 -e^{\pm \pi i\beta_{1/2}}e^{\frac{|s|}{2}\zeta} & 1
\end{pmatrix}, & \zeta\in \Gamma_2/\Gamma_3\setminus \overline{\mathcal{U}_+\cup\mathcal{U}_-} 
\end{cases},
\end{align*}

\noindent where the appropriate $\beta_i$ and sign $\pm$ are chosen depending on the part of the contour we are on, and 

\begin{align*}
J_Q(\zeta,s)&=\begin{cases}
e^{-\frac{|s|}{4}\zeta\sigma_3}\begin{pmatrix}
1 & 0\\
-1 & 1
\end{pmatrix}e^{\frac{|s|}{4}\zeta\sigma_3}, & \zeta\in \Gamma_5'\setminus \overline{\mathcal{U}_+\cup\mathcal{U}_-}\\
e^{-\frac{|s|}{4}\zeta\sigma_3}\begin{pmatrix}
1 & 1\\
0 & 1
\end{pmatrix}e^{\frac{|s|}{4}\zeta\sigma_3}, & \zeta\in \Gamma_5''\setminus \overline{\mathcal{U}_+\cup\mathcal{U}_-}
\end{cases}\\
&=\begin{cases}
\begin{pmatrix}
1 & 0\\
-e^{\frac{|s|}{2}\zeta} & 1
\end{pmatrix}, & \zeta\in \Gamma_5'\setminus \overline{\mathcal{U}_+\cup\mathcal{U}_-}\\
\begin{pmatrix}
1 & e^{-\frac{|s|}{2}\zeta}\\
0 & 1
\end{pmatrix}, & \zeta\in \Gamma_5''\setminus \overline{\mathcal{U}_+\cup\mathcal{U}_-}
\end{cases}.
\end{align*}

We then note that as $Q$ is normalized at infinity, for $\zeta\notin\Gamma_U$, we can write

\begin{align}\label{eq:plemelj}
Q(\zeta,s)=I+\int_{\Gamma_U} Q_-(z,s)\frac{J_Q(z,s)-I}{z-\zeta}\frac{dz}{2\pi i}.
\end{align}

\noindent This follows simply by Sokhotski-Plemelj and Liouville's theorem: one can check from Sokhotski-Plemelj, that the difference between the left hand side and the right hand side has no jumps, then after checking that points where the curve intersects itself can't be isolated singularities (implying that the difference between the RHS and LHS is an entire function), one sees that as both the LHS and RHS tend to $I$ as $\zeta\to \infty$, that difference must vanish identically by Liouville's theorem.

Now note that \eqref{eq:plemelj} can be written as

\begin{align*}
\zeta\left[Q(\zeta,s)-I\right]=\int_{\Gamma_U}\left[Q_-(z,s)-I\right]\left[J_Q(z,s)-I\right]\frac{\zeta}{z-\zeta}\frac{dz}{2\pi i}+\int_{\Gamma_U}\left[J_Q(z,s)-I\right]\frac{\zeta}{z-\zeta}\frac{dz}{2\pi i}.
\end{align*}

\noindent For simplicity, let us also assume that $\zeta$ is in such a direction that $\frac{|\zeta|}{|\zeta-z|}\leq A$ which is independent of $\zeta$ and $z\in \Gamma_U$ (if this were not true, then as the jump matrix is piecewise analytic, one could perform a contour deformation rotating the unbounded parts of the jump contours slightly, which would then reduce to this case). We then find 

\begin{align*}
\left|\zeta\left[Q(\zeta,s)-I\right]\right|&\leq A \left(\int_{\Gamma_U}\left|Q_-(z,s)-I\right|^2 |dz|\right)^{1/2}\left(\int_{\Gamma_U} \left|J_Q(z,s)-I\right|^2 |dz|\right)^{1/2}\\
&\quad +A\int_{\Gamma_U} \left|J_Q(z,s)-I\right| |dz|.
\end{align*}

\noindent As $|J_Q(z,s)-I|=\mathcal{O}(e^{-\frac{|s|}{2}|\mathrm{Re}(z)|})$, we see that the $J_Q$-integrals can be bounded by a finite constant depending only on $c,C$ (which determined the region where $s$ is). We will thus be done if we can show that 

\begin{align*}
s\mapsto \left(\int_{\Gamma_U}\left|Q_-(z,s)-I\right|^2|dz|\right)^{1/2} 
\end{align*}

\noindent is a continuous function. Note that as $Q_-(z,s)-I=\mathcal{O}(z^{-1})$ as $z\to \infty$, and is a piecewise continuous function, we see that at least $Q_-(z,s)-I$ is in $L^2(\Gamma_U)$. We will then be done if we can show that if $t\to s$, then $Q_-(\cdot,t)\to Q_-(\cdot,s)$ in $L^2(\Gamma_U)$. To prove this, let us rewrite the singular integral equation satisfied by $Q_-$ in a slightly different way. Let $C_J$ denote the operator $C_-^{\Gamma_U}(\cdot (J_Q-I))$, where $C_-^{\Gamma_U}$ denotes the boundary value of the Cauchy integral on $\Gamma_U$ taken from the right. Taking boundary values of \eqref{eq:plemelj}, we see that $Q_-$ satisfies on $\Gamma_U$ the singular  integral equation

\begin{align*}
Q_-=I+C_-(Q_-(J_Q-I))=I+C_J(Q_-)
\end{align*}

\noindent or 

\begin{align*}
(I-C_J)(Q_-)=I
\end{align*}

\noindent or put yet another way:

\begin{align*}
(I-C_J)(Q_--I)=C_J(I).
\end{align*}

\noindent Let us try to use this equation to invert $I-C_J$ on $L^2(\Gamma_U)$. Consider the equation $(I-C_J)f=g$ for $f\in L^2(\Gamma)$. Define

\begin{align*}
\mathcal{M}(\zeta,s)=\frac{1}{2\pi i}\int_{\Gamma_U} f(z)\frac{J_Q(z,s)-I}{z-\zeta}dz.
\end{align*}

\noindent By Sokhotski-Plemelj, we have for $\zeta\in \Gamma$

\begin{align*}
\mathcal{M}_+(\zeta,s)=[C_+ f(J_Q-I)](\zeta,s)=[C_J f](\zeta,s)+f(\zeta)(J_Q(\zeta,s)-I)=f(\zeta,s)J_Q(\zeta,s)-g(\zeta)
\end{align*}

\noindent or in other words, 

\begin{align}\label{eq:fzs}
f(\zeta,s)=\left[g(\zeta)+\mathcal{M}_+(\zeta,s)\right]J_Q(\zeta,s)^{-1}.
\end{align}

\noindent On the other hand, we see that 

\begin{align*}
\mathcal{M}_-(\zeta,s)=[C_J f](\zeta,s)=-g(\zeta)+f(\zeta,s)
\end{align*}

\noindent so we see that 

\begin{align*}
\mathcal{M}_+(\zeta,s)=(\mathcal{M}_-(\zeta,s)+g(\zeta))J_Q(\zeta,s)-g(\zeta)=\mathcal{M}_-(\zeta,s)J_Q(\zeta,s)+g(\zeta)(J_Q(\zeta,s)-I).
\end{align*}

Consider now $\mathcal{M}Q^{-1}$. This satisfies

\begin{align*}
\mathcal{M}_+(\zeta,s)Q_+(\zeta,s)^{-1}&=\left[\mathcal{M}_-(\zeta,s) J_Q(\zeta,s)+g(\zeta)(J_Q(\zeta,s)-I)\right]J_Q(\zeta,s)^{-1} Q_-(\zeta,s)^{-1}\\
&=\mathcal{M}_-(\zeta,s)Q_-(\zeta,s)^{-1}+g(\zeta)(J_Q(\zeta,s)-I)Q_+(\zeta,s)^{-1}.
\end{align*}

\noindent Thus by Sokhotski-Plemelj,

\begin{align*}
\mathcal{M}(\zeta,s)Q(\zeta,s)^{-1}=\int_{\Gamma_U}g(z)\frac{J_Q(z,s)-I}{z-\zeta}Q_+(z,s)^{-1}\frac{dz}{2\pi i}
\end{align*}

\noindent which implies by \eqref{eq:fzs}

\begin{align*}
f&=C_+(g (J_Q-I) Q_+^{-1}) Q_+ J_Q^{-1}+g J_Q^{-1}\\
&=C_+(g(I-J_Q^{-1})Q_-^{-1})Q_-+g J_Q^{-1}\\
&=C_-(g(I-J_Q^{-1})Q_-^{-1})Q_-+g(I-J_Q^{-1}))+g J_Q^{-1}\\
&=C_-(g(I-J_Q^{-1})Q_-^{-1})Q_-+ g.
\end{align*}

Now since $Q_-^{\pm 1}$ and $J_Q^{\pm 1}$ are bounded functions on $J_U$, one sees by the $L^2$-boundedness of $C_-$ that this is a bounded operator on $L^2(\Gamma_U)$ and we conclude that indeed $(I-C_J)$ is invertible on this space -- the inverse being $\mathrm{Id}+C_-(\cdot (I-J_Q^{-1})Q_-^{-1})Q_-$.

As $J_Q(\zeta,t)\to J_Q(\zeta,s)$ as $t\to s$ uniformly in $\zeta$ and one can then check easily that $C_{J(\cdot ,t)}\to C_{J(\cdot,s)}$ in the operator norm on $L^2(\Gamma_U)$, thus we also have $(I-C_{J(\cdot,t)})^{-1}\to (I-C_{J(\dot,s)})^{-1}$ in the operator norm $L^2(\Gamma_U)$. We then find 

\begin{align*}
\|Q_-(\cdot,t)-Q_-(\cdot,s)\|_{L^2(\Gamma_U)}&=\|(I-C_{J(\cdot,t)})^{-1}(C_{J(\cdot,t)}(I))-(I-C_{J(\cdot,s)})^{-1}(C_{J(\cdot,s)}(I))\|_{L^2(\Gamma_U)}\\
&\leq \|(I-C_{J(\cdot,t)})^{-1}-(I-C_{J(\cdot,s)})^{-1}\|_{L^2(\Gamma_U)\to L^2(\Gamma_U)}\|C_{J(\cdot,t)}(I)\|_{L^2(\Gamma_U)}\\
&\quad +\|(I-C_{J(\cdot,s)})^{-1}\|_{L^2(\Gamma_U)\to L^2(\Gamma_U)}\|C_{J(\cdot,t)}-C_{J(\dot,s)}\|_{L^2(\Gamma_U)\to L^2(\Gamma_U)}
\end{align*}

\noindent which tends to zero as $t\to s$ by our above discussion ($\|C_{J(\cdot,t)}(I)\|_{L^2(\Gamma_U)}$ is bounded in $t$ since we already saw it to be continuous). We are thus done.

\end{document}